\documentclass[12pt,a4paper]{amsart}
\usepackage{amsfonts}
\usepackage{amsthm}
\usepackage{amsmath}
\usepackage{amssymb}
\usepackage[foot]{amsaddr}
\usepackage{amscd}
\usepackage{t1enc}
\usepackage[mathscr]{eucal}
\usepackage{indentfirst}
\usepackage{graphicx}
\usepackage{graphics}
\usepackage{pict2e}
\usepackage{epic}
\numberwithin{equation}{section}
\usepackage[margin=2.9cm]{geometry}
\usepackage{epstopdf} 
\usepackage{hyperref}
\usepackage{color}
\usepackage{todonotes}

\allowdisplaybreaks

\theoremstyle{plain}
\newtheorem{Thm}{Theorem}[section]
\newtheorem{Lem}[Thm]{Lemma}
\newtheorem{Cor}[Thm]{Corollary}
\newtheorem{Prop}[Thm]{Proposition}

\theoremstyle{definition}
\newtheorem{Def}[Thm]{Definition}

\newtheorem{Rem}[Thm]{Remark}
\newtheorem{?}[Thm]{Problem}

\newcommand{\dv}{\mathrm{div}}
\newcommand{\bpartial}{\bar{\partial}}

\newcommand{\vv}{\mathbf{v}}

\begin{document}

\title[Local theory of Free boundary problem of the FNS in 3D]{Local existence and uniqueness of strong solutions to the free boundary problem of the full compressible Navier-Stokes equations in 3D}

\author[X. Liu]{Xin LIU}
\address[X. Liu]{Department of Mathematics, Texas A\&M University, College Station, TX, United States}
\email{xliu@math.tamu.edu}

\author[Y. Yuan]{Yuan YUAN}
\address[Y. Yuan]{The Institute of Mathematical Sciences \& Department of Mathematics, The Chinese University of Hong Kong,	Shatin, N.T., Hong Kong}
\email{yyuan@math.cuhk.edu.hk}

\date{\today}

\begin{abstract}
In this paper we establish the local-in-time existence and uniqueness of strong solutions to the free boundary problem of the full compressible Navier-Stokes equations in three-dimensional space.
The vanishing density and temperature condition is imposed on the free boundary, which captures the motions of the non-isentropic viscous gas surrounded by vacuum with bounded entropy.
We also assume some proper decay rates of the density towards the boundary and singularities of derivatives of the temperature across the boundary on the initial data, which coincides with the physical vacuum condition for the isentropic flows.
This extends the previous result of Liu [ArXiv:1612.07936] by removing the spherically symmetric assumption and considering more general initial density and temperature profiles.
\end{abstract}

\maketitle


\section{Introduction}

In this work, we aim to study the motions of the non-isentropic viscous gas surrounded by the vacuum,
which is modeled by the free boundary problem of the full compressible Navier-Stokes equations in $\mathbb{R}^3$:
\begin{equation}\label{FNS}
\left\{
\begin{aligned}
& \rho_t+\dv\, (\rho u)=0 && \text{in}\ \Omega(t), \\
& (\rho u)_t +\dv \, (\rho u \otimes u)+\nabla p= \dv \, \mathbb{S} && \text{in}\ \Omega(t),\\
&\partial_t (\rho E)+\dv\, (\rho u E+p u )=\dv\, (u \mathbb{S})+\kappa\Delta\theta && \text{in}\ \Omega(t)\\
&\rho>0,~\theta>0&& \text{in}\ \Omega(t),\\
& \rho=0,\ \theta=0,\ (\mathbb{S}-p\mathbb{I}_3) n=0,  && \text{on} \ \Gamma(t),\\
& \mathcal{V}(\Gamma (t))=u\cdot n && \text{on} \ \Gamma(t):=\partial \Omega(t),\\
&(\rho,u,\theta)=(\rho_0,u_0,\theta_0) && \text{on} \ \Omega:=\Omega(0),
\end{aligned}
\right.
\end{equation}
where $\rho$, $u$, $p$, $\mathbb{S}$  and $\theta$ represent the scaler density, the velocity field, the pressure potential, the viscous tensor and the absolute temperature respectively. In particular, $ \mathbb S - p\mathbb I_3 $ is the total stress tensor. 
$E=\frac{1}{2}|\vv|^2+e$ is the specific energy with $e $ being the specific internal energy.
The constant $\kappa>0$ is the coefficient of heat conductivity.
$\Omega(t) := \lbrace x\in\mathbb{R}^3; \rho(x,t) > 0 \rbrace $ is the evolving, occupied domain determined by the kinetic boundary condition $ \eqref{FNS}_{6} $, where $\Gamma(t)$ represents the gas-vacuum interface  and $\mathcal{V}(\Gamma(t))$ is the normal velocity of the evolving interface. 
Here $n$ is the exterior unit normal vector on $ \Gamma(t) $.

We assume that the fluid is Newtonian and the viscous stress tensor $\mathbb{S}$ is taken as 
$$
\mathbb{S}=\mu (\nabla u+(\nabla u)^\intercal)+\lambda \ \dv \, u \ \mathbb{I}_3 
$$
with the Lam\'{e} viscosity coefficients $\mu, \lambda$ being constants  which satisfy $\mu>0 $, $ 2\mu+3\lambda > 0 $. 
Also, we assume that the gas is polytropic, and the pressure potential $p$, the specific inner energy $ e $ are given by the equations of state,
\begin{equation}
\label{state1}
p=R\rho\theta = \bar{A}e^{S(\gamma-1)/R}\rho^{\gamma}, \quad e=c_\nu \theta = \dfrac{R}{\gamma-1}\theta,
\end{equation}
where $S$ is the entropy, and $\gamma>1$, $R>0$, $ c_{\nu}>0$ are referred to as the adiabatic exponent, the universal gas constant and the specific heat coefficient respectively. 
Here $ \bar A $ is any positive constant. 
In particular, $ p, e, \rho $ satisfy the Gibbs-Helmholtz equation (\cite{Feireisl2009}).

The goal of this work is to establish the local existence and uniqueness theorems for the system \eqref{FNS}. In particular, we want to show the existence of classical solutions with vanishing density on the free boundary and bounded entropy $S$ in the space-time. 
This yields the compatible boundary condition $\theta = 0 $ on $ \Gamma(t) $ by \eqref{state1}.
In fact, we show that if the entropy is initially bounded, then for a short time it remains bounded (see Remark \ref{rem_entropy}). 
This is of great interest in the following sense. For the Cauchy problem, the author in \cite{Xin1998} has shown that if the initial density is supported in a compact domain, the classical solution with bounded entropy blows up in finite time if $ \kappa = 0 $. See also \cite{Xin2013}. Therefore, it is worth considering the free boundary problem of the compressible Navier-Stokes equations.
That is the system \eqref{FNS}. 
On the other hand, the radiation gaseous star can be modeled by such hydrodynamic equations. In fact, on the boundary of a radiation gaseous star, the temperature vanishes ($ \theta|_{\Gamma(t)} = 0 $). This indicates the fact that on the surface of a star, radiation transfers the heat to the space in the form of light (through emission of photons), and thus the temperature of the gas is relatively very low. See \cite{Chandrasekhar1957} and \cite{Liu2016a}. 

The compressible Navier-Stokes equations have been studied widely among the literatures. To name a few, let us start with the classical result concerning the Cauchy and first initial boundary value problems. In the absence of vacuum ($ \rho \geq \underline \rho > 0 $), the local and global well-posedness of classical solutions have been investigated for a long time. For instance, Serrin \cite{Serrin1959} considered the uniqueness of both viscous and inviscid compressible flows. Itaya \cite{Itaya1971} and Tani \cite{Tani1977} considered the local well-posedness and uniqueness for the Cauchy and first initial boundary problems. On the other hand, Matsumura and Nishida \cite{Matsumura1980,Matsumura1983} established the first results concerning the global well-posedness of classical solutions to the Navier-Stokes equations. Such global solutions are evolving near a uniform non-vacuum state. However, when there exist vacuum areas of the solutions, the solutions to compressible Navier-Stokes equations may behavior completely differently. For example, the aforementioned works by Xin and Yan \cite{Xin1998,Xin2013} point out that the appearance of vacuum areas of the initial density profile will cause finite time blow-up of the classical solutions. See also \cite{Jiu2015,Cho2006}. Recently, Li, Wang and Xin \cite{Li2017b} show that there are no classical solutions with bounded entropy for the compressible Navier-Stokes equations in some Sobolev spaces in any time interval $ (0,T) $ if the initial data contains vacuum. On the other hand, a local well-posedness theory was developed by Cho and Kim \cite{Cho2006b,Cho2006a} for barotropic and heat-conductive flows. See also \cite{Zhang2007a}. However, the solutions obtained by Cho and Kim are not in the same functional space as those in \cite{Xin1998}. In particular, they can not track the entropy in the vacuum area. With a small initial energy, Huang, Li, Xin \cite{HuangLiXin2012} showed the global existence of classical solutions but with large oscillations to the isentropic compressible Navier-Stokes equations. With the density-dependent but non-degenerated viscosities,  Va\v{i}gant and Kazhikhov in \cite{Vaygant1995} first developed the global existence of strong solutions to the isentropic Navier-Stokes equations in two-dimensional space when the initial density is away from the vacuum. Recently, by Jiu, Wang and Xin in \cite{Jiu2014}, such a result is developed in the case when there exists vacuum area of the solutions . Latter, Huang and Li in \cite{Huang2016} refines the results to a wider range of the adiabatic index with arbitrary large initial data. Unfortunately, as pointed out by Liu, Yang, Xin in \cite{Liu1998a}, such solutions may not lead to a physically satisfactory one when there exists a vacuum area. Therefore, there are basically two strategies to resolve such problems caused by the vacuum. One is to study the Navier-Stokes system with degenerated viscosities which vanishes on the vacuum area. The other is to study the problem in a free boundary setting. In fact, the authors introduced the problem with degenerated viscosities and showed the well-posedness of the free boundary problem locally in time.

As for the free boundary problem, the local well-posedness theory can be tracked back to Solonnikov and Tani \cite{Solonnikov1992}, Zadrzy\'nska and Zaj\c{a}czkowski \cite{Zadrzynska2001,Zadrzynska1994,Zajaczkowski1993,Zajaczkowski1995}. The global well-posedness theory was studied in \cite{Zajaczkowski1999,Zajaczkowski1993,Zajaczkowski1994}. Among these works, the total stress tensor on the moving surface is balanced by a force induced by the surface tension or an external pressure. Thus there exists a uniformly non-vacuum equilibrium state. The global solutions are obtained as a perturbation of such a state. See \cite{Zhang2007} for the spherically symmetric motions when a external force is present. When the external force is given by the self-gravitation, Zhang and Fang in \cite{Zhang2009} studied the existence of global weak solutions and the asymptotic convergence to the equilibrium with small initial perturbations. The stationary equilibrium is given by the gaseous star equation in \cite{Auchmuty1971,Chandrasekhar1957,Lieb1987}. 
Recently, 
the nonlinear asymptotic stability of the stationary equilibrium is established by Luo, Xin, Zeng in \cite{Luo2016a,Luo2016b}, where the coordinate singularity at the center of the spherical coordinates are resolved. Let us emphasize that the equilibrium states among these works are important in the sense that they give a reference domain for the change of coordinates and the stability of the equilibrium gives the control of the Jacobian. 

Most of the works in this direction focus on the spherically symmetric motions. When the density connects to vacuum with a jump, a global weak solution to the problem with density-dependent viscosities and spherically symmetric motions was obtained in \cite{Guo2012a} by Guo, Li, Xin. It was shown that the solution is smooth away from the symmetric center. Recently, such a problem with different conditions on the viscosity tensor is studied in two-dimensional space by Li and Zhang \cite{Li2016}, where the authors established the existence of the global strong solution. Another noticeable result from Yeung and Yuen in \cite{Yeung2009} constructed some analytic solutions in the case of density-dependent viscosities. This was further studied in \cite{Guo2012b}. Such solutions indicate that the domain of the gas (fluid) will expand as time grows up, and the density will decrease to zero everywhere including the centre. When the density connects to vacuum continuously, Luo, Xin, Yang constructed a class of global solutions in one-dimensional space in \cite{Luo2000}. This is further generalized by Zeng in \cite{Zeng2015}. Recently by Liu \cite{Liu2018}, the existence of global solutions with small initial energy in the spherical motions is established with the density connected to vacuum continuously or discontinuously. 

Another problem in the gas-vacuum interface problem arises in the case  when the density connects to vacuum satisfying the physical vacuum condition for the isentropic flows. That is, the sound speed is only $1/2$-H\"older continuous across the interface. 
This singularity can be found in the self-similar solutions to compressible Euler equations with damping (\cite{Liu1996}), or the stationary solutions to the gaseous star equation (\cite{Chandrasekhar1957, Lin1997}). 
As it is pointed out by Liu in \cite{Liu1996}, such a singularity of the density makes the standard hyperbolic method fail in establishing the local well-posedness theory for the inviscid flow. Only recently, the local well-posedness theory for inviscid flows is developed by Coutand, Lindblad and Shkoller \cite{Coutand2010,Coutand2011,Coutand2012}, Jang and Masmoudi \cite{Jang2009,Jang2015}, Gu and Lei \cite{Gu2012,Gu2015}, Luo, Xin and Zeng \cite{Luo2014} in variant settings. 
See \cite{Jang2010, Duan2015a} for the viscous case.
The nonlinear stabilities of some special solutions in variant settings with the physical vacuum condition can be found in \cite{ Hadzic2016a,Hadzic2016,Liu2017a, Liu2018,Luo2016b, Luo2016a, Luo2016,Shkoller2017,Zeng2015,Zeng2017}. However, all the above works are concerned on the isentropic case and seldom results in the non-isentropic case have been obtained. In \cite{Li2018}, Li proved the global well-posedness of strong solutions to the full Navier-Stokes equations in 1D case. Liu established the local well-posedness theory for the heat conductive flows with self-gravitation in \cite{Liu2016a}, and the nonlinear stability of a special class of expanding solutions in \cite{Liu2018}. 

Our purpose of this work is to study the gas-vacuum interface problem for heat conductive flows in the multi-dimensional space without imposing the symmetry property. The result we achieve in this work extends the one in \cite{Liu2016a} into general initial density and temperature profiles. In particular, the density and temperature can live in the same functional space as those in \cite{Liu2016a}, which contains the equilibrium state of the radiation flows with self-gravitation. 
The entropy of solutions is bounded in short time provided the initial entropy is bounded. Inspired by the above physical vacuum condition, we also impose proper decay rates of the density towards the boundary and singularities of derivatives of the temperature across the boundary on the initial data, which coincide with the physical vacuum condition for the isentropic fluids in certain sense (see Remark \ref{physical vacuum}). 

To solve the free boundary problem of \eqref{FNS}, we follow the method of Coutand and Shkoller in \cite{Coutand2012}. By transforming \eqref{FNS} in Lagrangian coordinates, it turns out that the equation \eqref{1} are two parabolic equations of the velocity $v$ and the temperature $\Theta$ on a fixed domain. The density $\rho_0$ appears as coefficients, and the coefficients of the time derivatives degenerate on the boundary. As standard energy estimates of parabolic equations, we use dissipation terms to control the others. However, the Poincar\'{e} inequality is not applicable to estimate $v$ since it does not vanish on the boundary. Instead, here the Hardy's inequality and the assumption on the decay rate of the density profile towards the boundary are used. We use the singularities of the temperature derivatives crossing the boundary to maintain the positivity of $\Theta$ inside the region and the boundedness of the entropy.
The ellipticity of the coefficients of the dissipation terms and the boundedness of the Jacobian of the flow map are controlled by the norms of velocity and small enough time. Having obtained the a priori estimate of \eqref{1}, the existence and uniqueness of strong solutions is achieved by Banach fixed point theorem.

This work is organized as follows: in Section 2 we formulate \eqref{FNS} in the Lagrangian coordinates and state the result; in Section 3 we prepare the inequalities which will be frequently used; the a priori estimates and the fixed point argument to prove the existence are presented in Section 4 and Section 5 respectively.


\vspace{0.7cm}
\section{Results}

\subsection{Lagrangian formation}~

Denote the Lagrangian and Eulerian coordinates as $x$ and $y$ respectively.
Also let the flow map $y=\eta(x)$ be defined by the following ODE:
\begin{equation}\label{def-eta}
\left\{
\begin{aligned}
&\eta_t=u\circ \eta\\
&\eta(0)=\text{Id}~
\end{aligned}
\right. ,
\end{equation}
where Id is the identity map.
Meanwhile, the associated deformation matrix and Jacobian are denoted as 
$$
A=[D\eta]^{-1}, 
J=\det (D\eta), 
a=JA,
(A^i_j=\frac{\partial x_i}{\partial y_j}=J^{-1}a^i_j)
$$
Then the Lagrangian unknowns are defined by the following change of variables
$$
f=\rho \circ \eta,
v=u\circ \eta,
\Theta=\theta \circ \eta.
$$

We will use $F_{,k}$ to represent $\frac{\partial F}{\partial x_k}$.
Throughout this paper, $\delta^i_j$, $\delta^{ij}$ and $\delta_{ij}$ all present the Kronecker symbols;
the Einstein summation convection will be employed and we will also omit the summation symbols of dot products of vectors or double dot products of matrices for convenience.
Then in the Lagrangian coordinates, the system \eqref{FNS} is in the following form
\begin{equation}
\label{0}
\left\{
\begin{aligned}
& f_t+f \dv_{\eta} v=0 && \text{in}\ \Omega\times(0,T],\\
& fv^i_t+(\nabla_{\eta} (Rf\Theta))^i = (\dv_{\eta} \mathbb{S}_{\eta}[v])^i && \text{in}\ \Omega\times(0,T],\\
& c_v f\Theta_t +Rf\Theta \ \dv_{\eta} v= \mathbb{S}_{\eta}[v]:\nabla_{\eta}v+\kappa \Delta_{\eta}\Theta && \text{in}\ \Omega\times(0,T],\\
&f>0,~ \Theta>0 && \text{in}\ \Omega\times(0,T],\\
&f=0,\ \Theta=0,\ (\mathbb{S}_{\eta}[v]-Rf\Theta) (n\circ \eta)=0 && \text{on} \ \Gamma\times(0,T]:=\partial\Omega\times (0,T],\\
&(f,v,\Theta)=(\rho_0,u_0,\theta_0) && \text{on}\ \Omega\times\{t=0\},
\end{aligned}
\right.
\end{equation}
where
for a scalar field $F$ and a vector field $W$,
$$
(\nabla_{\eta}F)^i=A_i^k F_{,k},~
\dv_{\eta} W=A_l^k W^l_{,k}\ ,~
\Delta_{\eta}F=\dv_{\eta}\nabla_{\eta} F,
$$
and the viscosity tensor in the Lagrangian coordinates is represented by
\begin{equation}\label{S}
\mathbb{S}^{ij}_{\eta}[W]=\mu (A^{k}_j W^i_{,k}+A^k_i W^j_{,k})+\lambda (A_l^k W^l_{,k} )\delta^i_j.
\end{equation}

It follows from $\eqref{0}_1$ and \eqref{deri_Ja} that
\begin{equation*}
(fJ)_t=0.
\end{equation*}
Therefore, if initially $\rho_0>0$ in $\Omega$ and $\rho_0=0$ on $\Gamma$, then so is $f$ for a regular solution.
So the full Navier-Stokes equations in the Lagrangian coordinates can be written as
\begin{equation}\label{1}
\left\{
\begin{aligned}
& \rho_0 v^i_t+a^r_i (\frac{R\rho_0 \Theta}{J})_{,r} = a^r_j \mathbb{S}^{ij}_{\eta}[v]_{,r} &&\text{in}\ \Omega\times(0,T],\\
& c_v \rho_0\Theta_t +\frac{R\rho_0\Theta}{J}a^r_iv^i_{,r}=  \mathbb{S}^{ij}_{\eta}[v]a^r_jv^i_{,r}+\kappa a^r_i(\nabla_{\eta}\Theta)^i_{,r} &&\text{in}\ \Omega\times(0,T],\\
& \Theta>0 &&\text{in}\ \Omega\times(0,T],\\
& \Theta=0,\  (\mathbb{S}_{\eta}^{ij}[v]) (n\circ \eta) =0 &&\text{on}\ \Gamma\times(0,T],\\
& (v,\Theta)=(u_0,\theta_0) && \text{on}\ \Omega\times\{t=0 \}.
\end{aligned}
\right.
\end{equation}
(Notice that here we omit the summation symbol $\sum\limits_{i,j=1,2,3}$ of the double dot product of matrices $\mathbb{S}^{ij}_{\eta}[v]a^r_jv^i_{,r}$ for convenience.)
\vspace{0.3cm}

\subsection{The reference domain}~

For the sake of simplicity, we follow the assumption of \cite{Coutand2012, Jang2015} on the initial domain.
That is $\Omega=\mathbb{T}^2\times(0,1)$ and $\Gamma=\{x=(x_1,x_2,x_3)|~x_3=0,1\}$.
Here $\mathbb{T}^2=(0,1)^2$ is the periodic unit box in $\mathbb{R}^2$, and $x_3=0,1$ are the coordinates of the free boundaries in the Lagrangian coordinates.
Then 
\begin{equation*}
(n\circ \eta)^i=\begin{cases}
\left(\frac{\eta_{,1}\times \eta_{,2}}{|\eta_{,1}\times \eta_{,2}|}\right)^i=\frac{a^3_i}{\sqrt{(a^3_1)^2+(a^3_2)^2+(a^3_3)^2}} & x_3=1\\
-\left(\frac{\eta_{,1}\times \eta_{,2}}{|\eta_{,1}\times \eta_{,2}|}\right)^i=-\frac{a^3_i}{\sqrt{(a^3_1)^2+(a^3_2)^2+(a^3_3)^2}} & x_3=0
\end{cases},
\end{equation*}
and the boundary condition $(\mathbb{S}_{\eta}^{ij}[v]) (n\circ \eta) =0$ on $\Gamma\times(0,T]$ is equivalent to 
$$
a^3_j \mathbb{S}_{\eta}^{ij}[v]=0 ~\text{on}~\Gamma\times(0,T].
$$

We will use $D,\ \bpartial,\ \partial_3$ to represent
$$
D \in \left\{\frac{\partial}{\partial x_1}, \frac{\partial}{\partial x_2}, \frac{\partial}{\partial x_3} \right\},~
\bpartial\in\left\{\frac{\partial}{\partial x_1}, \frac{\partial}{\partial x_2}\right\},~\partial_3=\frac{\partial}{\partial x_3}.
$$ 
and use the simplified notations for tensors 
\begin{align*}
&|D\rho_0|=\sum\limits_{i=1,2,3}|\rho_{0,i}|,~
|\bpartial\rho_0|=\sum\limits_{i=1,2}|\rho_{0,i}|,~|Dv|=\sum\limits_{i,j=1,2,3}|v^i_{,j}|,~\\
&|\bpartial v|=\sum\limits_{\substack{i=1,2,3,\\ j=1,2}}|v^i_{,j}|,~
|a|=\sum\limits_{i,j=1,2,3}|a^i_{,j}|,~|\delta|=\sum\limits_{i,j=1,2,3}|\delta^i_{j}|.
\end{align*}
Moreover, we use the notation $H^k$ as the usual Sobolev space $H^k(\Omega)$ (or $H^k(\Omega; \mathbb{R}^3)$), and $H^1_0$ as
$$
H^1_0(\Omega)=\{u\in H^1;~ u=0~\text{on}~\Gamma\}.
$$

\subsection{Main theorem}~

The initial density and temperature profiles are assumed to satisfy the following conditions:
\begin{equation}\label{rho1}
	\rho_0(x)>0 \ \text{in}\ \Omega \ \text{and}\ \rho_0=0 \ \text{on}\ \Gamma,
\end{equation}
\begin{equation}
	\label{rho2}
	\rho_0=\mathcal{O} (d^{\alpha}) \qquad \text{as}\ d \rightarrow 0^+ \ \text{for some}\ \alpha>0,
\end{equation} 
	where $d(x)=x_3(1-x_3)$ represents the distance from the point $x$ to the boundary $\Gamma$. Also
 \begin{equation}
	\label{rho3}
	\|\rho_0\|_{L^{\infty}}+\|D \rho_0\|_{L^3}+\|\bpartial \rho_0\|_{H^1}
	+\|d \Big(| D^2 \rho_0 |+|\bpartial D^2 \rho_0|\Big) \|_{L^2}\leq C,
\end{equation}
\begin{equation}
\label{theta1}
\theta_0(x)>0 \ \text{in}\ \Omega \ \text{and}\ \theta_0=0 \ \text{on}\ \Gamma,
\end{equation}
\begin{equation}\label{theta2}
-\infty<\nabla_{\mathbf{n}}(\theta_0)<0, \quad \text{on}\; \Gamma.
\end{equation}

On the other hand, the energy functional in this work is defined by
\begin{equation}\label{E}
\begin{aligned}
E(v,\Theta)(t)=&\|\rho_0^{1/2}v_{tt}(\cdot,t)\|_{L^2}^2+\|v_t(\cdot,t)\|_{H^1}^2+\|v(\cdot,t)\|_{H^3}^2\\
&+\|\rho_0^{1/2}\Theta_{tt}(\cdot,t)\|_{L^2}^2+\|\Theta_t(\cdot,t)\|_{H_0^1}^2+\|\Theta(\cdot,t)\|_{H^3}^2.
\end{aligned}
\end{equation}
We also denote $M_0$ as a constant associated with the initial energy 
\begin{equation}\label{M_0}
   \begin{aligned}
    M_0=&\|\rho_0^{1/2}u_{0tt}\|_{L^2}^2+\|u_{0t}\|_{H^1}^2+\|u_0\|_{H^3}^2\\
	&+\|\rho_0^{1/2}\theta_{0tt}\|_{L^2}^2+\|\theta_{0t}\|_{H_0^1}^2+\|\theta_0\|_{H^3}^2+1.
	\end{aligned}
\end{equation}
where $u_{0t},~u_{0tt},~\theta_{0t},~\theta_{0tt}$ are determined by $u_0,~\theta_0$ and the equations \eqref{def-eta}, \eqref{1} at $t=0$ (just as \cite{Coutand2011, Gu2015}):
\begin{align*}
	u_{0t}&=\frac{1}{\rho_{0}}\big[- a^r_i (\frac{R\rho_0 \Theta}{J})_{,r} + a^r_j \mathbb{S}^{ij}_{\eta}[v]_{,r}\big]\Big|_{t=0}\\
	&=\frac{1}{\rho_0}\big[-\nabla (R\rho_0 \theta_0)+\dv (\mathbb{S}_{\text{Id}}[u_0])\big],\\
    \theta_{0t}&=\frac{1}{c_v\rho_{0}}\big[-\frac{R\rho_0\Theta}{J}a^r_iv^i_{,r}+  \mathbb{S}^{ij}_{\eta}[v]a^r_jv^i_{,r}+\kappa a^r_i(\nabla_{\eta}\Theta)^i_{,r}\big]\Big|_{t=0}\\
	&=\frac{1}{c_v\rho_0}\big[-R\rho_0\theta_0\dv u_0+\mathbb{S}_{\text{Id}}[u_0]:\nabla u_0 +\kappa \Delta \theta_0\big],\\
	u^i_{0tt}&=\partial_t \big\{\frac{1}{\rho_{0}}\big[- a^r_i (\frac{R\rho_0 \Theta}{J})_{,r} + a^r_j \mathbb{S}^{ij}_{\eta}[v]_{,r}\big]\big\} \Big|_{t=0},\\
	\theta_{0tt}&=\partial_t \big\{\frac{1}{c_v\rho_{0}}\big[-\frac{R\rho_0\Theta}{J}a^r_iv^i_{,r}+  \mathbb{S}^{ij}_{\eta}[v]a^r_jv^i_{,r}+\kappa a^r_i(\nabla_{\eta}\Theta)^i_{,r}\big]\big\} \Big|_{t=0}.
\end{align*}

In the following we denote $P$ as a generic polynomial.
And we shall use the notation ``$f\lesssim g $'' to represent ``$f\leq C g $'', where $C$ is a generic constant independent of $v$ and $\Theta$. Without further statements, $\iint \cdot \ dxdt$ or $\iint \cdot$ will be denoted as $\int_{0}^{T} \int_{\Omega} \cdot \ dxdt$. 

\vspace{0.3cm}
\begin{Thm}\label{mainthm}
	If the initial data $(\rho_0, u_0, \theta_0)$ satisfy $M_0<\infty$, 
	and the initial conditions \eqref{rho1}, \eqref{rho2}, \eqref{rho3}, \eqref{theta1},\eqref{theta2} and the following compatible condition hold,
	\begin{equation}\label{compatible}
	\theta_0,\bpartial \theta_0,\bpartial^2 \theta_0, \theta_{0t} \in H^1_0
	,~ \mathbb{S}^{i3}_{\text{Id}}[v_0]=0, \mathbb{S}^{i3}_{\text{Id}}[\bpartial v_0]=0~\text{on}~\Gamma.
	\end{equation}
	then there exists a unique strong solution $(v(x,t),\Theta(x,t))$ to \eqref{1} in $\Omega\times[0,T]$ for some small enough $T>0$ depending only on $M_0$ and 
	\begin{equation}
	\label{th}
	\sup\limits_{t\in[0,T]} E(v,\Theta) \lesssim P (M_0).
	\end{equation}
\end{Thm}

\begin{Rem}
	Suppose that $(v,\Theta)$ is the strong solution  to $\eqref{1}$ in Theorem \ref{mainthm}. Then 
	$$(\frac{\rho_0}{J},v,\Theta)\circ \eta^{-1}$$
	is the strong solution to \eqref{FNS}.
\end{Rem}

\begin{Rem}
	If $\rho_0=\mathcal{O}(d^{\alpha})$, $\partial_3 \rho_0=\mathcal{O}(d^{\alpha-1})$, $d\partial_3^2 \rho_0=\mathcal{O}(d^{\alpha-1})$ as $d\rightarrow 0^+$ for $\alpha>\frac{2}{3}$, then \eqref{rho2} and \eqref{rho3} are satisfied.
\end{Rem}

\begin{Rem}\label{physical vacuum}
	Suppose that the energy inequality \eqref{th} holds, then \eqref{theta2} ensures that $\Theta$ keeps positive inside $\Omega$ in short time.
	
	Indeed, when $x$ is near the boundary, since $-C\leq\nabla_{\mathbf{n}} \theta_0\leq-\frac{1}{C}$ for some $C>0$ by \eqref{theta2} and 
    \begin{align*}
    \nabla_{\mathbf{n}} \Theta&=\nabla_{\mathbf{n}} \theta_0+ \int_{0}^{t} \nabla_{\mathbf{n}} \Theta_t,\\
    \left|\int_{0}^{t} \nabla_{\mathbf{n}} \Theta_t\right|
    &\leq t^{1/2}\left(\int_{0}^{t}\|D\Theta_t\|_{L^{\infty}}^2\right)^{1/2}\\
    &\lesssim t^{1/2}\left(\int_{0}^{t}\|\Theta_t\|_{H^3}^2\right)^{1/2}
    \lesssim t^{1/2}P(M_0)^{1/2},
    \end{align*}
	then one has $-2C\leq\nabla_{\mathbf{n}} \theta_0\leq-\frac{1}{2C}$ in short time.
	Noted that $\Theta=0$ on the boundary, it follows that  $\Theta$ is positive when $d(x)\leq d_0$ for some $d_0>0$.
	When $d(x)\geq d_0$, there exists a constant $\delta(d_0)>0$ such that $\theta_0(x)\geq \delta(d_0)$.
	Then for any $d(x)\geq d_0$
	\begin{equation*}
	\Theta\geq\theta_0-\left|\int_{0}^{t} \Theta_t \right|\geq \delta(d_0)-t^{1/2}P(M_0)^{1/2},
	\end{equation*}  
	which shows that $\Theta\geq \frac{\delta(d_0)}{2}$ in short time.
	
	Moreover, if considering the non-isentropic Euler equations ($\mu=\lambda=\kappa=0$) of $\rho, u$ and $S$, then the sound speed is given by $\sqrt{\frac{\partial p(\rho, S)}{\partial \rho}}=\sqrt{\gamma \frac{p}{\rho}}=\sqrt{\gamma R\theta}$.
    Therefore, the singular condition on the derivatives of the temperature profiles \eqref{theta2} coincides with the physical vacuum condition for the isentropic fluids (\cite{Jang2009, Coutand2011, Luo2014}) in certain sense. 
\end{Rem}

\begin{Rem}
	\label{rem_entropy}
	Suppose that the energy inequality \eqref{th} holds. 
	If initially the entropy $S(\cdot,0)=S_0(\cdot)$ is bounded, i.e. $\|S_0\|_{L^{\infty}} \leq C$, and $\alpha=\frac{1}{\gamma-1}$ in \eqref{rho2}, then it follows from the regularity of the strong solutions that 
	\begin{align*}
	\frac{\bar A}{R}e^{S(\gamma-1)/R}&=\frac{\Theta}{\rho^{\gamma-1}}
	=\mathcal{O}(\frac{\Theta}{d})=\mathcal{O}(|\nabla_{\mathbf{n}} \Theta|)~\qquad\text{as}~d\rightarrow 0^+.
	\end{align*}
	Remark \ref{physical vacuum} has shown that $\nabla_{\mathbf{n}}\Theta$ has upper and lower bounds on the boundary in short time, and so is the entropy.
	Therefore, this result achieves the aim to find solutions to \eqref{FNS} with vanishing density on the free boundary and bounded entropy in the space-time, which extends the results of the isentropic case (\cite{Coutand2012, Jang2015, Jang2010}) to the non-isentropic case.
\end{Rem}

\vspace{0.7cm}
\section{Preliminaries}~

Before using the energy method to prove Theorem \ref{mainthm}, in this section we will present the inequalities which will be frequently used in the energy estimates.
\vspace{0.3cm}

\subsection{The Korn's inequality and the Hardy's inequality}~

\begin{Lem}[Korn's inequality]
	Suppose that $v$ is a function defined in $\Omega$ and $\int_{\Omega} \sum\limits_{i,j}(v^i_{,j} +v^j_{,i})^2\ dx+\|v\|_{L^2}^2<\infty$. Then 
\begin{equation}\label{Korn}
\|v\|_{H^1}^2\lesssim \int_{\Omega} \sum\limits_{i,j}(v^i_{,j} +v^j_{,i})^2\ dx+\|v\|_{L^2}^2.
\end{equation}
\end{Lem}
\begin{proof}
	See \cite[Theorem~2.1]{Ciarlet2010} or \cite[Lemma~5.1]{Zajaczkowski1993}. The original Korn's inequality can be found in \cite{Friedrichs1947}.
\end{proof}

\begin{Lem}[Hardy's inequality]
If $k\in \mathbb{R}$ and $g$ satisfies $\int_{0}^{1} s^k[g^2+(g')^2]\ ds<\infty$, then we have that
\begin{enumerate}
\item[(1)] if $k>1$, then
\begin{equation}
\int_{0}^{1}s^{k-2}g^2 \ ds\lesssim  \int_{0}^{1}s^k[g^2+(g')^2]\ ds.
\end{equation}
\item[(2)] if $k<1$, then  $g$ has a trace at $x=0$ and
\begin{equation}
\int_{0}^{1} s^{k-2} (g-g(0))^2 ds \lesssim \int_{0}^{1} s^k(g')^2 \ ds.
\end{equation}
\end{enumerate}
\end{Lem}
\begin{proof}
	See \cite[Lemma~4.2]{Jang2014}.
\end{proof}

\begin{Cor}
	Suppose that $v$ is a function defined on $\Omega$ and $\int_{\Omega} \sum\limits_{i,j}(v^i_{,j} +v^j_{,i})^2\ dx+\|\rho_0^{1/2} v\|_{L^2}$. Then
	\begin{equation}\label{Korn2}
	\|v\|_{H^1}^2\lesssim \int_{\Omega} \sum\limits_{i,j}(v^i_{,j} +v^j_{,i})^2\ dx+\|\rho_0^{1/2} v\|_{L^2}.
	\end{equation}
\end{Cor}
\begin{proof}
	It follows from the Hardy's inequality that,
	\begin{align*}
	\int_{0}^{1}g^2 \ ds&\lesssim\int_{0}^{1}s^2[g^2+(g')^2]\ ds\lesssim  \int_{0}^{1}s^4[g^2+(g')^2]\ ds+ \int_{0}^{1}s^2(g')^2\ ds\\
	&\lesssim \cdots\\
	&\lesssim \int_{0}^{1}s^{2m}g^2+(s^2+s^4+\cdots+s^{2m})(g')^2\ ds,\quad \text{for any~}m\in\mathbb{N}
	\end{align*}
    Therefore, for any $\alpha<\infty$, after choosing $2m\geq \alpha$, one can obtain, 
	\begin{equation*}
	\int_{0}^{1}g^2 \ ds\lesssim \int_{0}^{1}s^{\alpha}g^2+s^2(g')^2\ ds,
	\end{equation*}
	where we have employed the fact that $s\in[0,1]$.
	Then by rescaling the inequality, we have
	$$
	\int_{0}^{\omega}g^2 \ ds \lesssim \frac{1}{\omega^{\alpha}}\int_{0}^{\omega}s^{\alpha}g^2\ ds
	+\int_{0}^{\omega}s^2(g')^2 \ ds.
	$$
	Therefore, after applying the above inequality to the integral with respect to $x_3$, one has
	$$
	\begin{aligned}
	\|v\|_{L^2}^2&=\int_{0}^{1}\int_{0}^{1}\left( \int_{\omega}^{1-\omega}+\int_{0}^{\omega}+\int_{1-\omega}^{1}|v|^2 \ dx_3\right) dx_2 dx_1\\
	&\lesssim \|\frac{1}{\rho_0}\|_{L^{\infty}(\{\omega<x_3<1-\omega\})}\int_{\Omega}\rho_0 |v|^2 \ dx+ \frac{2}{\omega^{\alpha}}\int_{\Omega}\rho_0 |v|^2\ dx \\
	&\qquad +  \int_{0}^{1}\int_{0}^{1}\left( \int_{0}^{\omega}+\int_{1-\omega}^{1} d^2 \left|\frac{\partial v}{\partial x_3}\right|^2 \ dx_3\right) dx_2 dx_1\\
	&\lesssim \frac{1}{\omega^{\alpha}} \|\rho_0^{1/2} v\|_{L^2}^2 + \omega^2 \|v\|_{H^1}^2,
	\end{aligned}
	$$
	where $\|\frac{1}{\rho_0}\|_{L^{\infty}(\{\omega<x_3<1-\omega\})}\lesssim \frac{1}{\omega^{\alpha}}$ for \eqref{rho2}.
	Then substituting the above inequality to the Korn's inequality \eqref{Korn} yields
     $$	
     \|v\|_{H^1}^2\lesssim \int_{\Omega} \sum\limits_{i,j}(v^i_{,j} +v^j_{,i})^2\ dx+\frac{2}{\omega^{\alpha}} \|\rho_0^{1/2} v\|_{L^2}^2 + \omega^2 \|v\|_{H^1}^2.
     $$
	With $0<\omega<1$ small enough, this finishes the proof.
\end{proof}

\vspace{0.3cm}
\subsection{Differentiation and inequalities of $J$, $a$ and $A$}~

When differentiating the equations \eqref{1} and doing energy estimates, differentiating $J$,
$a$ and $A$ and estimating their norms are unavoidable.
By the definition of J and a, the time and spatial derivatives of J and a are
\begin{equation}\label{deri_Ja}
\begin{aligned}
\partial_t J&=a^s_r v^r_{,s}=J\dv_{\eta} v,\\
D J&=a^s_r D\eta^r_{,s},\\
\partial_t a^k_i&=J^{-1}v^r_{,s}(a^s_ra^k_i-a^s_ia^k_r),\\
D a^k_i&=J^{-1}D\eta^r_{,s}(a^s_ra^k_i-a^s_ia^k_r).
\end{aligned}
\end{equation}
The derivatives of $A$ can be obtained by the relation $a = JA$ and the above derivatives.
With the above expressions, one can directly obtain the important Piola identity
\begin{equation}
a^k_{i,k}=0.
\end{equation}


The following inequalities can be obtained by direct calculations of the derivatives of $J$, $a$ and $A$.
\begin{align*}
& |a|\lesssim |D\eta|^2,\qquad\qquad\qquad\qquad\qquad|\partial_t a |\lesssim  \Big(J^{-1}|Dv| |a|^2\Big), \\
& |\partial_t^2 a |\lesssim \Big(J^{-1}|Dv_t| |a|^2+J^{-2}|Dv|^2|a|^3\Big),\\
& |\bpartial a |\lesssim  \Big(J^{-1}|\bpartial  D \eta| |a|^2\Big),\qquad\qquad |\bpartial\partial_t a |\lesssim  \Big(J^{-1}|\bpartial  D v| |a|^2+J^{-2}|\bpartial D\eta| |Dv| |a|^3\Big),\\
& |\bpartial^2 a |\lesssim  \Big(J^{-1}|\bpartial^2  D \eta| |a|^2 + J^{-2}|\bpartial  D \eta|^2|a|^3\Big),\\
& |\bpartial^2 \partial_t a|\lesssim  \Big(J^{-1}|\bpartial^2   D v| |a|^2 +J^{-2}|\bpartial^2   D \eta| |Dv| |a|^3 + J^{-2}|\bpartial  D \eta| |\bpartial  D v| |a|^3+J^{-3}|\bpartial  D \eta|^2 |Dv| |a|^4\Big),\\
&|\bpartial^3 a|\lesssim  \Big(J^{-1}|\bpartial^3   D \eta| |a|^2 +J^{-2}|\bpartial^2   D \eta| |\bpartial D \eta| |a|^3 +J^{-3}|\bpartial  D \eta|^3 |a|^4 \Big)\\
& |A|\lesssim J^{-1}|D\eta|^2,\qquad\qquad\qquad\qquad |\partial_t A |\lesssim  \Big(J^{-2}|Dv| |a|^2\Big),\\
& |\partial_t^2 A |\lesssim  \Big(J^{-2}|Dv_t| |a|^2+J^{-3}|Dv|^2|a|^3\Big),\\
& |\bpartial A |\lesssim  \Big(J^{-2}|\bpartial  D \eta| |a|^2\Big),\qquad\qquad |\bpartial\partial_t A |\lesssim  \Big(J^{-2}|\bpartial  D v| |a|^2+J^{-3}|\bpartial D\eta| |Dv| |a|^3\Big),\\
& |\bpartial^2 A |\lesssim  \Big(J^{-2}|\bpartial^2  D \eta| |a|^2 + J^{-3}|\bpartial  D \eta|^2|a|^3\Big),\\
& |\bpartial^2 \partial_t A|\lesssim  \Big(J^{-2}|\bpartial^2   D v| |a|^2 +J^{-3}|\bpartial^2   D \eta| |Dv| |a|^3 + J^{-3}|\bpartial  D \eta| |\bpartial  D v|  |a|^3+J^{-4}|\bpartial  D \eta|^2 |Dv| |a|^4\Big),\\
& |\bpartial^3  A|\lesssim  \Big(J^{-2}|\bpartial^3   D \eta| |a|^2 +J^{-3}|\bpartial^2   D \eta| |\bpartial D\eta| |a|^3  +J^{-4}|\bpartial  D \eta|^3 |a|^4\Big).
\end{align*}

\vspace{0.3cm}
\subsection{Sobolev embedding and the Minkowski integral inequalities}~

In the energy estimates, we will use the following Sobolev embedding inequalities
\begin{equation}
\|\cdot\|_{L^{\infty}}\lesssim \|\cdot\|_{H^2},\ \|\cdot\|_{L^4}\lesssim  \|\cdot\|_{H^1},\ \|\cdot\|_{L^6}\lesssim  \|\cdot\|_{H^1},
\end{equation}
and the Minkowski integral inequality
\begin{equation}
\Big\|\int_{0}^{t} f(\cdot,t') \ dt'\Big\|_{L^p}\leq \int_{0}^{t} \|f(\cdot,t')\|_{L^p}\ dt'.
\end{equation}

\vspace{0.7cm}
\section{A priori estimates}~

Define
\begin{equation}
\begin{aligned}
F(v,\Theta)(t)=&E(v,\Theta)(t)+\int_{0}^{t} \|v_{tt}(\cdot,t')\|_{H^1}^2+\|v_t(\cdot,t')\|_{H^3}^2+\|\bpartial v(\cdot,t')\|_{H^3}^2
\\
&\qquad+\|\Theta_{tt}(\cdot,t')\|_{H_0^1}^2+\|\Theta_t(\cdot,t')\|_{H^3}^2+\|\bpartial \Theta(\cdot,t')\|_{H^3}^2\ dt'.
\end{aligned}
\end{equation}
In this section we plan to prove:
\begin{Prop}\label{prop-of-estimates}
	Suppose that the assumptions in Theorem \ref{mainthm} hold, and $(v(x,t),\Theta(x,t))$ is the smooth solution to \eqref{1} in $\Omega\times[0,T]$ for some $T>0$, and the following a priori assumption is satisfied
	\begin{equation}
	\label{a priori assumption}
	|D\eta|\lesssim 1,\
	\frac{1}{2}\leq J(t)\leq \frac{3}{2} \qquad \text{for}\ t\in[0,T].
	\end{equation}
	Then 
    \begin{equation}
    \label{a priori estimates}
    \sup\limits_{t\in[0,T]}F(v,\Theta)\lesssim  TP\Big(\sup\limits_{t\in[0,T]}F(v,\Theta)\Big)+P(M_0).
    \end{equation}
\end{Prop}

\begin{Rem}\label{rem-assumption}
	Indeed, when the assumptions in Theorem \ref{mainthm} hold, $(v,\Theta)$ is the regular solution to \eqref{1} in $\Omega \times[0,T]$ and $\sup\limits_{t\in[0,T]} F(v,\Theta)<2P(M_0)$,
	\eqref{a priori assumption} can be directly verified in short time following the arguments below:
	\begin{equation}
	\begin{aligned}
	|D\eta| &\leq \int_{0}^{t} \|Dv\|_{L^{\infty}}\ dt' + |D\eta_0|\leq \int_{0}^{t} \|v\|_{H^3}\ dt' + |D\eta_0|\\
	&\leq t \sup\limits_{t'\in[0,t]}F^{1/2} +M_0\leq t \Big(2P(M_0)\Big)^{1/2}+M_0,\\
	|J-1|   &\leq \int_{0}^{t} \|J_t\|_{L^{\infty}}\ dt'\leq \int_{0}^{t} \|D\eta\|_{L^{\infty}}^2\|Dv\|_{L^{\infty}} \ dt' \\
	&\leq  t \Big(t \sup\limits_{t'\in[0,t]}F^{1/2} +M_0\Big)^2 \sup\limits_{t'\in[0,t]}F^{1/2} \leq t \Big\{t[2P(M_0)]^{1/2} +M_0\Big\}^2 \Big(2P(M_0)\Big)^{1/2}.
	\end{aligned}
	\end{equation}
\end{Rem}

This section comprises four parts: Section 4.1 is the preliminaries which are the estimates directly obtained by the Sobolev embedding inequalities or the Minkowski integral inequality; Section 4.2 and Section 4.3 are the part of higher order estimates and lower oder estimates respectively, which are derived by differentiating the equation \eqref{1}, multiplying proper functions and integrating them; Section 4.4 is the elliptic estimates, which make use of the ellipticity of diffusion terms. 
For simplicity, we write  $E(t):=E(v,\Theta)(t)$ and $F(t):=F(v,\Theta)(t)$.

\vspace{0.3cm}

\subsection{Preliminaries}~
\vspace{0.3cm}

\textsl{\textbf{Estimates of $\|v_t(\cdot,t)\|_{H^1},~\|\Theta(\cdot,t)\|_{H_0^1},~ 
		\|(v, \Theta)(\cdot,t)\|_{H^3}$:}}
Direct calculations give that
\begin{equation}
\label{direct}
\begin{aligned}
\|v_t(\cdot,t)\|_{H^1}^2
&\leq \left(\int_{0}^{t}\|v_{tt}\|_{H^1}\ dt'+M_0^{1/2}\right)^2\\
&\leq \left( \int_{0}^{t} 1 \ dt'  \right)
\left(\int_{0}^{t}\|v_{tt}\|_{H^1}^2\ dt'\right)+M_0
\leq t \sup\limits_{t'\in[0,t]}F+M_0,\\
\|\Theta_t(\cdot,t)\|_{H_0^1}^2
&\leq t\sup\limits_{t'\in[0,t]}F+M_0,\\
\|\Big(v,\Theta\Big)(\cdot,t)\|_{H^3}^2
&\leq \left( \int_{0}^{t}\|\Big(v_t,\ \Theta_t\Big)\|_{H^3}\ dt'+M_0^{1/2}\right)^2
\leq t \sup\limits_{t'\in[0,t]}F+M_0.
\end{aligned}
\end{equation}

\textsl{\textbf{Estimates of $a$, $A$:}}
By the a priori assumption \eqref{a priori assumption}, in the following we will use the coarser estimates of $a$ and $A$ shown below
\begin{align*}
& |a|,\ |A|\lesssim 1, \\
& |\partial_t a |,\ |\partial_t A |\lesssim  \Big(|Dv| \Big),\\
& |\partial_t^2 a |,\ |\partial_t^2 A |\lesssim \Big( |Dv_t|+|Dv|^2\Big),\\
& |\bpartial a |,\ |\bpartial A |\lesssim  \Big( |\bpartial  D \eta|\Big),\\
& |\bpartial\partial_t a |,\ |\bpartial\partial_t A |\lesssim  \Big(|\bpartial  D v|+|\bpartial D\eta| |Dv|\Big),\\
& |\bpartial^2 a |,\ |\bpartial^2 A |\lesssim  \Big(|\bpartial^2  D \eta| + |\bpartial  D \eta|^2\Big),\\
& |\bpartial^2 \partial_t a|,\ |\bpartial^2\partial_t A |\lesssim  \Big( |\bpartial^2   D v| +|\bpartial^2   D \eta| |Dv|  + |\bpartial  D \eta| |\bpartial  D v|+|\bpartial  D \eta|^2 |Dv|\Big),\\
&|\bpartial^3 a|,\ |\bpartial^3 A |\lesssim  \Big( |\bpartial^3   D \eta| +|\bpartial^2   D \eta| |\bpartial D \eta| +|\bpartial  D \eta|^3 \Big).
\end{align*}

\textsl{\textbf{Other direct estimates:}}
It follows from the Sobolev embedding inequalities and
the Minkowski integral inequality,  
and the above estimates \eqref{direct} that: 
if $t<1$ and \eqref{a priori assumption} holds, then
\begin{align*}
&\|Dv(\cdot,t)\|_{L^{\infty}}^2+\| D^2 v(\cdot,t)\|_{L^4}^2+\|\Big(D^3 v,\ Dv_t\Big)(\cdot,t)\|_{L^2}^2\lesssim t \sup\limits_{t'\in[0,t]}F+M_0,\\
&\|\bpartial D\eta(\cdot,t)\|_{L^{\infty}}^2+\|\bpartial D^2 \eta(\cdot,t)\|_{L^4}^2+\| \bpartial D^3 \eta(\cdot,t)\|_{L^2}^2 \lesssim t \sup\limits_{t'\in[0,t]}F+M_0,\\
&\| D^2 \eta(\cdot,t)\|_{L^4}^2+\|  D^3 \eta(\cdot,t)\|_{L^2}^2 \lesssim 
( t\sup\limits_{t'\in[0,t]}\|v\|_{H^3}+M_0^{1/2})^2\lesssim t  \sup\limits_{t'\in[0,t]}F+M_0,\\
&\|\bpartial \rho_0\|_{L^4}\lesssim 1,
\end{align*}
\begin{align*}
\| A^i_j-\delta^i_j\|_{L^{\infty}}^2
&\lesssim  \Big(t\sup\limits_{t'\in[0,t]}\|\partial_t A\|_{L^{\infty}}\Big)^2
\lesssim \Big(t\sup\limits_{t'\in[0,t]}\|D v\|_{L^{\infty}}\Big)^2
\lesssim t^2 \sup\limits_{t'\in[0,t]}F,\\
\| aA,\  a,\  A\|_{L^{\infty}}^2
&\lesssim \|A-\delta\|_{L^{\infty}}^2+\|\delta\|_{L^{\infty}}^2
\lesssim t P(\sup\limits_{t'\in[0,t]}F)+P(M_0),\\
\|\partial_t (aA),\ \partial_t a,\ \partial_t A\|_{L^{\infty}}^2
&\lesssim \|Dv\|_{L^{\infty}} ^2
\lesssim t P(\sup\limits_{t'\in[0,t]}F)+P(M_0),\\
\|\partial_t^2(aA),\ \partial_t^2 a,\ \partial_t^2 A\|_{L^2}^2
&\lesssim \Big( \|Dv_t\|_{L^2}+\|Dv\|_{L^{\infty}}^2 \Big)^2
\lesssim t P(\sup\limits_{t'\in[0,t]}F)+P(M_0),\\
\|\bpartial (aA),\ \bpartial a,\ \bpartial A\|_{L^{\infty}}^2
&\lesssim \|\bpartial D \eta\|_{L^\infty} ^2
\lesssim t P(\sup\limits_{t'\in[0,t]}F)+P(M_0),\\
\|\bpartial\partial_t(aA),\ \bpartial \partial_t a,\ \bpartial \partial_t A\|_{L^4}^2
&\lesssim \Big(\|\bpartial Dv\|_{L^4}+\|\bpartial D\eta\|_{L^{\infty}}\|Dv\|_{L^{\infty}}   \Big)^2
\lesssim t P(\sup\limits_{t'\in[0,t]}F)+P(M_0),\\
\|\bpartial^2 (aA),\ \bpartial^2 a, \bpartial^2 A\|_{L^4}^2
&\lesssim \Big(\|\bpartial^2 D\eta\|_{L^4}+\|\bpartial D\eta\|_{L^{\infty}}^2\Big)^2
\lesssim t P(\sup\limits_{t'\in[0,t]}F)+P(M_0),\\
\|\bpartial^2 \partial_t (aA),\ \bpartial^2 \partial_t a,\ \bpartial^2 \partial_t A\|_{L^2}^2
&\lesssim \Big(\|\bpartial^2 Dv\|_{L^2}+\|\bpartial ^2 D\eta\|_{L^4}\|Dv\|_{L^{\infty}}+\|\bpartial D\eta\|_{L^{\infty}}\|\bpartial Dv\|_{L^4}\\
&+\|\bpartial D\eta\|_{L^{\infty}}^2\|Dv\|_{L^{\infty}}\Big)^2
\lesssim t P(\sup\limits_{t'\in[0,t]}F)+P(M_0),\\
\|\bpartial^3 (aA),\ \bpartial^3 a, \bpartial^3 A\|_{L^2}^2
&\lesssim \Big(\|\bpartial^3 D\eta\|_{L^2}+\|\bpartial^2 D\eta\|_{L^4}\|\bpartial D\eta\|_{L^{\infty}}+\|\bpartial D\eta\|_{L^{\infty}}^3\Big)^2\\
&\lesssim t P(\sup\limits_{t'\in[0,t]}F)+P(M_0),\\
\|D(aA),\ Da,\ DA\|_{L^4}^2
&\lesssim \| D^2\eta\|_{L^4}^2\lesssim t P(\sup\limits_{t'\in[0,t]}F)+P(M_0),\\
\|\bpartial D (aA),\ \bpartial D a,\ \bpartial D A\|_{L^4}^2
&\lesssim \Big( \|\bpartial D^2 \eta\|_{L^4}+\|D^2 \eta\|_{L^4}\|\bpartial D\eta\|_{L^{\infty}} \Big)^2
\lesssim t P(\sup\limits_{t'\in[0,t]}F)+P(M_0),\\
\|D^2(aA),\ D^2a,\ D^2A\|_{L^2}^2
&\lesssim \Big(\|D^3\eta\|_{L^2}+\|D^2\eta\|_{L^4}^2 \Big)^2
\lesssim t P(\sup\limits_{t'\in[0,t]}F)+P(M_0),\\
\|\bpartial D^2(aA),\ \bpartial D^2a,\ \bpartial D^2A\|_{L^2}^2
&\lesssim \Big(\|\bpartial D^3\eta\|_{L^2}+\|D^3\eta\|_{L^2}\|\bpartial D\eta\|_{L^{\infty}}+\|\bpartial D^2 \eta\|_{L^4}\|D^2\eta\|_{L^4}\\
&\|D^2\eta\|_{L^4}^2\|\bpartial D\eta\|_{L^{\infty}} \Big)^2
\lesssim t P(\sup\limits_{t'\in[0,t]}F)+P(M_0),\\
\|\partial_t D(aA),\ \partial_t Da,\ \partial_t DA\|_{L^4}^2
&\lesssim \Big( \| D^2 v\|_{L^4}+\|D^2\eta\|_{L^4}\|Dv\|_{L^{\infty}}\Big)^2
\lesssim t P(\sup\limits_{t'\in[0,t]}F)+P(M_0).
\end{align*}
Note that the estimate of $\|A^i_j-\delta^i_j\|_{L^{\infty}}$ does not involve $M_0$.
\vspace{0.5cm}

\subsection{Higher order energy estimates}~

\vspace{0.3cm}
\textsl{\textbf{Estimates of $\|\rho_0^{1/2}v_{tt}(\cdot,t)\|_{L^2}^2, \int_{0}^{t}\|v_{tt}\|_{H^1}^2 \ dt'$:}}~~
After taking twice temporal derivatives of $\eqref{1}_1$ and taking inner product with $v_{tt}$, that is $\int_{0}^{t}\int_{\Omega}\partial_t^2 \Big(\eqref{1}_1\Big)\cdot  v_{tt} \ dx dt'$, we have 
$$
\frac{1}{2}\int_{\Omega} \rho_0 |v_{tt}|^2 \bigg|_{t'=0}^{t}  
-\underbrace{\iint  \partial_t^2 (a^r_i\frac{R\rho_0\Theta}{J}) v^i_{tt,r}}_{\mathrm{I}}
+\underbrace{\iint \partial_t^2 (a^r_j\mathbb{S}^{ij}_{\eta}[v]) v^i_{tt,r}}_{\mathrm{II}}=0.
$$
(Notice that here we omit the summation symbol $\sum\limits_{i=1,2,3}$ of the dot product of vectors $ \partial_t^2 (a^r_j\mathbb{S}^{ij}_{\eta}[v]) v^i_{tt,r}$ for convenience.)
Here $\mathrm{II}$ can be written as
$$
\mathrm{II}=\iint a^r_j\Big[\mu (A^{k}_j  v^i_{tt,k}+A^k_i  v^j_{tt,k})+\lambda (A_l^k  v^l_{tt,k})\delta^i_j\Big] v^i_{tt,r}+\mathrm{II}',
$$
where $\mathrm{II}'$ denotes all the remaining terms, and
\begin{align*}
|\mathrm{II}'|&\lesssim  \iint \Big[|\partial_t^2(aA)| |Dv|
   +|\partial_t (aA)| |Dv_t|\Big]|Dv_{tt}|\\
&\lesssim  \int_{0}^{t} \Big[\|\partial_t^2(aA)\|_{L^2} \|Dv\|_{L^{\infty}}
   +\|\partial_t (aA)\|_{L^{\infty}} \|Dv_t\|_{L^2}\Big]\|Dv_{tt}\|_{L^2}\ dt'\\
&\lesssim \int_{0}^{t}\frac{1}{\epsilon}
   \Big[\|\partial_t^2(aA)\|_{L^2} \|v\|_{H^3}
      +\|\partial_t (aA)\|_{L^{\infty}} \|v_t\|_{H^1}\big]^2
   +\epsilon \|Dv_{tt}\|_{L^2}^2\ dt'\\
&\lesssim\frac{1}{\epsilon} tP(\sup\limits_{t'\in[0,t]} F)+\frac{1}{\epsilon}P(M_0)
   + \epsilon\int_{0}^{t} \|Dv_{tt}\|_{L^2}^2\ dt'.
\end{align*}
Noted that $\frac{\partial}{\partial y_j} v_{tt}^i=A_j^k v^i_{tt,k}$, one has 
\begin{align*}
\mathrm{II}-\mathrm{II}'& =\iint J \bigg[ \mu \Big( \frac{\partial}{\partial y_j}v_{tt}^i+\frac{\partial}{\partial y_i}v_{tt}^j   \Big) +\lambda \Big( \frac{\partial}{\partial y_l}v_{tt}^l \Big) \delta^i_j \bigg] \frac{\partial}{\partial y_j}v_{tt}^i\\
&=\iint J\mu \bigg[ \sum\limits_{i,j} \Big( \frac{\partial}{\partial y_j}v_{tt}^i   \Big)^2 +\sum\limits_{i,j} \Big( \frac{\partial}{\partial y_j}v_{tt}^i   \Big)\Big( \frac{\partial}{\partial y_i}v_{tt}^j   \Big)\bigg] +J\lambda \Big( \sum\limits_{i} \frac{\partial}{\partial y_i}v_{tt}^i \Big)^2 \\
&=\iint J \bigg[2\mu \sum\limits_{i} \Big( \frac{\partial}{\partial y_i}v_{tt}^i   \Big)^2 +\lambda \Big(\sum\limits_{i} \frac{\partial}{\partial y_i}v_{tt}^i   \Big)^2\bigg] +J \mu \sum\limits_{i>j}\Big(  \frac{\partial}{\partial y_j}v_{tt}^i +\frac{\partial}{\partial y_i}v_{tt}^j \Big)^2\\
&\geq\begin{cases}
\iint J (2\mu+3\lambda) \sum\limits_{i} \Big( \frac{\partial}{\partial y_i}v_{tt}^i   \Big)^2 +J \mu \sum\limits_{i>j}\Big(  \frac{\partial}{\partial y_j}v_{tt}^i +\frac{\partial}{\partial y_i}v_{tt}^j \Big)^2, &\text{if}~ \lambda\leq 0,\\
\iint J (2\mu) \sum\limits_{i} \Big( \frac{\partial}{\partial y_i}v_{tt}^i   \Big)^2 +J \mu \sum\limits_{i>j}\Big(  \frac{\partial}{\partial y_j}v_{tt}^i +\frac{\partial}{\partial y_i}v_{tt}^j \Big)^2, &\text{if}~ \lambda\geq 0,
\end{cases}
\end{align*}
\begin{align*}
&\geq \iint J \min\{ \frac{2\mu+3\lambda}{4},\frac{\mu}{2}\}\sum\limits_{i,j}\Big(  \frac{\partial}{\partial y_j}v_{tt}^i +\frac{\partial}{\partial y_i}v_{tt}^j \Big)^2 
\geq \frac{1}{C}\iint  \sum\limits_{i,j}\Big(A^k_j  v_{tt,k}^i+A^k_i  v_{tt,k}^j\Big)^2\\
&\geq\frac{1}{C}\iint  \sum\limits_{i,j}\Big(\delta^k_j  v_{tt,k}^i+\delta^k_i  v_{tt,k}^j\Big)^2-C|A-\delta| |Dv_{tt}|^2 \ dxdt'\\
&\geq\frac{1}{C}\iint\sum\limits_{i,j}( v_{tt,j}^i+ v_{tt,i}^j)^2 \ dxdt'
-C \int_{0}^{t}  (t' \sup\limits_{[0,t']}F^{1/2} ) \| Dv_{tt}\|_{L^2}^2\ dt'\\
&\geq\frac{1}{C}\iint\sum\limits_{i,j}( v_{tt,j}^i+ v_{tt,i}^j)^2-CtP(\sup\limits_{t'\in[0,t]}F).
\end{align*}
Together with the Korn's inequality \eqref{Korn2}, this yields
\begin{align*}
\mathrm{II}-\mathrm{II}' &\geq \frac{1}{C}\int_{0}^{t} \|v_{tt}\|_{H^1}^2\ dt'
   -C\int_{0}^{t}\|\rho_0^{1/2} v_{tt}\|_{L^2}^2\ dt'-CtP(\sup\limits_{t'\in[0,t]}F)\\
&\geq \frac{1}{C}\int_{0}^{t} \|v_{tt}\|_{H^1}^2\ dt'-CtP(\sup\limits_{t'\in[0,t]}F).
\end{align*}
On the other hand,
\begin{align*}
|\mathrm{I}|&\lesssim \iint \Big(|\partial_t^2 A| |\rho_0\Theta|
+|\partial_t A| |\rho_0 \Theta_t|+|A| |\rho_0 \Theta_{tt}|\Big) |Dv_{tt}| \\
&\lesssim \int_{0}^{t} \Big(\|\partial_t^2 A\|_{L^2} \|\rho_0\|_{L^{\infty}}\|\Theta\|_{L^{\infty}}
+\|\partial_t A\|_{L^{\infty}} \|\rho_0\|_{L^{\infty}}\|\Theta_t\|_{L^2}\\
&\qquad+\|A\|_{L^{\infty}}\|\rho_0^{1/2}\|_{L^{\infty}}\|\rho_0^{1/2}\Theta_{tt}\|_{L^2}\Big) \|Dv_{tt}\|_{L^2} \ dt'\\
&\lesssim \int_{0}^{t} \frac{1}{\epsilon}\Big[
   \Big( t'P(\sup\limits_{[0,t']}F)+P(M_0)  \Big)^{1/2}
   \Big(\|\Theta\|_{H^2}+\|\Theta_t\|_{H_0^1}+\|\rho_0^{1/2}\Theta_{tt}\|_{L^2}\Big)  \Big]^2+\epsilon \|Dv_{tt}\|_{L^2}^2\ dt'\\
&\lesssim \frac{1}{\epsilon}tP(\sup\limits_{t'\in[0,t]}F)+\frac{1}{\epsilon}P(M_0)+\epsilon \int_{0}^{t} \|Dv_{tt}\|_{L^2}^2\ dt'.
\end{align*}

After combining the estimates of $\mathrm{I}, \mathrm{II}$ with small enough $\epsilon$ (independent of $v$ or $\Theta$), we have that
\begin{equation}\label{vtt}
\|\rho_0^{1/2}v_{tt}(\cdot,t)\|^2_{L^2}+\int_{0}^{t}\|v_{tt}\|_{H^1}^2\ dt'
\lesssim tP(\sup\limits_{t'\in[0,t]}F)+P(M_0).
\end{equation}

\vspace{0.3cm}
\textsl{\textbf{Estimates of $\|\rho_0^{1/2}\Theta_{tt}(\cdot,t)\|_{L^2}^2, \int_{0}^{t}\|\Theta_{tt}\|_{H_0^1}^2 \ dt'$:}}~~After taking twice temporal derivatives of $\eqref{1}_2$ and taking inner product with $\Theta_{tt}$, that is $\int_{0}^{t}\int_{\Omega}\partial_t^2 \Big(\eqref{1}_2\Big)\cdot  \Theta_{tt}\ dx dt'$, we have
\begin{align*}
&\frac{c_v}{2}\int_{\Omega} \rho_0 |\Theta_{tt}|^2 \bigg|_{t'=0}^{t}  
+\underbrace{\iint  \partial_t^2 \Big(\frac{R\rho_0\Theta}{J} a^r_i v^i_{,r}\Big) \Theta_{tt}}_{\mathrm{I}}
-\underbrace{\iint \partial_t^2 (\mathbb{S}^{ij}_{\eta}[v] a^r_j v^i_{,r}) \Theta_{tt}}_{\mathrm{II}}\\
&\qquad\qquad\qquad
+\underbrace{\kappa\iint \partial_t^2 (a^r_i A^k_j \delta^{ij}\Theta_{,k}) \Theta_{tt,r}}_{\mathrm{III}}=0.
\end{align*}
$\mathrm{III}$ can be expended as 
$$
\mathrm{III}=\iint a^r_i A^{k}_j \delta^{ij}  \Theta_{tt,k}\Theta_{tt,r}+\mathrm{III}',
$$
where $\mathrm{III}'$ denotes all the remaining terms.
Similarly with $\mathrm{II}'$ in the estimates of $v_{tt}$, we have
\begin{align*}
|\mathrm{III}'|
&\lesssim \int_{0}^{t}\frac{1}{\epsilon}
   \Big[\|\partial_t^2(aA)\|_{L^2} \|\Theta\|_{H^3}
      +\|\partial_t (aA)\|_{L^{\infty}} \|\Theta_t\|_{H_0^1}\big]^2
   +\epsilon \|D\Theta_{tt}\|_{L^2}^2\ dt'\\
&\lesssim \frac{1}{\epsilon} tP(\sup\limits_{t'\in[0,t]} F)+\frac{1}{\epsilon}P(M_0)
   + \epsilon\int_{0}^{t} \|D\Theta_{tt}\|_{L^2}^2\ dt',\\
\mathrm{III}-\mathrm{III}'&\geq \frac{1}{C}\iint \delta_i^k \delta^{ir} \Theta_{tt,k}\Theta_{tt,r}
    -C\iint |A-\delta||D\Theta_{tt}|^2\\
&\geq \frac{1}{C} \int_{0}^{t} \|D\Theta_{tt}\|_{L^2}^2 \ dt'
    -C \int_{0}^{t}  \|A-\delta\|_{L^{\infty}}\|D\Theta_{tt}\|_{L^2}^2\\
&\geq \frac{1}{C} \int_{0}^{t} \|D\Theta_{tt}\|_{L^2}^2 \ dt'-CtP(\sup\limits_{t'\in[0,t]}F).
\end{align*}
Following similar argument as those in previous estimates on the temporal derivatives of $v$, one has
\begin{align*}
|\mathrm{I}|&\lesssim \iint \Big[ |\partial_t^2 (A\rho_0 \Theta)| |Dv|+|\partial_t (A\rho_0 \Theta)| |Dv_t|+ |A\rho_0\Theta| |Dv_{tt}| \Big] |\Theta_{tt}|\\
&\lesssim \int_{0}^{t} \Big[ \|\partial_t^2 (A\rho_0 \Theta)\|_{L^2} \|Dv\|_{L^{\infty}}
+\Big(\|\partial_t A\|_{L^{\infty}}\|\rho_0\|_{L^{\infty}}\|\Theta\|_{L^\infty}
+\|A\|_{L^{\infty}}\|\rho_0\|_{L^{\infty}}\|\Theta_t\|_{L^4}\Big) 
\|Dv_t\|_{L^2}\\
&\qquad
+\|A\|_{L^{\infty}} \|\rho_0\|_{L^{\infty}}\|\Theta\|_{L^{\infty}} \|Dv_{tt}\|_{L^2} \Big]\|\Theta_{tt}\|_{L^4}\\
&\lesssim \int_{0}^{t} \frac{1}{\epsilon}\bigg[ \Big( \|\partial_t^2 A\|_{L^2}\|\Theta\|_{H^2}+\|\partial_t A\|_{L^{\infty}}\|\Theta_t\|_{H_0^1}+ \|A\|_{L^{\infty}}\|\rho_0^{1/2}\Theta_{tt}\|_{L^2} \Big) \|v\|_{H^3}\\
&\qquad
+ \Big(\|\partial_t A\|_{L^{\infty}}\|\Theta\|_{H^2}
+\|A\|_{L^{\infty}}\|\Theta_t\|_{H_0^1}\Big)\|v_t\|_{H^1}
+\|A\|_{L^{\infty}}\|\Theta\|_{H^2}\|v_{tt}\|_{H^1}\bigg]^2+\epsilon\|\Theta_{tt}\|_{H_0^1}^2\ dt'\\
&\lesssim  \frac{1}{\epsilon}tP(\sup\limits_{t'\in[0,t]}F)+\frac{1}{\epsilon}P(M_0)+\int_{0}^{t}\frac{1}{\epsilon}\Big( t'P(\sup\limits_{[0,t']}F)+P(M_0)  \Big) \|v_{tt}\|_{H^1}^2+\epsilon \|\Theta_{tt}\|_{H_0^1}^2\ dt'\\
&\lesssim  \frac{1}{\epsilon}tP(\sup\limits_{t'\in[0,t]}F)+\frac{1}{\epsilon}P(M_0)+\frac{1}{\epsilon}\left(\int_{0}^{t}\|v_{tt}\|_{H^1}^2 \ dt'\right)^2+\epsilon \int_{0}^{t}\|\Theta_{tt}\|_{H_0^1}^2\ dt',\\
|\mathrm{II}|&\lesssim \iint \Big[ |\partial_t^2 (aA)| |Dv|^2 + |\partial_t (aA)| |Dv_t| |Dv|+ |(aA)|\Big( |Dv_{tt}||Dv|+|Dv_t|^2\Big)    \Big] |\Theta_{tt}|\\
&\lesssim \int_{0}^{t} \Big[ \|\partial_t^2 (aA)\|_{L^2} \|Dv\|_{L^{\infty}}^2 + \|\partial_t (aA)\|_{L^{\infty}} \|Dv_t\|_{L^2} \|Dv\|_{L^{\infty}}+ \|aA\|_{L^{\infty}}\Big(\|Dv_{tt}\|_{L^2}\|Dv\|_{L^4}\\
&\qquad
+\|Dv_t\|_{L^4}\|Dv_t\|_{L^2}\Big)    \Big] \|\Theta_{tt}\|_{L^4} \ dt'\\
&\lesssim \int_{0}^{t} \Big[ \Big( t'P(\sup\limits_{[0,t']}F)+P(M_0)  \Big)^{1/2}
  \Big(\|v\|_{H^3}^2+ \|v_t\|_{H^1} \|v\|_{H^3}+\|v_{tt}\|_{H^1}\|v\|_{H^2}
  \\
&\qquad+\|v_t\|_{H^2}\|v_t\|_{H^1}\Big)  \Big] \|\Theta_{tt}\|_{H_0^1} \ dt'\\
&\lesssim  \frac{1}{\epsilon}tP(\sup\limits_{t'\in[0,t]}F)+\frac{1}{\epsilon}P(M_0)
+\int_{0}^{t}\frac{1}{\epsilon}\Big( t'P(\sup\limits_{[0,t']}F)+P(M_0) \Big) (\|v_{tt}\|_{H^1}^2+\|v_t\|_{H^2}^2)+\epsilon \|\Theta_{tt}\|_{H_0^1}^2\ dt'\\
&\lesssim \frac{1}{\epsilon}tP(\sup\limits_{t'\in[0,t]}F)+\frac{1}{\epsilon}P(M_0)
+\frac{1}{\epsilon}\left(\int_{0}^{t}\|v_{tt}\|_{H^1}^2+\|v_t\|_{H^2}^2 \ dt'\right)^2
+\epsilon \int_{0}^{t}\|\Theta_{tt}\|_{H_0^1}^2\ dt'.
\end{align*}
Note that here in the estimates of $\mathrm{I}$ and $\mathrm{II}$, \eqref{direct} is applied to substituting $\|v_t\|_{H^1}$ or $\|v\|_{H^2}$.

Note that $\Theta_{tt}=0$ on $\Gamma$.
After combining the estimates of $\mathrm{I}, \mathrm{II}, \mathrm{III}$ 
and choosing small enough $\epsilon$ (independent of $v$ or $\Theta$), 
we have that
\begin{equation}\label{thetatt}
\|\rho_0^{1/2}\Theta_{tt}(\cdot,t)\|^2_{L^2}+\int_{0}^{t}\|\Theta_{tt}\|_{H_0^1}^2\ dt'
\lesssim tP(\sup\limits_{t'\in[0,t]}F)+P(M_0)+P\left(\int_{0}^{t}\|v_{tt}\|_{H^1}^2+\|v_t\|_{H^2}^2\ dt'\right).
\end{equation}

\vspace{0.3cm}
\textsl{\textbf{Estimates of $\int_{0}^{t}\|\bpartial^2 v_{t}\|_{H^1}^2\ dt'$:}}~~
Here we establish the estimates by the formular $\int_{0}^{t}\int_{\Omega}\bpartial^2\partial_t \Big(\eqref{1}_1\Big)\cdot \bpartial^2 v_{t}\ dx dt'$. It follows that
$$
\underbrace{\iint \bpartial^2\partial_t(\rho_0 v_t^i) \bpartial^2 v_t^i}_{\mathrm{I}}
-\underbrace{\iint \bpartial^2 \partial_t (a^r_i\frac{R\rho_0\Theta}{J}) \bpartial^2 v^i_{t,r}}_{\mathrm{II}}
+\underbrace{\iint \bpartial^2\partial_t (a^r_j\mathbb{S}^{ij}_{\eta}[v]) \bpartial^2v^i_{t,r}}_{\mathrm{III}}=0.
$$
Similarly as before, we rewrite $\mathrm{III}$ as
$$
\mathrm{III}=\iint a^r_j\Big[\mu (A^{k}_j \bpartial^2 v^i_{t,k}+A^k_i \bpartial^2 v^j_{t,k})
+\lambda (A_l^k \bpartial^2 v^l_{t,k})\delta^i_j\Big]
\bpartial^2 v^i_{t,r}+\mathrm{III}',
$$
where $\mathrm{III}'$ denotes all the remaining terms. Then it is straightforward that 
\begin{align*}
|\mathrm{III}'|&\lesssim  \int_{0}^{t}  \Big[\|\bpartial^2\partial_t(aA)\|_{L^2}\|Dv\|_{L^{\infty}}
  +\|\bpartial^2 (aA)\|_{L^4} \|Dv_t\|_{L^4}
  +\|\bpartial\partial_t (aA) \|_{L^4} \|\bpartial Dv\|_{L^4}\\
  &\qquad
  +\|\bpartial (aA)\|_{L^{\infty}}\|\bpartial Dv_t\|_{L^2}
  +\|\partial_t(aA)\|_{L^{\infty}} \|\bpartial^2 Dv\|_{L^2}
  \Big]\|\bpartial^2 Dv_{t}\|_{L^2}\ dt'\\
&\lesssim  \int_{0}^{t}  \Big[\|\bpartial^2\partial_t(aA)\|_{L^2} \|v\|_{H^3}
  +\|\bpartial^2 (aA)\|_{L^4} \|v_t\|_{H^2}
  +\|\bpartial\partial_t (aA) \|_{L^4} \|v\|_{H^3}\\
  &\qquad
  +\|\bpartial (aA)\|_{L^{\infty}} \|v_t\|_{H^2}
  +\|\partial_t(aA)\|_{L^{\infty}} \|v\|_{H^3}
  \Big]\|\bpartial^2 Dv_{t}\|_{L^2}\ dt'\\
&\lesssim\frac{1}{\epsilon} tP(\sup\limits_{t'\in[0,t]} F)+\frac{1}{\epsilon}P(M_0)
  +\int_{0}^{t} \frac{1}{\epsilon}(t'P(\sup\limits_{[0,t']}F)+P(M_0)) \|v_t\|_{H^2}^2 
  +\epsilon\|\bpartial^2 Dv_{t}\|_{L^2}^2\ dt'\\
&\lesssim \frac{1}{\epsilon}tP(\sup\limits_{t'\in[0,t]}F)+\frac{1}{\epsilon}P(M_0)
  +\frac{1}{\epsilon}\left(  \int_{0}^{t}\|v_t\|_{H^2}^2 \ dt'\right)^2
  +\epsilon\int_{0}^{t}\|\bpartial^2 Dv_{t}\|_{L^2}^2\ dt'.
\end{align*}
After applying the Korn's inequality we have
\begin{align*}
\mathrm{III}-\mathrm{III}'& 
\geq  \frac{1}{C}
\iint\sum\limits_{i,j}\Big(A^k_j  \bpartial^2 v_{t,k}^i+A^k_i  \bpartial^2 v_{t,k}^j\Big)^2 \\
&\geq \frac{1}{C}\iint  \sum\limits_{i,j}
    \Big(\delta^k_j  \bpartial^2 v_{t,k}^i+\delta^k_i \bpartial^2 v_{t,k}^j\Big)^2
    -C|A-\delta| |\bpartial^2 Dv_{t}|^2 \ dxdt'\\
&\geq\frac{1}{C}\iint\sum\limits_{i,j}
    ( \bpartial^2 v_{t,j}^i+ \bpartial^2 v_{t,i}^j)^2 \ dxdt'
    -C \int_{0}^{t}  (t'\sup\limits_{[0,t']}F^{1/2})  \| \bpartial^2 Dv_{t}\|_{L^2}^2\ dt'\\
& \geq \frac{1}{C}\int_{0}^{t} \|\bpartial^2 v_{t}\|_{H^1}^2 \ dt'
    -C\int_{0}^{t}\|\rho_0^{1/2} \bpartial^2 v_{t}\|_{L^2}^2\ dt'-CtP(\sup\limits_{t'\in[0,t]}F)\\
& \geq \frac{1}{C}\int_{0}^{t} \|\bpartial^2 v_{t}\|_{H^1}^2 \ dt'
    -C\int_{0}^{t} \|  v_t\|_{H^2}^2\ dt'-CtP(\sup\limits_{t'\in[0,t]}F).
\end{align*}
After integrating by part with respect to the tangential variables, one has 
\begin{align*}
|\mathrm{I}|&=\left| \iint \bpartial \partial_t(\rho_0 v_t) \cdot \bpartial^3 v_t \right|\\
&\lesssim \iint \Big( |\bpartial \rho_0| |v_{tt}|+|\rho_0| |\bpartial v_{tt}|  \Big)
          |\bpartial^3 v_t|\\
     &\lesssim \int_{0}^{t} \Big( \|\bpartial \rho_0\|_{L^4} \|v_{tt}\|_{L^4}+\|\rho_0\|_{L^{\infty}} \|\bpartial v_{tt}\|_{L^2}  \Big)
           \|\bpartial^3 v_t\|_{L^2} \ dt'\\
     &\lesssim \int_{0}^{t} ( \|\bpartial \rho_0\|_{H^1}+ \|\rho_0\|_{L^{\infty}})\|v_{tt}\|_{H^1} 
                \|\bpartial^3 v_t\|_{L^2} \ dt'\\
     &\lesssim  \frac{1}{\epsilon}\int_{0}^{t}  \|v_{tt}\|_{H^1}^2 \ dt'
           +  \epsilon \int_{0}^{t}  \|\bpartial^2 v_t\|_{H^1} ^2\ dt'. 
\end{align*}
The rest is to estimate $\mathrm{II}$, which can be done as follows.
\begin{align*}
|\mathrm{II}|&\lesssim \int_{0}^{t} \Big[\|A\|_{L^{\infty}}\|\bpartial^2 (\rho_0\Theta_t)\|_{L^2}
                 +\|\bpartial A\|_{L^{\infty}} \|\bpartial (\rho_0 \Theta_t)\|_{L^2}
                 +\|\bpartial^2 A\|_{L^4} \|\rho_0\Theta_t\|_{L^4} \ dt'\\
             &\qquad  
                 +\|\partial_t A\|_{L^{\infty}} \|\bpartial^2 (\rho_0\Theta)\|_{L^2}
                 +\|\bpartial\partial_t A\|_{L^4} \|\bpartial (\rho_0\Theta)\|_{L^4}
                 +\|\bpartial^2 \partial_t A\|_{L^2} \|\rho_0 \Theta\|_{L^{\infty}}
                 \Big]\|\bpartial^2 Dv_t\|_{L^2}\\
        &\lesssim \int_{0}^{t} \Big[ \Big( t'P(\sup\limits_{[0,t']}F)+P(M_0)  \Big)^{1/2}
        (\|\bpartial\rho_0\|_{H^1}+\|\rho_0\|_{L^{\infty}}) 
          \Big(\|\Theta_t\|_{H^2}+\|\Theta_t\|_{H_0^1}+\|\Theta\|_{H^2}\Big)
                \Big]\\
            &\qquad \cdot\|\bpartial^2 Dv_t\|_{L^2}\ dt'\\
     &\lesssim \frac{1}{\epsilon}tP(\sup\limits_{t'\in[0,t]}F)+\frac{1}{\epsilon}P(M_0)
         +\frac{1}{\epsilon}\int_{0}^{t}\|\Theta_t\|_{H^2}^2\ dt' 
         +\epsilon\int_{0}^{t} \|\bpartial^2 v_t\|_{H^1}^2\ dt'.
\end{align*}
Here we have used the following estimates,
\begin{align*}
\|\bpartial^2 (\rho_0\Theta_t)\|_{L^2}&\lesssim\|\bpartial^2\rho_0 \|_{L^2}\|\Theta_t\|_{L^{\infty}}
    +\|\bpartial\rho_0 \|_{L^4}\|\bpartial\Theta_t\|_{L^4}
    +\|\rho_0 \|_{L^{\infty}}\|\bpartial^2\Theta_t\|_{L^2}\\
                                      &\lesssim(\|\bpartial\rho_0\|_{H^1}+\|\rho_0\|_{L^{\infty}})\|\Theta_t\|_{H^2},\\
\|\bpartial (\rho_0 \Theta_t)\|_{L^2} &\lesssim\|\bpartial \rho_0\|_{L^4}\|\Theta_t\|_{L^4}
    +\|\rho_0\|_{L^{\infty}}\|\bpartial \Theta_t\|_{L^2}
   \lesssim(\|\bpartial\rho_0\|_{H^1}+\|\rho_0\|_{L^{\infty}})\|\Theta_t\|_{H_0^1},\\
\|\rho_0\Theta_t\|_{L^4}              & \lesssim\|\rho_0\|_{L^{\infty}}\| \Theta_t\|_{L^4}
    \lesssim (\|\bpartial\rho_0\|_{H^1}+\|\rho_0\|_{L^{\infty}})\|\Theta_t\|_{H_0^1},\\
\|\bpartial (\rho_0 \Theta)\|_{L^4}   &\lesssim\|\bpartial \rho_0\|_{L^4}\|\Theta\|_{L^{\infty}}
    +\|\rho_0\|_{L^{\infty}}\|\bpartial \Theta\|_{L^4}
       \lesssim(\|\bpartial\rho_0\|_{H^1}+\|\rho_0\|_{L^{\infty}})\|\Theta\|_{H^2},\\
\|\rho_0\Theta\|_{L^{\infty}}         &\lesssim \|\rho_0\|_{L^{\infty}}\| \Theta\|_{L^{\infty}}
       \lesssim (\|\bpartial\rho_0\|_{H^1}+\|\rho_0\|_{L^{\infty}})\|\Theta\|_{H^2}.
\end{align*}

After combining the estimates of $\mathrm{I}, \mathrm{II}, \mathrm{III}$ and choosing small enough $\epsilon$ (independent of $v$ or $\Theta$), we have that
\begin{equation}\label{vtxx}
\int_{0}^{t}\|\bpartial^2 v_t\|_{H^1}^2\ dt'
\lesssim tP(\sup\limits_{t'\in[0,t]}F)+P(M_0)+P\left(  \int_{0}^{t}\|v_{tt}\|_{H^1}^2+\|v_t\|_{H^2}^2 +\|\Theta_{t}\|_{H^2}^2\ dt'\right).
\end{equation}

\vspace{0.3cm}
\textsl{\textbf{Estimates of $\int_{0}^{t}\|\bpartial^2 \Theta_{t}\|_{H_0^1}^2\ dt'$:}}~~Here we evaluate the integral 
$\int_{0}^{t}\int_{\Omega}\bpartial^2 \partial_t \Big(\eqref{1}_2\Big)\cdot  \bpartial^2 \Theta_t\ dx dt'$. It follows that
\begin{align*}
c_v\underbrace{\iint \bpartial ^2 \Big(\rho_0\Theta_{tt}\Big)\bpartial^2 \Theta_t}_{\mathrm{I}}
+\underbrace{\iint  \bpartial^2 \partial_t \Big(\frac{R\rho_0\Theta}{J} a^r_i v^i_{,r}\Big) \bpartial^2 \Theta_t}_{\mathrm{II}}
-\underbrace{\iint \bpartial^2 \partial_t (\mathbb{S}^{ij}_{\eta}[v] a^r_j v^i_{,r})  \bpartial^2 \Theta_t}_{\mathrm{III}}&\\
+\underbrace{\kappa\iint \bpartial^2 \partial_t(a^r_i A^k_j \delta^{ij}\Theta_{,k}) \bpartial^2 \Theta_{t,r}}_{\mathrm{IV}}=0.
\qquad\qquad\qquad\qquad&
\end{align*}
Again, we rewrite $\mathrm{IV}$ as 
$$
\mathrm{IV}=\kappa \iint a^r_i A^k_j \delta^{ij} \bpartial^2 \Theta_{t,k}
\bpartial^2 \Theta_{t,r}+\mathrm{IV}',
$$
where $\mathrm{IV}'$ denotes all the remaining terms.
Similarly as before, we list the estimates below,
\begin{align*}
|\mathrm{IV}'|
&\lesssim \frac{1}{\epsilon}tP(\sup\limits_{t'\in[0,t]}F)+\frac{1}{\epsilon}P(M_0)
  +\frac{1}{\epsilon}\left(  \int_{0}^{t}\|\Theta_t\|_{H^2}^2 \ dt'\right)^2
  +\epsilon\int_{0}^{t}\|\bpartial^2 D\Theta_{t}\|_{L^2}^2\ dt',\\
\mathrm{IV}-\mathrm{IV}'&\geq \frac{1}{C}\iint \delta_i^k \delta^{ir} \bpartial^2 \Theta_{t,k}\bpartial^2 \Theta_{t,r}
-C \int_{0}^{t}  \|A-\delta\|_{L^{\infty}}^2 \|\bpartial^2 D\Theta_{t}\|_{L^2}^2\\
&\geq \frac{1}{C} \int_{0}^{t} \|\bpartial^2 D\Theta_{t}\|_{L^2}^2 \ dt'-CtP(\sup\limits_{t'\in[0,t]}F),\\
|\mathrm{I}| &\lesssim  \frac{1}{\epsilon}\int_{0}^{t}  \|\Theta_{tt}\|_{H_0^1}^2 \ dt'
           +  \epsilon \int_{0}^{t}  \|\bpartial^2 \Theta_t\|_{H_0^1} ^2\ dt', 
\end{align*}
\begin{align*}
|\mathrm{II}|
&\lesssim \int_{0}^{t} \Big( \|\bpartial^2\partial_t (A\rho_0 \Theta)\|_{L^2} \|Dv\|_{L^{\infty}}
    +\|\bpartial\partial_t (A\rho_0 \Theta)\|_{L^2} \|\bpartial Dv\|_{L^4}
    +\|\partial_t (A\rho_0 \Theta)\|_{L^4}  \|\bpartial^2 Dv\|_{L^2}\\
   &\qquad+\|\bpartial^2 (A\rho_0 \Theta)\|_{L^2} \|Dv_t\|_{L^4}
    +\|\bpartial (A\rho_0 \Theta)\|_{L^4} \|\bpartial Dv_t\|_{L^2}
    +\| (A\rho_0 \Theta)\|_{L^{\infty}}  \|\bpartial^2 Dv_t\|_{L^2}\Big)\|\bpartial^2 \Theta_t\|_{L^4}\\
&\lesssim \int_{0}^{t} \bigg[ \Big( (t'P(\sup\limits_{[0,t']}F)+P(M_0))^{1/2} \|\Theta_t\|_{H^2}+t'P(\sup\limits_{[0,t']}F)+P(M_0)  \Big)\|v\|_{H^3}\\
   & \qquad+ \Big( t'P(\sup\limits_{[0,t']}F)+P(M_0)  \Big)
    \Big(\|v\|_{H^3}+\|v_t\|_{H^2}+\|\bpartial^2 v_t\|_{H^1}\Big)\bigg] \|\bpartial^2 \Theta \|_{H_0^1}\\
&\lesssim \frac{1}{\epsilon}tP(\sup\limits_{t'\in[0,t]}F)+\frac{1}{\epsilon}P(M_0)+\frac{1}{\epsilon}\left(\int_{0}^{t}\|\Theta_t\|_{H^2}^2+\|v_t\|_{H^2}^2 +\|\bpartial^2 v_t\|_{H^1}^2\ dt'\right)^2+\epsilon \int_{0}^{t}\|\bpartial^2\Theta_{t}\|_{H_0^1}^2\ dt',
\end{align*}
\begin{align*}
|\mathrm{III}|
&\lesssim \int_{0}^{t} \Big[ \|\bpartial^2 \partial_t(aA)\|_{L^2} \|Dv\|_{L^{\infty}}^2 
            + \|\bpartial\partial_t (aA)\|_{L^4} \|\bpartial Dv\|_{L^4} \|Dv\|_{L^{\infty}}
            + \|\partial_t (aA)\|_{L^{\infty}}\\
              &\qquad\Big(\|\bpartial^2 Dv\|_{L^2}\|Dv\|_{L^{\infty}}+\|\bpartial Dv\|_{L^4}^2\Big)
            +\|\bpartial^2 (aA)\|_{L^4} \|Dv_t\|_{L^2} \|Dv\|_{L^{\infty}}
            + \|\bpartial (aA)\|_{L^{\infty}} \\  
              &\qquad\Big(\|\bpartial Dv_t\|_{L^2} \|Dv\|_{L^{\infty}}+\|Dv_t\|_{L^2}\|\bpartial Dv\|_{L^4} \Big)
           + \| aA\|_{L^{\infty}}\Big(\|\bpartial^2 Dv_t\|_{L^2}\|Dv\|_{L^{\infty}}\\
              &\qquad+\|\bpartial Dv_t\|_{L^2}\|\bpartial Dv\|_{L^4}+\| Dv_t\|_{L^4}\|\bpartial^2 Dv\|_{L^2} \Big)  \Big] \|\bpartial^2\Theta_{t}\|_{L^4} \ dt'\\
&\lesssim \int_{0}^{t} \Big[ \Big( t'P(\sup\limits_{[0,t']}F)+P(M_0)  \Big)^{1/2}
          \Big(\|v\|_{H^3}+ \|v_t\|_{H^2}+\|\bpartial^2 v_t\|_{H^1}\Big) \|v\|_{H^3}
           \Big] \|\bpartial^2 \Theta_{t}\|_{H_0^1} \ dt'\\
&\lesssim \frac{1}{\epsilon}tP(\sup\limits_{t'\in[0,t]}F)+\frac{1}{\epsilon}P(M_0)+\frac{1}{\epsilon}\left(\int_{0}^{t}\|v_t\|_{H^2}^2 +\|\bpartial^2 v_t\|_{H^1}^2\ dt'\right)^2+\epsilon \int_{0}^{t}\|\bpartial^2\Theta_{t}\|_{H_0^1}^2\ dt'.
\end{align*}
Note that here in the estimates of $\mathrm{II}$ and $\mathrm{III}$, \eqref{direct} is applied to substituting $\|v_t\|_{H^1}$ and $\|v\|_{H^3}$.

Note that $\bpartial^2 \Theta_t=0$ on $\Gamma$.
After combining the estimates of $\mathrm{I}, \mathrm{II}, \mathrm{III}, \mathrm{IV}$ 
and choosing small enough $\epsilon$ (independent of $v$ or $\Theta$), 
we have that
\begin{equation}\label{thetatxx}
\begin{aligned}
\int_{0}^{t}\|\bpartial^2 \Theta_t\|_{H_0^1}^2\ dt'
&\lesssim tP(\sup\limits_{t'\in[0,t]}F)+P(M_0)+P \Big(\int_{0}^{t}\|\Theta_{tt}\|_{H_0^1}^2\\
&\qquad \|\bpartial^2 v\|_{H^1}^2+\|v_t\|_{H^2}^2+\|\Theta_t\|_{H^2}^2 \ dt'\Big).
\end{aligned}
\end{equation}

\vspace{0.3cm}

\subsection{Lower order energy estimates}~
\vspace{0.3cm}

\textsl{\textbf{Estimates of $\int_{0}^{t}\|\bpartial v_{t}\|_{H^1}^2\ dt'$:}}~~
The calculation
$\int_{0}^{t}\int_{\Omega}\bpartial\partial_t \Big(\eqref{1}_1\Big)\cdot \bpartial v_{t}\ dx dt'$ yields
$$
\underbrace{\iint \bpartial \partial_t (\rho_0 v_t^i) \cdot \bpartial v_t^i}_{\mathrm{I}} 
-\underbrace{\iint \bpartial \partial_t (a^r_i\frac{R\rho_0\Theta}{J})\bpartial v^i_{t,r}}_{\mathrm{II}}
+\underbrace{\iint \bpartial\partial_t (a^r_j\mathbb{S}^{ij}_{\eta}[v])\bpartial v^i_{t,r}}_{\mathrm{III}}=0.
$$
Here
$$
\mathrm{III}=\iint a^r_j\Big[\mu (A^{k}_j \bpartial v^i_{t,k}+A^k_i \bpartial v^j_{t,k})+\lambda A_l^k \bpartial v^l_{t,k}\delta^i_j\Big]\bpartial v^i_{t,r}+\mathrm{III}',
$$
where $\mathrm{III}'$ denotes all the remaining terms. Then it follows that
\begin{align*}
|\mathrm{III}'|
&\lesssim \int_{0}^{t}\Big(\|\bpartial\partial_t (aA)\|_{L^4}\|Dv\|_{L^{\infty}}
    +\|\bpartial (aA)\|_{L^{\infty}}\|Dv_t\|_{L^2}
    +\|\partial_t (aA)\|_{L^{\infty}}\|\bpartial Dv\|_{L^4}\Big)
    \|\bpartial Dv_t\|_{L^2}\ dt'\\
&\lesssim \int_{0}^{t} \Big( t'P(\sup\limits_{[0,t']}F)+P(M_0)  \Big)^{1/2} \Big(\|v\|_{H^3}+\|v_t\|_{H^1}\Big)\|\bpartial Dv_t\|_{L^2} \ dt'\\
&\lesssim \frac{1}{\epsilon}tP(\sup\limits_{t'\in[0,t]}F)+\frac{1}{\epsilon}P( M_0 )+\epsilon \int_{0}^{t}\|\bpartial Dv_t\|_{L^2}^2\ dt',
\end{align*}
\begin{align*}
\mathrm{III}-\mathrm{III}'&\geq\frac{1}{C}\iint\sum\limits_{i,j}(\bpartial v^i_{t,j}+\bpartial v^j_{t,i})^2 -CtP(\sup\limits_{t'\in[0,t]}F)\\
&\geq \frac{1}{C}\int_{0}^{t}\|\bpartial v_t\|_{H^1}^2\ dt'-C\iint \rho_0 |\bpartial v_t|^2-CtP(\sup\limits_{t'\in[0,t]}F)\\
&\geq \frac{1}{C}\int_{0}^{t}\|\bpartial v_t\|_{H^1}^2\ dt'-CtP(\sup\limits_{t'\in[0,t]}F),\\
|\mathrm{I}|&=\Big|-\iint \rho_0 v^i_{tt} \bpartial^2v^i_t \Big|
\lesssim \int_{0}^{t} \|\rho_0\|_{L^{\infty}}^{1/2} \|\rho_0^{1/2}v_{tt}\|_{L^2} \|\bpartial^2 v_t\|_{L^2} \ dt'\\
&\lesssim \frac{1}{\epsilon}tP(\sup\limits_{t'\in[0,t]}F)+\epsilon \int_{0}^{t}\|\bpartial^2 v_t\|_{L^2}^2\ dt'. 
\end{align*}
Similarly as before,
\begin{align*}
|\mathrm{II}|
&\lesssim\int_{0}^{t} \| \bpartial \partial_t (A\rho_0\Theta)\|_{L^2} \|\bpartial Dv_t\|_{L^2} \ dt' \\
&\lesssim\int_{0}^{t} \Big( t'P(\sup\limits_{[0,t']}F)+P(M_0)  \Big)^{1/2} \Big(\|\Theta_t\|_{H_0^1}+\|\Theta\|_{H^2}\Big)\|\bpartial Dv_t\|_{L^2} \ dt' \\
&\lesssim \frac{1}{\epsilon}tP(\sup\limits_{t'\in[0,t]}F)+\frac{1}{\epsilon} P(M_0) +\epsilon \int_{0}^{t}\|\bpartial Dv_t\|_{L^2}^2\ dt',
\end{align*}

After combining the estimates of $\mathrm{I}, \mathrm{II}$ and $\mathrm{III}$ and choosing small enough $\epsilon$ (independent of $v$), we have that 
\begin{equation}\label{vtx}
\int_{0}^{t}\|\bpartial v_t\|_{H^1}^2\ dt'\lesssim tP(\sup\limits_{t'\in[0,t]}F)+P( M_0).
\end{equation}

\vspace{0.3cm}

\textsl{\textbf{Estimates of $\int_{0}^{t}\|\bpartial \Theta_{t}\|_{H_0^1}^2\ dt'$:}}~~
We follow the calculation 
$\int_{0}^{t}\int_{\Omega}\bpartial\partial_t \Big(\eqref{1}_2\Big)\cdot  \bpartial \Theta_t\ dx dt'$,
\begin{align*}
c_v\underbrace{\iint \bpartial  \Big(\rho_0\Theta_{tt}\Big)\bpartial \Theta_t}_{\mathrm{I}}
+\underbrace{\iint  \bpartial \partial_t \Big(\frac{R\rho_0\Theta}{J} a^r_i v^i_{,r}\Big) \bpartial \Theta_t}_{\mathrm{II}}
-\underbrace{\iint \bpartial \partial_t (\mathbb{S}^{ij}_{\eta}[v] a^r_j v^i_{,r})  \bpartial \Theta_t}_{\mathrm{III}}&\\
+\underbrace{\kappa\iint \bpartial \partial_t(a^r_i A^k_j \delta^{ij}\Theta_{,k}) \bpartial \Theta_{t,r}}_{\mathrm{IV}}=0,
\qquad\qquad\qquad\qquad&
\end{align*}
with
$$
\mathrm{IV}=\kappa \iint a^r_i A^k_j \delta^{ij} \bpartial \Theta_{t,k}
\bpartial \Theta_{t,r}+\mathrm{IV}',
$$
where $\mathrm{IV}'$ denotes all the remaining terms.
Similarly as before,
\begin{align*}
|\mathrm{IV}'|
&\lesssim \frac{1}{\epsilon}tP(\sup\limits_{t'\in[0,t]}F)+\frac{1}{\epsilon} P(M_0) +\epsilon \int_{0}^{t}\|\bpartial D\Theta_t\|_{L^2}^2\ dt',\\
\mathrm{IV}
&\geq \frac{1}{C} \int_{0}^{t} \|\bpartial D\Theta_{t}\|_{L^2}^2 \ dt'-CtP(\sup\limits_{t'\in[0,t]}F),\\
|\mathrm{I}| &\lesssim \frac{1}{\epsilon}tP(\sup\limits_{t'\in[0,t]}F)
           + \epsilon \int_{0}^{t}  \|\bpartial^2 \Theta_t\|_{L^2} ^2\ dt.  
\end{align*}
Moreover,
\begin{align*}
|\mathrm{II}|
&\lesssim \int_{0}^{t} \Big( \|\bpartial\partial_t (A\rho_0 \Theta)\|_{L^2} \| Dv\|_{L^{\infty}}
    +\|\partial_t (A\rho_0 \Theta)\|_{L^4}  \|\bpartial Dv\|_{L^4}\\
   &\qquad
    +\|\bpartial (A\rho_0 \Theta)\|_{L^4} \| Dv_t\|_{L^2}
    +\| (A\rho_0 \Theta)\|_{L^{\infty}}  \|\bpartial Dv_t\|_{L^2}\Big)\|\bpartial \Theta_t\|_{L^4}\\
&\lesssim \int_{0}^{t}  \Big( t'P(\sup\limits_{[0,t']}F)+P(M_0)  \Big)^{1/2}\Big(\|\Theta_t\|_{H_0^1}+\|\Theta\|_{H^2}\Big)\Big(\|v\|_{H^3}+\|v_t\|_{H^1}+\|\bpartial v_t\|_{H^1}\Big) \|\bpartial \Theta \|_{H_0^1}\\
&\lesssim \frac{1}{\epsilon}tP(\sup\limits_{t'\in[0,t]}F)+\frac{1}{\epsilon}P(M_0)+\frac{1}{\epsilon}\left(\int_{0}^{t} \|\bpartial v_t\|_{H^1}^2\ dt'\right)^2+\epsilon \int_{0}^{t}\|\bpartial\Theta_{t}\|_{H_0^1}^2\ dt',\\
|\mathrm{III}|
&\lesssim \int_{0}^{t} \Big[ \|\bpartial\partial_t (aA)\|_{L^4}  \|Dv\|_{L^{\infty}}^2
            + \|\partial_t (aA)\|_{L^{\infty}}\|\bpartial Dv\|_{L^4} \|Dv\|_{L^{\infty}}
           +\|\bpartial (aA)\|_{L^{\infty}} \|Dv_t\|_{L^2} \|Dv\|_{L^{\infty}}\\
            &+ \| (aA)\|_{L^{\infty}} \Big(\|\bpartial Dv_t\|_{L^2} \|Dv\|_{L^{\infty}}+\|Dv_t\|_{L^2} \|\bpartial Dv\|_{L^4} \Big) \Big] \|\bpartial\Theta_{t}\|_{L^4} \ dt'\\
&\lesssim \int_{0}^{t} \Big[ \Big( t'P(\sup\limits_{[0,t']}F)+P(M_0)  \Big)^{1/2}
          \Big(\|v\|_{H^3}+ \|v_t\|_{H^1}+\|\bpartial v_t\|_{H^1}\Big) \|v\|_{H^3}
           \Big] \|\bpartial \Theta_{t}\|_{H_0^1} \ dt'\\
&\lesssim \frac{1}{\epsilon}tP(\sup\limits_{t'\in[0,t]}F)+\frac{1}{\epsilon}P(M_0)+\frac{1}{\epsilon}\left(\int_{0}^{t} \|\bpartial v_t\|_{H^1}^2\ dt'\right)^2+\epsilon \int_{0}^{t}\|\bpartial\Theta_{t}\|_{H_0^1}^2\ dt'.
\end{align*}
Note that here in the estimates of $\mathrm{II}$ and $\mathrm{III}$, \eqref{direct} is applied to substituting $\|v_t\|_{H^1}$ or $\|v\|_{H^2}$ which are multiplied by $\|\bpartial v_t\|_{H^1}$.

Note that $\bpartial\Theta_t=0$ on $\Gamma$.
After combining the estimates of $\mathrm{I}, \mathrm{II}, \mathrm{III}, \mathrm{IV}$ 
and choosing small enough $\epsilon$ (independent of $v$ or $\Theta$), 
we have that
\begin{equation}\label{thetatx}
\begin{aligned}
\int_{0}^{t}\|\bpartial \Theta_t\|_{H_0^1}^2\ dt'
&\lesssim tP(\sup\limits_{t'\in[0,t]}F)+P(M_0)+P \Big(\int_{0}^{t} \|\bpartial v_t\|_{H^1}^2 \ dt'\Big).
\end{aligned}
\end{equation}

\vspace{0.3cm}

\textsl{\textbf{Estimates of $\int_{0}^{t}\|\bpartial^3 v\|_{H^1}^2 \ dt'$:}}~~
Here we calculate
$\int_{0}^{t}\int_{\Omega}\bpartial^3 \Big(\eqref{1}_1\Big)\cdot \bpartial^3 v \ dx dt'$
$$
\underbrace{\iint \bpartial^3 (\rho_0 v^i_t) \bpartial^3 v^i}_{\mathrm{I}} 
-\underbrace{\iint \bpartial^3 (a^r_i\frac{R\rho_0\Theta}{J})\bpartial^3 v^i_{,r}}_{\mathrm{II}}
+\underbrace{\iint \bpartial^3 (a^r_j\mathbb{S}^{ij}_{\eta}[v])\bpartial^3 v^i_{,r}}_{\mathrm{III}}=0,
$$
where
$$
\mathrm{III}=\iint a^r_j\Big[\mu (A^{k}_j \bpartial^3 v^i_{,k}+A^k_i \bpartial^3 v^j_{,k})+\lambda A_l^k \bpartial^3 v^l_{,k}\delta^i_j\Big]\bpartial^3 v^i_{,r}+\mathrm{III}',
$$
where $\mathrm{I}_3'$ denotes all the remaining terms.
\begin{align*}
|\mathrm{III}'|
&\lesssim  \int_{0}^{t} \bigg[ \|\bpartial^3 (aA)\|_{L^2}\|Dv\|_{L^{\infty}}
     + \|\bpartial^2 (aA)\|_{L^4}\|\bpartial Dv\|_{L^4}
     \\
     &\qquad+ \|\bpartial (aA)\|_{L^{\infty}}\|\bpartial^2 Dv\|_{L^2}
       \bigg] \|\bpartial^3 Dv\|_{L^2}\ dt'\\ 
&\lesssim \int_{0}^{t} \Big( t'P(\sup\limits_{[0,t']}F)+P(M_0)  \Big)^{1/2}\|v\|_{H^3} \|\bpartial^3 Dv\|_{L^2}\ dt'\\
&\lesssim  \frac{1}{\epsilon}tP(\sup\limits_{t'\in[0,t]}F)+\frac{1}{\epsilon}P(M_0)+\epsilon \int_{0}^{t}\|\bpartial^3 Dv\|_{L^2}^2\ dt',\\
\mathrm{III}-\mathrm{III}'
&\geq\frac{1}{C}\iint\sum\limits_{i,j}(\bpartial^3 v^i_{,j}+\bpartial^3 v^j_{,i})^2-CtP(\sup\limits_{t'\in[0,t]}F)\\
&\geq \frac{1}{C}\int_{0}^{t}\|\bpartial^3 v\|_{H^1}^2\ dt'-C\iint \rho_0 |\bpartial^3 v|^2-CtP(\sup\limits_{t'\in[0,t]}F)\\
&\geq \frac{1}{C}\int_{0}^{t}\|\bpartial^3 v\|_{H^1}^2\ dt'-CtP(\sup\limits_{t'\in[0,t]}F),\\
|\mathrm{I}|&=\left| -\iint \bpartial^2(\rho_0 v^i_t)\bpartial^4 v^i   \right|\\
&\lesssim  \int_{0}^{t} \Big(\|\bpartial^2 \rho_0\|_{L^2}\|v_t\|_{L^{\infty}}+\|\bpartial \rho_0\|_{L^4}\|\bpartial v_t\|_{L^4}+\|\rho_0\|_{L^{\infty}}\|\bpartial^2 v_t\|_{L^2}\Big)\|\bpartial^4 v\|_{L^2}\ dt'\\
&\lesssim  \int_{0}^{t} (\|\bpartial\rho_0\|_{H^1}+\|\rho_0\|_{L^\infty})\|v_t\|_{H^2}\|\bpartial^4 v\|_{L^2}\ dt'\\
&\lesssim \frac{1}{\epsilon}tP(\sup\limits_{t'\in[0,t]}F)+\frac{1}{\epsilon}P(M_0)+\frac{1}{\epsilon}\left(\int_{0}^{t} \| v_t\|_{H^2}^2\ dt'\right)^2+\epsilon \int_{0}^{t}\|\bpartial^4 v\|_{L^2}^2\ dt',\\
|-\mathrm{II}|&= \iint \sum\limits_{m=0}^3| \bpartial^{m}  A| |\bpartial^{3-m} (  R\rho_0 \Theta)  | |\bpartial^3 D v|\\
&\lesssim \int_{0}^{t}  \bigg[ \|\bpartial ^3 A\|_{L^2} \|\rho_0\|_{L^{\infty}}\|\Theta\|_{L^{\infty}}
    +\|\bpartial^2 A\|_{L^4}
    \Big(\|\bpartial \rho_0\|_{L^4}\|\Theta\|_{L^{\infty}}
    +\| \rho_0\|_{L^{\infty}}\|\bpartial\Theta\|_{L^4}\Big)\\
&\qquad+\|\bpartial A\|_{L^{\infty}}\Big( \|\bpartial^2 \rho_0\|_{L^2}\|\Theta\|_{L^{\infty}}
    + \|\bpartial \rho_0\|_{L^4}\|\bpartial\Theta\|_{L^4}
    +\|\rho_0\|_{L^{\infty}}\|\bpartial^2 \Theta \|_{L^2} \Big)\\
&\qquad+\|A\|_{L^{\infty}}\Big( \|d \bpartial^3 \rho_0\|_{L^2} \|\frac{\Theta}{d}\|_{L^{\infty}}
+\|\bpartial^2 \rho_0\|_{L^2}\|\bpartial\Theta\|_{L^{\infty}}
+ \|\bpartial \rho_0\|_{L^4}\|\bpartial^2\Theta\|_{L^4}\\
\qquad&\qquad+\|\rho_0\|_{L^{\infty}}\|\bpartial^3 \Theta \|_{L^2}\Big)
\bigg]\|\bpartial^3 Dv\|_{L^2} \ dt' \\
&\lesssim   \int_{0}^{t} \Big( t'P(\sup\limits_{[0,t']}F)+P(M_0)  \Big)^{1/2}
\Big( 
\|\bpartial \rho_0\|_{H^1}
+\|\rho_0\|_{L^{\infty}}+\|d \bpartial^3 \rho_0\|_{L^2}\Big)\|\Theta\|_{H^3} \|\bpartial^3 Dv\|_{L^2}\ dt'\\
&\lesssim  \frac{1}{\epsilon}tP(\sup\limits_{t'\in[0,t]}F)+\frac{1}{\epsilon}P(M_0)+\epsilon \int_{0}^{t}\|\bpartial^3 Dv\|_{L^2}^2\ dt'.
\end{align*}

Note that in $\|\frac{\Theta}{d}\|_{L^{\infty}}\lesssim \|D\Theta\|_{L^{\infty}}\lesssim  \|\Theta\|_{H^3}$, L'Hospital rule and boundary condition of $\Theta$ are used.
After combining the estimates of $\mathrm{I}, \mathrm{II}, \mathrm{III}$ and choosing small enough $\epsilon$ (independent of $v$), we have that 
\begin{equation}\label{vxxx}
 \int_{0}^{t}\|\bpartial^3 v\|_{H^1}^2\ dt'\lesssim tP(\sup\limits_{t'\in[0,t]}F)+P(M_0)+P\left(\int_{0}^{t} \| v_t\|_{H^2}^2\ dt'\right).
\end{equation}

\vspace{0.3cm}

\textsl{\textbf{Estimates of $\int_{0}^{t}\|\bpartial^3 \Theta\|_{H_0^1}^2 \ dt'$:}}~~
Here we calculate
$\int_{0}^{t}\int_{\Omega}\bpartial^3 \Big(\eqref{1}_2\Big)\cdot  \bpartial^3 \Theta\ dx dt'$.
\begin{align*}
c_v\underbrace{\iint \bpartial ^3 \Big(\rho_0\Theta_{t}\Big)\bpartial^3 \Theta}_{\mathrm{I}}
+\underbrace{\iint  \bpartial^3  \Big(\frac{R\rho_0\Theta}{J} a^r_i v^i_{,r}\Big) \bpartial^3 \Theta}_{\mathrm{II}}
-\underbrace{\iint \bpartial^3 (\mathbb{S}^{ij}_{\eta}[v] a^r_j v^i_{,r})  \bpartial^3 \Theta}_{\mathrm{III}}&\\
+\underbrace{\kappa\iint \bpartial^3(a^r_i A^k_j \delta^{ij}\Theta_{,k}) \bpartial^3 \Theta_{,r}}_{\mathrm{IV}}=0,
\qquad\qquad\qquad\qquad&
\end{align*}
with
$$
\mathrm{IV}=\kappa \iint a^r_i A^k_j \delta^{ij} \bpartial^3 \Theta_{,k}
\bpartial^3 \Theta_{,r}+\mathrm{IV}',
$$
where $\mathrm{IV}'$ denotes all the remaining terms.
Similarly as before,
\begin{align*}
|\mathrm{IV}'|
&\lesssim  \frac{1}{\epsilon}tP(\sup\limits_{t'\in[0,t]}F)+\frac{1}{\epsilon}P( M_0) +\epsilon \int_{0}^{t}\|\bpartial^3 D\Theta\|_{L^2}^2\ dt',\\
\mathrm{IV}-\mathrm{IV}'
&\geq \frac{1}{C} \int_{0}^{t} \|\bpartial^3 D\Theta\|_{L^2}^2 \ dt'-CtP(\sup\limits_{t'\in[0,t]}F),\\
|\mathrm{I}| &\lesssim  \frac{1}{\epsilon}tP(\sup\limits_{t'\in[0,t]}F)+\frac{1}{\epsilon}P(M_0)+\frac{1}{\epsilon}\left(\int_{0}^{t} \| \Theta_t\|_{H^2}^2\ dt'\right)^2+\epsilon \int_{0}^{t}\|\bpartial^4 \Theta\|_{L^2}^2\ dt', 
\end{align*}
\begin{align*}
|\mathrm{II}|
&\lesssim \int_{0}^{t} \Big( \|\bpartial^3 (A\rho_0 \Theta)\|_{L^2} \| Dv\|_{L^{\infty}}
    +\|\bpartial^2 (A\rho_0 \Theta)\|_{L^2}  \|\bpartial Dv\|_{L^4}\\
   &\qquad
    +\|\bpartial (A\rho_0 \Theta)\|_{L^4} \| \bpartial^2 Dv\|_{L^2}
    +\| (A\rho_0 \Theta)\|_{L^{\infty}}  \|\bpartial^3 Dv\|_{L^2}\Big)\|\bpartial^3 \Theta\|_{L^4}\\
&\lesssim \int_{0}^{t}  \Big( t'P(\sup\limits_{[0,t']}F)+P(M_0)  \Big)^{1/2}
   \Big( \|\bpartial \rho_0\|_{H^1}
+\|\rho_0\|_{L^{\infty}}+\|d \bpartial^3 \rho_0\|_{L^2}\Big)\|\Theta\|_{H^3}\\
&\qquad \cdot\Big(\|v\|_{H^3}+\|\bpartial^3 v\|_{H^1}\Big) \|\bpartial^3 \Theta \|_{H_0^1}\\
&\lesssim  \frac{1}{\epsilon}tP(\sup\limits_{t'\in[0,t]}F)+\frac{1}{\epsilon}P(M_0)+\frac{1}{\epsilon}\left(\int_{0}^{t} \|\bpartial^3 v\|_{H^1}^2\ dt'\right)^2+\epsilon \int_{0}^{t}\|\bpartial^3 \Theta \|_{H_0^1}^2\ dt',
\end{align*}
\begin{align*}
|\mathrm{III}|
&\lesssim \int_{0}^{t} \Big[ \|\bpartial^3 (aA)\|_{L^2}  \|Dv\|_{L^{\infty}}^2
            + \|\bpartial^2 (aA)\|_{L^4}\|\bpartial Dv\|_{L^4} \|Dv\|_{L^{\infty}}
            +\|\bpartial (aA)\|_{L^{\infty}} \Big(\|\bpartial^2 Dv\|_{L^2} \|Dv\|_{L^{\infty}}\\
            &\qquad+\|\bpartial Dv\|_{L^4}^2\Big)
            + \| (aA)\|_{L^{\infty}} \Big(\|\bpartial^3 Dv\|_{L^2} \|Dv\|_{L^{\infty}}+\|\bpartial^2 Dv\|_{L^2} \|\bpartial Dv\|_{L^4} \Big) \Big] \|\bpartial^3\Theta\|_{L^4} \ dt'\\
&\lesssim \int_{0}^{t} \Big[ \Big( t'P(\sup\limits_{[0,t']}F)+P(M_0)  \Big)^{1/2}
          \Big(\|v\|_{H^3}+ \|\bpartial^3 v\|_{H^1}\Big) \|v\|_{H^3}
           \Big] \|\bpartial^3 \Theta\|_{H_0^1} \ dt'\\
&\lesssim  \frac{1}{\epsilon}tP(\sup\limits_{t'\in[0,t]}F)+\frac{1}{\epsilon}P(M_0)+\frac{1}{\epsilon}\left(\int_{0}^{t} \|\bpartial^3 v\|_{H^1}^2\ dt'\right)^2+\epsilon \int_{0}^{t}\|\bpartial^3\Theta\|_{H_0^1}^2\ dt'.
\end{align*}
Note that here in the estimates of $\mathrm{II}$ and $\mathrm{III}$, \eqref{direct} is applied to substituting $\|v_t\|_{H^1}$ or $\|v\|_{H^3}$.

Note that $\bpartial^3 \Theta=0$ on $\Gamma$.
After combining the estimates of $\mathrm{I}, \mathrm{II}, \mathrm{III}, \mathrm{IV}$ 
and choosing small enough $\epsilon$ (independent of $v$ or $\Theta$), 
we have that
\begin{equation}\label{thetaxxx}
\begin{aligned}
\int_{0}^{t}\|\bpartial^3 \Theta\|_{H_0^1}^2\ dt'
&\lesssim tP(\sup\limits_{t'\in[0,t]}F)+P(M_0)+P \Big(\int_{0}^{t} \|\Theta_t\|_{H^2}^2+\|\bpartial^3 v\|_{H^1}^2 \ dt'\Big).
\end{aligned}
\end{equation}

\vspace{0.3cm}

\subsection{Elliptic estimates}~

\vspace{0.3cm}
\textsl{\textbf{Estimates of $\int_{0}^{t}\| v_t\|_{H^2}^2\ dt'$:}}~~
After applying a temporal derivative to $\eqref{1}_1$, one has
$$
\rho_0v^i_{tt}+\partial_t \left(a^r_i \frac{R\rho_0\Theta}{J}\right)_{,r}=\partial_t \Big(a^r_j \mathbb{S}^{ij}_{\eta}[v]\Big)_{,r}.
$$
Therefore, by the definition of $\mathbb{S}^{ij}_{\eta}[v]$ we have
\begin{align*}
&\int_{0}^{t}\sum\limits_{i=1}^3\left\| a^r_j\Big[\mu (A^{k}_j  v^i_{t,rk}+A^k_i  v^j_{t,rk})+\lambda A_l^k  v^l_{t,rk}\delta^i_j\Big]   \right\|_{L^2}^2 \ dt'\\
\lesssim & \int_{0}^{t} \|\rho_0 v_{tt}\|_{L^2}^2
   +\| \partial_t D(A\rho_0\Theta)\|_{L^2}^2
   +\| \partial_t D(a A \cdot Dv)- (aA) D^2v_t\|_{L^2}^2\ dt'\\
:=&\mathrm{I}+\mathrm{II}+\mathrm{III},
\end{align*}
It is straightforward that
\begin{align*}
\mathrm{I}&\lesssim tP(\sup\limits_{t'\in[0,t]}F),\\
\mathrm{II}&\lesssim \int_{0}^{t} \bigg[\|\partial_t D A\|_{L^4} \|\rho_0\|_{L^{\infty}} \|\Theta\|_{ L^{\infty}}
 +\|\partial_t A\|_{ L^{\infty}}\Big(\|D\rho_0\|_{ L^3} \|\Theta\|_{ L^{\infty}}+\|\rho_0\|_{L^{\infty}} \|D\Theta\|_{ L^{\infty}}\Big)\\
&\qquad  +\|DA\|_{L^4} \|\rho_0\|_{L^{\infty}} \| \Theta_t\|_{ L^4}+ \|A\|_{L^{\infty}}\Big( \| D\rho_0\|_{L^3} \|\Theta_t\|_{ L^6}+\|\rho_0\|_{ L^{\infty}} \|D\Theta_t\|_{ L^2}\Big)\bigg]^2\ dt'\\
&\lesssim \int_{0}^{t} \Big( t'P(\sup\limits_{[0,t']}F)+P(M_0)  \Big)\Big(\|\rho_0\|_{L^{\infty}}+\|D\rho_0\|_{L^3}\Big)^2
\Big(\|\Theta\|_{H^3}+\|\Theta_t\|_{H_0^1}\Big)^2 \ dt'\\
&\lesssim tP(\sup\limits_{t'\in[0,t]}F),
\end{align*}
\begin{align*}
\mathrm{III}&\lesssim \int_{0}^{t}  \Big(\|\partial_t D(aA)\|_{L^4}\|Dv\|_{L^{\infty}}
+\|\partial_t (aA)\|_{L^{\infty}}\|D^2 v\|_{L^4}
+\|D (aA)\|_{L^4} \| Dv_t \|_{L^4}\Big)^2 \ dt'\\
&\lesssim \int_{0}^{t} \Big( t'P(\sup\limits_{[0,t']}F)+P(M_0)  \Big)
\Big(\|v\|_{H^3}+\|Dv_t\|_{L^2}^{1/4}\|D^2v_t\|_{L^2}^{3/4}+\|Dv_t\|_{L^2}\Big)^2 \ dt'\\
&\lesssim  \frac{1}{\epsilon}tP(\sup\limits_{t'\in[0,t]}F)+\frac{1}{\epsilon}P(M_0)+\epsilon \int_{0}^{t}\| v_t \|_{H^2}^2\ dt'.
\end{align*}
Here we have used the embedding inequality $\|Dv_t\|_{L^4}\lesssim \|Dv_t\|_{L^2}^{1/4}\|D^2v_t\|_{L^2}^{3/4}+\|Dv_t\|_{L^2}$.
On the other hand
\begin{align*}
&\int_{0}^{t}\sum\limits_{i=1}^3\left\| a^r_j\Big[\mu (A^{k}_j  v^i_{t,rk}+A^k_i  v^j_{t,rk})+\lambda A_l^k  v^l_{t,rk}\delta^i_j\Big]   \right\|_{L^2}^2 \ dt'\\
\geq &\int_{0}^{t}\sum\limits_{i=1}^3\left\| a^3_j\Big[\mu (A^3_j  v^i_{t,33}+A^3_i  v^j_{t,33})+\lambda A_l^3  v^l_{t,33}\delta^i_j\Big]   \right\|_{L^2}^2 \ dt'-C\int_{0}^{t} \|\bpartial v_t\|_{H^1}^2 \ dt'\\
\geq & \int_{0}^{t}\sum\limits_{i=1}^3\left\| \delta^3_j\Big[\mu (\delta^3_j  v^i_{t,33}+\delta^3_i  v^j_{t,33})+\lambda \delta_l^3  v^l_{t,33}\delta^i_j\Big]   \right\|_{L^2}^2 \ dt'\\
&\qquad-C\int_{0}^{t} \|A-\delta\|_{L^{\infty}}^2\|v\|_{H^2}^2-C\int_{0}^{t} \|\bpartial v_t\|_{H^1}^2 \ dt'\\
\geq & \frac{1}{C}\int_{0}^{t}\|\partial_3^2 v_t\|_{L^2}^2 \ dt'-CtP(\sup\limits_{t'\in[0,t]}F)-C\int_{0}^{t} \|\bpartial v_t\|_{H^1}^2 \ dt'.
\end{align*}
So it follows from the estimates of $\mathrm{I},\mathrm{II},\mathrm{III}$ that
\begin{equation}\label{dvtx}
\begin{aligned}
\int_{0}^{t}\| v_t\|_{H^2}^2 \ dt'\lesssim tP(\sup\limits_{t'\in[0,t]}F)+P(M_0)+\int_{0}^{t} \|\bpartial v_t\|_{H^1}^2 \ dt'.
\end{aligned}
\end{equation}

\vspace{0.3cm}

\textsl{\textbf{Estimates of $\int_{0}^{t}\| \Theta_t\|_{H^2}^2\ dt'$:}}~~
Similarly, after applying a temporal derivative to $\eqref{1}_2$, it holds
\begin{align*}
c_v\rho_0 \Theta_{tt}+\partial_t \left( \frac{R\rho_0\Theta}{J}a^r_iv^i_{,r}\right)=\partial_t(\mathbb{S}_{\eta}^{ij}[v]a^r_jv^i_{,r})+\kappa \partial_t (a^r_i A^k_j \delta^{ij}\Theta_{,k})_{,r},
\end{align*}
Then it follows
\begin{align*}
\int_{0}^{t}\sum\limits_{i=1}^3\left\|a^r_i A^k_j \delta^{ij}\Theta_{t,rk} \right\|_{L^2}^2 \ dt'
&\lesssim \int_{0}^{t} \|\rho_0 \Theta_{tt}\|_{L^2}^2
   +\| \partial_t (A\rho_0\Theta\cdot Dv)\|_{L^2}^2
   +\| \partial_t (a A \cdot |Dv|^2)\|_{L^2}^2\\
& +\|\partial_t D(aA\cdot D\Theta)-(aA) D^2 \Theta_t\|_{L^2}^2\ dt'
:=\mathrm{I}+\mathrm{II}+\mathrm{III} +\mathrm{IV}.
\end{align*}
Straightforward calculation shows that
\begin{align*}
\mathrm{I}&\lesssim tP(\sup\limits_{t'\in[0,t]}F),\\
\mathrm{II}&\lesssim \int_{0}^{t} \Big(\|\partial_t(A\rho_0\Theta)\|_{L^2}\|Dv\|_{L^{\infty}}+\|A\rho_0\Theta\|_{L^{\infty}}\|Dv_t\|_{L^2}\Big)^2\ dt' \\
&\lesssim \int_{0}^{t} \Big( t'P(\sup\limits_{[0,t']}F)+P(M_0)  \Big)\|\rho_0\|_{L^{\infty}}^2
\Big(\|\Theta\|_{H^3}+\|\Theta_t\|_{H_0^1}\Big)^2
\Big(\|v\|_{H^3}+\|v_t\|_{H^1}\Big)^2 \ dt'\\
&\lesssim tP(\sup\limits_{t'\in[0,t]}F)+P(M_0),\\
\mathrm{III}&\lesssim \int_{0}^{t}  \Big(\|\partial_t (aA)\|_{L^{\infty}}\|Dv\|_{L^{\infty}}^2
+\|(aA)\|_{L^{\infty}}\|D v_t\|_{L^2}\|Dv\|_{L^{\infty}}
\Big)^2 \ dt'\\
&\lesssim \int_{0}^{t} \Big( t'P(\sup\limits_{[0,t']}F)+P(M_0)  \Big)
\Big(\|v\|_{H^3}^2+\|v_t\|_{H^1}\|v\|_{H^3}\Big)^2 \ dt'\\
&\lesssim tP(\sup\limits_{t'\in[0,t]}F)+P(M_0).\\
\mathrm{IV}
&\lesssim \frac{1}{\epsilon}tP(\sup\limits_{t'\in[0,t]}F)+\frac{1}{\epsilon}P(M_0)+\epsilon \int_{0}^{t}\| \Theta_t \|_{H^2}^2\ dt',\\
\end{align*}
\begin{align*}
\int_{0}^{t}\sum\limits_{i=1}^3\left\|a^r_iA^k_i\Theta_{t,rk} \right\|_{L^2}^2 \ dt'
\geq & \frac{1}{C}\int_{0}^{t}\|\partial_3^2 \Theta_t\|_{L^2}^2 \ dt'-CtP(\sup\limits_{t'\in[0,t]}F)-C\int_{0}^{t} \|\bpartial \Theta_t\|_{H_0^1}^2 \ dt'.
\end{align*}
So it follows from the estimates of $\mathrm{I},\mathrm{II},\mathrm{III}$ that
\begin{equation}\label{dthetatx}
\begin{aligned}
\int_{0}^{t}\| \Theta_t\|_{H^2}^2 \ dt'\lesssim tP(\sup\limits_{t'\in[0,t]}F)+P(M_0)+\int_{0}^{t} \|\bpartial \Theta_t\|_{H_0^1}^2 \ dt'.
\end{aligned}
\end{equation}

\vspace{0.3cm}
\textsl{\textbf{Estimates of $\int_{0}^{t}\| v_t\|_{H^3}^2\ dt'$:}}~~
After applying one spatial and temporal derivative to $\eqref{1}_1$, that is $\partial_t D \Big(\eqref{1}_1\Big)$,
$$
D(\rho_0v^i_{tt})+\partial_t D \left(a^r_i \frac{R\rho_0\Theta}{J}\right)_{,r}=\partial_t D \Big(a^r_j \mathbb{S}^{ij}_{\eta}[v]\Big)_{,r}.
$$
Then 
\begin{align*}
&\int_{0}^{t}\sum\limits_{i=1}^3\left\| a^r_j\Big[\mu (A^{k}_j  D v^i_{t,rk}+A^k_i D v^j_{t,rk})+\lambda A_l^k D v^l_{t,rk}\delta^i_j\Big]   \right\|_{L^2}^2 \ dt'\\
\lesssim & \int_{0}^{t} \|D(\rho_0 v_{tt})\|_{L^2}^2
 +\| \partial_t D^2(A\rho_0\Theta)\|_{L^2}^2
 +\|\partial_t D^2(a A \cdot Dv)- (aA) D^3v_t\|_{L^2}^2\ dt'\\
:=&\mathrm{I}+\mathrm{II}+\mathrm{III},
\end{align*}
where
\begin{align*}
\mathrm{I}&\lesssim \int_{0}^{t} (\|D\rho_0\|_{L^3}\|v_{tt}\|_{L^6}+\|\rho_0\|_{L^{\infty}}\|Dv_{tt}\|_{L^2})^2\ dt' \lesssim  \int_{0}^{t} \|v_{tt}\|_{H^1}^2\ dt',\\
\mathrm{II}&\lesssim \int_{0}^{t} \bigg[\|\partial_t D^2 A\|_{L^2} \|\rho_0\|_{L^{\infty}} \|\Theta\|_{ L^{\infty}}
 +\|\partial_t D A\|_{L^6}\Big(\|D\rho_0\|_{ L^3} \|\Theta\|_{L^{\infty}}+\|\rho_0\|_{L^{\infty}} \|D\Theta\|_{ L^{\infty}}\Big)\\
 &\qquad +\|\partial_t  A\|_{ L^{\infty}}\Big(\|d (D^2\rho_0)\|_{L^2} \|\frac{\Theta}{d}\|_{ L^{\infty}}+\|D\rho_0\|_{L^3} \|D\Theta\|_{L^6}+\|\rho_0\|_{L^{\infty}}\|D^2\Theta\|_{L^2}\Big)\\
&\qquad  +\|D^2A\|_{L^2} \|\rho_0\|_{L^{\infty}} \| \Theta_t\|_{ L^{\infty}}+ \|DA\|_{L^6}\Big( \| D\rho_0\|_{L^3} \|\Theta_t\|_{ L^{\infty}}+\|\rho_0\|_{ L^{\infty}} \|D\Theta_t\|_{ L^4}\Big)\\
&\qquad +\|  A\|_{ L^{\infty}}\Big(\|d (D^2\rho_0)\|_{L^2} \|\frac{\Theta_t}{d}\|_{ L^{\infty}}+\|D\rho_0\|_{L^3} \|D\Theta_t\|_{L^6}+\|\rho_0\|_{L^{\infty}}\|D^2\Theta_t\|_{L^2}\Big)\bigg]^2\ dt'\\
&\lesssim \int_{0}^{t} \Big( t'P(\sup\limits_{[0,t']}F)+P(M_0)  \Big)\Big(\|\rho_0\|_{L^{\infty}}+\|D\rho_0\|_{L^3}+\|d(D^2 \rho_0)\|_{L^2}\Big)^2\\
&\qquad\cdot\Big(\|\Theta\|_{H^3}+\|\Theta_t\|_{H^2}+\|D\Theta_t\|_{L^2}^{1/4}\|D^3\Theta_t\|_{L^2}^{3/4}\Big)^2 \ dt'\\
&\lesssim  \frac{1}{\epsilon}tP(\sup\limits_{t'\in[0,t]}F)+\frac{1}{\epsilon}P(M_0)+\frac{1}{\epsilon}P\left(\int_{0}^{t}\|\Theta_t\|_{H^2}^2 \ dt'\right)+\epsilon \int_{0}^{t}\|D^3\Theta_t\|_{L^2}^2 \ dt'.
\end{align*}
Note that $\Theta, \Theta_t=0$ on the boundary, and thus $\|\frac{\Theta}{d}, \frac{\Theta_t}{d}\|_{ L^{\infty}}\lesssim \|D\Theta, D\Theta_t\|_{ L^{\infty}}$. Here we have used the embedding inequalities $\|\cdot\|_{L^6}\lesssim \|\cdot\|_{H^1}$ and $\| D\Theta_t\|_{ L^{\infty}} \lesssim  \|D\Theta_t\|_{L^2}^{1/4}\|D^3\Theta_t\|_{L^2}^{3/4} +\|D\Theta_t\|_{L^2}$.
\begin{align*}
\mathrm{III}&\lesssim   
\int_{0}^{t}  \Big(\|\partial_t D^2(aA)\|_{L^2}\|Dv\|_{L^{\infty}}
+\|\partial_t D (aA)\|_{L^4}\|D^2 v\|_{L^4}
+\|\partial_t (aA)\|_{L^{\infty}}\|D^3 v\|_{L^2}\\
&\qquad
+\|D^2 (aA)\|_{L^2} \| Dv_t \|_{L^{\infty}}
+\|D(aA)\|_{L^4} \|D^2v_t\|_{L^4}\Big)^2 \ dt'\\
&\lesssim \int_{0}^{t} \Big( t'P(\sup\limits_{[0,t']}F)+P(M_0)  \Big)
\Big(\|v\|_{H^3}+\|v_t\|_{H^2}+\|Dv_t\|_{L^2}^{1/4}\|D^3v_t\|_{L^2}^{3/4}+\|D^2v_t\|_{L^2}^{1/4}\|D^3v_t\|_{L^2}^{3/4}\Big)^2 \ dt'\\
&\lesssim  \frac{1}{\epsilon}tP(\sup\limits_{t'\in[0,t]}F)+\frac{1}{\epsilon}P(M_0)+\frac{1}{\epsilon}P\left(\int_{0}^{t}\|v_t\|_{H^2}^2 \ dt'\right)+\epsilon \int_{0}^{t}\|D^3v_t\|_{L^2}^2 \ dt',
\end{align*}
Here we have used the embedding inequalities
$\|Dv_t\|_{L^{\infty}}\lesssim \|Dv_t\|_{L^2}^{1/4}\|D^3v_t\|_{L^2}^{3/4}+\|Dv_t\|_{L^2}$ and $\|D^2v_t\|_{L^4}\lesssim \|D^2v_t\|_{L^2}^{1/4}\|D^3v_t\|_{L^2}^{3/4}+\|D^2v_t\|_{L^2}$.
\begin{align*}
&\int_{0}^{t}\sum\limits_{i=1}^3\left\| a^r_j\Big[\mu (A^{k}_j  Dv^i_{t,rk}+A^k_i Dv^j_{t,rk})+\lambda A_l^k  Dv^l_{t,rk}\delta^i_j\Big]   \right\|_{L^2}^2 \ dt'\\
\geq & \frac{1}{C}\int_{0}^{t}\|\partial_3^3 v_t\|_{L^2}^2 \ dt'-CtP(\sup\limits_{t'\in[0,t]}F)-C\int_{0}^{t} \|\bpartial v_t\|_{H^2}^2 \ dt'.
\end{align*}
So it follows from the estimates of $\mathrm{I},\mathrm{II},\mathrm{III}$ that
\begin{equation}
\label{dvtxx2}
\begin{aligned}
\int_{0}^{t}\| v_t\|_{H^3}^2 \ dt' &\lesssim  \frac{1}{\epsilon}tP(\sup\limits_{t'\in[0,t]}F)+\frac{1}{\epsilon}P(M_0)+\frac{1}{\epsilon}P\left(\int_{0}^{t}\|v_{tt}\|_{H^1}^2+\|v_t\|_{H^2}^2+\|\Theta_t\|_{H^2}^2 +\|\bpartial v_t\|_{H^2}^2\ dt'\right)\\
&\qquad+\epsilon \int_{0}^{t}\|D^3\Theta_t\|_{L^2}^2 \ dt',
\end{aligned}
\end{equation}

\vspace{0.3cm}
Using the method to analysis $\bpartial\partial_t  \Big(\eqref{1}_1\Big)$, we have that
\begin{equation}
\label{dvtxx1}
\begin{aligned}
\int_{0}^{t}\| \bpartial v_t\|_{H^2}^2 \ dt' &\lesssim tP(\sup\limits_{t'\in[0,t]}F)+\frac{1}{\epsilon}P(M_0)+P\left(\int_{0}^{t}\|v_{tt}\|_{H^1}^2+\|v_t\|_{H^2}^2+\|\Theta_t\|_{H^2}^2 +\|\bpartial^2 v_t\|_{H^1}^2\ dt'\right)
.
\end{aligned}
\end{equation}
Here $\|D^3\Theta_t\|_{L^2}^2$ does not appear because instead of using the estimates $\|A\|_{L^{\infty}} \|d (D^2 \rho_0)\|_{L^2} \|\frac{\Theta_t}{d}\|_{L^{\infty}}$ in $\mathrm{II}$, we use $\|A\|_{L^{\infty}} \| \bpartial D\rho_0\|_{L^2} \|\Theta_t\|_{L^{\infty}}\lesssim \|\bpartial D\rho_0\|_{L^2}\|\Theta_t\|_{H^2}$ in the corresponding terms.

\vspace{0.3cm}
\textsl{\textbf{Estimates of $\int_{0}^{t}\| \Theta_t\|_{H^3}^2\ dt'$:}}~
From $\eqref{1}_2$, we have
\begin{align*}
c_vD(\rho_0 \Theta_{tt})+\partial_t D \left( \frac{R\rho_0\Theta}{J}a^r_iv^i_{,r}\right)=\partial_t D(\mathbb{S}_{\eta}^{ij}[v]a^r_jv^i_{,r})+\kappa \partial_t D (a^r_i A^k_j \delta^{ij}\Theta_{,k})_{,r},
\end{align*}
Then
\begin{align*}
\int_{0}^{t}\sum\limits_{i=1}^3\left\|a^r_i A^k_j \delta^{ij} D\Theta_{t,rk} \right\|_{L^2}^2 \ dt'
\lesssim & \int_{0}^{t} \|D(\rho_0 \Theta_{tt})\|_{L^2}^2
   +\| \partial_t D(A\rho_0\Theta\cdot Dv)\|_{L^2}^2\\
   +\| \partial_t D(a A \cdot |Dv|^2)\|_{L^2}^2&
   +\|\partial_t D^2(aA\cdot D\Theta)-(aA) D^3 \Theta_t\|_{L^2}^2\ dt'\\
:=&\mathrm{I}+\mathrm{II}+\mathrm{III} +\mathrm{IV},
\end{align*}
where
\begin{align*}
\mathrm{I}&\lesssim  \int_{0}^{t} \|\Theta_{tt}\|_{H_0^1}^2\ dt',\\
\mathrm{II}&\lesssim \int_{0}^{t} \Big[ 
    \|\partial_t D(A\rho_0\Theta)\|_{L^2} \|Dv\|_{L^{\infty}}
    +\|\partial_t (A\rho_0\Theta)\|_{L^4}\|D^2v\|_{L^4}\\
     &\qquad+ \| D(A\rho_0\Theta)\|_{L^3} \|Dv_t\|_{L^6}
     +\| (A\rho_0\Theta)\|_{L^{\infty}}\|D^2v_t\|_{L^2} \Big]^2 \ dt'\\
&\lesssim \int_{0}^{t} \Big( t'P(\sup\limits_{[0,t']}F)+P(M_0)  \Big)
     \Big(\|\rho_0\|_{L^{\infty}}+\|D\rho_0\|_{L^3}\Big)^2\\
&\qquad\cdot\Big((\|\Theta\|_{H^3}+\|\Theta_t\|_{H^2})\|v\|_{H^3}
      +\|\Theta\|_{H^3}\|v_t\|_{H^2}\Big)^2\ dt'\\
&\lesssim tP(\sup\limits_{t'\in[0,t]}F)+P(M_0)
     +P\left(\int_{0}^{t}\|v_t\|_{H^2}^2+\|\Theta_t\|_{H^2}^2 \ dt'\right),
\end{align*}
\begin{align*}
\mathrm{III}&\lesssim \int_{0}^{t}  \Big[
      \|\partial_t D(aA)\|_{L^4} \|Dv\|_{L^{\infty}}^2
     +\| D (aA)\|_{L^4}  \|D v_t\|_{L^4}\|Dv\|_{L^{\infty}}
     + \| \partial_t (aA)\|_{L^{\infty}}  \|D^2 v\|_{L^2}\|Dv\|_{L^{\infty}}\\
         &\qquad
     +\|(aA)\|_{L^{\infty}}\Big( \|D^2 v_t\|_{L^2}\|Dv\|_{L^{\infty}}+\|D v_t\|_{L^4}\|D^2v\|_{L^4}\Big)\Big]^2 \ dt'\\
&\lesssim  \int_{0}^{t} \Big( t'P(\sup\limits_{[0,t']}F)+P(M_0)  \Big)
      \Big(\|v\|_{H^3}+\|v_t\|_{H^2}\Big)^2 \|v\|_{H^3}^2 \ dt'\\
&\lesssim tP(\sup\limits_{t'\in[0,t]}F)+P(M_0)
     +P\left(\int_{0}^{t}\|v_t\|_{H^2}^2 \ dt'\right),\\
\mathrm{IV}
&\lesssim  \frac{1}{\epsilon}tP(\sup\limits_{t'\in[0,t]}F)+\frac{1}{\epsilon}P(M_0)+\frac{1}{\epsilon}P\left(\int_{0}^{t}\|\Theta_t\|_{H^2}^2 \ dt'\right)+\epsilon \int_{0}^{t}\|D^3\Theta_t\|_{L^2}^2 \ dt'.
\end{align*}
On the other hand, 
\begin{align*}
\int_{0}^{t}\sum\limits_{i=1}^3\left\|a^r_iA^k_i\Theta_{t,rk} \right\|_{L^2}^2 \ dt'
\geq  \frac{1}{C}\int_{0}^{t}\|\partial_3^3 \Theta_t\|_{L^2}^2 \ dt'-CtP(\sup\limits_{t'\in[0,t]}F)-C\int_{0}^{t} \|\bpartial \Theta_t\|_{H^2}^2 \ dt'.
\end{align*}
So it follows from the estimates of $\mathrm{I},\mathrm{II},\mathrm{III}$ that
\begin{equation}
\label{dthetatxx2}
\begin{aligned}
\int_{0}^{t}\| \Theta_t\|_{H^3}^2 \ dt' &\lesssim tP(\sup\limits_{t'\in[0,t]}F)+P(M_0)+P\left(\int_{0}^{t}\|\Theta_{tt}\|_{H_0^1}^2+\|v_t\|_{H^2}^2+\|\Theta_t\|_{H^2}^2 +\|\bpartial \Theta_t\|_{H^2}^2\ dt'\right).
\end{aligned}
\end{equation}
Using the method to analysis $\bpartial\partial_t  \Big(\eqref{1}_2\Big)$, we have that
\begin{equation}
\label{dthetatxx1}
\begin{aligned}
\int_{0}^{t}\| \bpartial \Theta_t\|_{H^2}^2 \ dt' &\lesssim tP(\sup\limits_{t'\in[0,t]}F)+\frac{1}{\epsilon}P(M_0)+P\left(\int_{0}^{t}\|\Theta_{tt}\|_{H_0^1}^2+\|v_t\|_{H^2}^2+\|\Theta_t\|_{H^2}^2 +\|\bpartial^2 \Theta_t\|_{H_0^1}^2\ dt' \right).
\end{aligned}
\end{equation}

\vspace{0.3cm}

\textsl{\textbf{Estimates of $\int_{0}^{t}\| \bpartial v\|_{H^3}^2\ dt'$:}}~
$\bpartial D \Big(\eqref{1}_1\Big)$ will yield
$$
\bpartial D(\rho_0v^i_{t})+\bpartial D \left(a^r_i \frac{R\rho_0\Theta}{J}\right)_{,r}=\bpartial D \Big(a^r_j \mathbb{S}^{ij}_{\eta}[v]\Big)_{,r} ~.
$$
We have
\begin{align*}
&\int_{0}^{t}\sum\limits_{i=1}^3\left\| a^r_j\Big[\mu (A^{k}_j  \bpartial D v^i_{,rk}+A^k_i\bpartial D v^j_{,rk})+\lambda A_l^k \bpartial D v^l_{,rk}\delta^i_j\Big]   \right\|_{L^2}^2 \ dt'\\
\lesssim & \int_{0}^{t} \|\bpartial D(\rho_0 v_{t})\|_{L^2}^2
 +\| \bpartial D^2(A\rho_0\Theta)\|_{L^2}^2
 +\|\bpartial D^2(a A \cdot Dv)- (aA) \bpartial D^3v\|_{L^2}^2\ dt'\\
:=&\mathrm{I}+\mathrm{II}+\mathrm{III} .
\end{align*}
Here
\begin{align*}
\mathrm{I}&\lesssim \int_{0}^{t} (\|\bpartial D\rho_0\|_{L^2}\|v_{t}\|_{L^{\infty}}
   +\|D\rho_0\|_{L^3}\|D v_t\|_{6}
   +\|\rho_0\|_{L^{\infty}}\|\bpartial Dv_{t}\|_{L^2})^2\ dt' 
   \lesssim  \int_{0}^{t} \|v_{t}\|_{H^2}^2\ dt',
\end{align*}
\begin{align*}
\mathrm{II}&\lesssim \int_{0}^{t} \bigg[\|\bpartial D^2 A\|_{L^2} \|\rho_0\|_{L^{\infty}} \|\Theta\|_{ L^{\infty}}
  +\|\bpartial D A\|_{L^6}\Big(\|D\rho_0\|_{ L^3} \|\Theta\|_{L^{\infty}}+\|\rho_0\|_{L^{\infty}} \|D\Theta\|_{ L^{\infty}}\Big)\\
  &\qquad +\|\bpartial  A\|_{ L^{\infty}}\Big(\|d (D^2\rho_0)\|_{L^2} \|\frac{\Theta}{d}\|_{ L^{\infty}}+\|D\rho_0\|_{L^3} \|D\Theta\|_{L^6}+\|\rho_0\|_{L^{\infty}}\|D^2\Theta\|_{L^2}\Big)  
  +\|D^2A\|_{L^2} \\
  &\qquad\cdot\Big(\|\rho_0\|_{L^{\infty}} \| \bpartial \Theta\|_{ L^{\infty}}+\|\bpartial\rho_0\|_{L^4}\|\Theta\|_{L^{\infty}}\Big)
  + \|DA\|_{L^6}\Big( \| d (\bpartial D\rho_0)\|_{L^4} \|\frac{\Theta}{d}\|_{ L^{\infty}}+\|D\rho_0\|_{L^3}\|D\Theta\|_{L^{\infty}}\\
&\qquad+ \|\rho_0\|_{L^{\infty}} \|\bpartial D\Theta\|_{ L^4}\Big) 
  +\|  A\|_{ L^{\infty}}\Big(\|d (\bpartial D^2\rho_0)\|_{L^2} \|\frac{\Theta}{d}\|_{ L^{\infty}}+\|\bpartial D\rho_0\|_{L^2} \|D\Theta\|_{L^{\infty}}\\
  &\qquad+\|\bpartial\rho_0\|_{L^4}\|D^2\Theta\|_{L^4}+\|d (D^2\rho_0)\|_{L^2} \|\frac{\bpartial\Theta}{d}\|_{ L^{\infty}}+\|D\rho_0\|_{L^3} \|\bpartial D\Theta\|_{L^6}+\|\rho_0\|_{L^{\infty}}\|\bpartial D^2\Theta\|_{L^2}\Big)\bigg]^2\ dt',\\
&\lesssim \int_{0}^{t} \Big( t'P(\sup\limits_{[0,t']}F)+P(M_0)  \Big)\Big(\|\rho_0\|_{L^{\infty}}+\|\bpartial \rho_0\|_{H^1}+\|D\rho_0\|_{L^3}+\|d(D^2 \rho_0)\|_{L^2}+\|d(\bpartial D^2 \rho_0)\|_{L^2}\Big)^2\\
&\qquad\cdot\Big(\|\Theta\|_{H^3}+\|\bpartial D\Theta\|_{L^2}^{1/4}\|\bpartial D^3\Theta\|_{L^2}^{3/4}\Big)^2 \ dt'\\
&\lesssim \frac{1}{\epsilon}tP(\sup\limits_{t'\in[0,t]}F)+\frac{1}{\epsilon}P(M_0)+\epsilon \int_{0}^{t}\|\bpartial D^3\Theta_t\|_{L^2}^2 \ dt'.
\end{align*}
Note that $\Theta, \Theta_t=0$ on the boundary, and thus $\|\frac{\Theta}{d}, \frac{\bpartial\Theta}{d}\|_{ L^{\infty}}\lesssim \|D\Theta, \bpartial D\Theta\|_{ L^{\infty}}$. Here we have used the embedding inequalities $\|\cdot\|_{L^6}\lesssim \|\cdot\|_{H^1}$ and $\|\bpartial D\Theta\|_{ L^{\infty}}\lesssim  \|\bpartial D\Theta\|_{L^2}^{1/4}\|\bpartial D^3\Theta\|_{L^2}^{3/4} +\|\bpartial D\Theta\|_{L^2}$. Moreover, 
\begin{align*}
\mathrm{III}&\lesssim 
\int_{0}^{t}  \Big(\|\bpartial D^2(aA)\|_{L^2}\|Dv\|_{L^{\infty}}
+\|\bpartial D (aA)\|_{L^4}\|D^2 v\|_{L^4}
+\|\bpartial (aA)\|_{L^{\infty}}\|D^3 v\|_{L^2}\\
&\qquad
+\|D^2 (aA)\|_{L^2} \| \bpartial Dv \|_{L^{\infty}}
+\|D(aA)\|_{L^4} \|\bpartial D^2v\|_{L^4}\Big)^2 \ dt'\\
&\lesssim \int_{0}^{t} \Big( t'P(\sup\limits_{[0,t']}F)+P(M_0)  \Big)
\Big(\|v\|_{H^3}+\|\bpartial Dv\|_{L^2}^{1/4}\|\bpartial D^3v\|_{L^2}^{3/4}+\|\bpartial D^2v\|_{L^2}^{1/4}\|\bpartial D^3v\|_{L^2}^{3/4}\Big)^2 \ dt'\\
&\lesssim  \frac{1}{\epsilon}tP(\sup\limits_{t'\in[0,t]}F)+\frac{1}{\epsilon}P(M_0)+\epsilon \int_{0}^{t}\|\bpartial D^3v\|_{L^2}^2 \ dt'.
\end{align*}
Here we have used the embedding inequality
$\|\bpartial Dv\|_{L^{\infty}}\lesssim\|\bpartial Dv\|_{L^2}^{1/4}\|\bpartial D^3v\|_{L^2}^{3/4}+\|\bpartial Dv\|_{L^2}$ and $\|\bpartial D^2v\|_{L^4}\lesssim\|\bpartial D^2v\|_{L^2}^{1/4}\|\bpartial D^3v\|_{L^2}^{3/4}+\|\bpartial D^2v\|_{L^2}$. Also, we have
\begin{align*}
&\int_{0}^{t}\sum\limits_{i=1}^3\left\| a^r_j\Big[\mu (A^{k}_j  \bpartial Dv^i_{,rk}+A^k_i \bpartial Dv^j_{,rk})+\lambda A_l^k \bpartial Dv^l_{,rk}\delta^i_j\Big]   \right\|_{L^2}^2 \ dt'\\
\geq & \frac{1}{C}\int_{0}^{t}\|\bpartial \partial_3^3 v\|_{L^2}^2 \ dt'-CtP(\sup\limits_{t'\in[0,t]}F)-C\int_{0}^{t} \|\bpartial^2 v\|_{H^2}^2 \ dt'.
\end{align*}
So it follows from the estimates of $\mathrm{I},\mathrm{II},\mathrm{III}$ that
\begin{equation}
\label{dvxxx2}
\begin{aligned}
\int_{0}^{t}\| \bpartial v\|_{H^3}^2 \ dt' &\lesssim \frac{1}{\epsilon}tP(\sup\limits_{t'\in[0,t]}F)+\frac{1}{\epsilon}P(M_0)+\frac{1}{\epsilon}P\left(\int_{0}^{t}\|v_{t}\|_{H^2}^2 +\|\bpartial^2 v\|_{H^2}^2\ dt'\right)\\
&\qquad\qquad+\epsilon \int_{0}^{t}\|\bpartial D^3\Theta\|_{L^2}^2 \ dt'.
\end{aligned}
\end{equation}
Using the method to analysis $\bpartial^2 \Big(\eqref{1}_1\Big)$, we have that
\begin{equation}
\label{dvxxx1}
\begin{aligned}
\int_{0}^{t}\| \bpartial^2 v\|_{H^2}^2 \ dt' &\lesssim tP(\sup\limits_{t'\in[0,t]}F)+\frac{1}{\epsilon}P(M_0)+P\left(\int_{0}^{t}\|v_{t}\|_{H^3}^2 +\|\bpartial^3 v\|_{H^1}^2\ dt'\right)
 .
\end{aligned}
\end{equation}
Here $\|\bpartial D^3\Theta\|_{L^2}^2$ does not appear because $\|\bpartial^2 D(\rho_0 \Theta)\|_{L^2}\lesssim \|\bpartial D\rho_0\|_{L^2}\|\bpartial\Theta\|_{L^{\infty}}+\cdots \lesssim \|\bpartial D\rho_0\|_{L^2}\|\Theta\|_{H^3}+\cdots$.

\vspace{0.3cm}

\textsl{\textbf{Estimates of $\int_{0}^{t}\|\bpartial \Theta\|_{H^3}^2\ dt'$:}}~~
$\bpartial D\Big(\eqref{1}_2\Big)$ yields that
\begin{align*}
c_v\bpartial D(\rho_0 \Theta_t)+\bpartial D \left( \frac{R\rho_0\Theta}{J}a^r_iv^i_{,r}\right)=\bpartial D(\mathbb{S}_{\eta}^{ij}[v]a^r_jv^i_{,r})+\kappa \bpartial D (a^r_i A^k_j \delta^{ij}\Theta_{,k})_{,r}.
\end{align*}
Then
\begin{align*}
\int_{0}^{t}\sum\limits_{i=1}^3\left\|a^r_i A^k_j \delta^{ij} \bpartial D\Theta_{,rk} \right\|_{L^2}^2 \ dt'
\lesssim& \int_{0}^{t} \|\bpartial D(\rho_0 \Theta_t)\|_{L^2}^2
   +\| \bpartial D(A\rho_0\Theta\cdot Dv)\|_{L^2}^2\\
   +\| \bpartial D(a A \cdot |Dv|^2)\|_{L^2}^2&
   +\|\bpartial D^2(aA\cdot D\Theta)-(aA)\bpartial D^3 \Theta\|_{L^2}^2\ dt'\\
:=&\mathrm{I}+\mathrm{II}+\mathrm{III} +\mathrm{IV}.
\end{align*}
The terms in the right hand side are estimated as following
\begin{align*}
\mathrm{I}&\lesssim \int_{0}^{t} \|\Theta_t\|_{H^2}^2\ dt',\\
\mathrm{II}&\lesssim \int_{0}^{t} \Big[ 
    \|\bpartial D(A\rho_0\Theta)\|_{L^2} \|Dv\|_{L^{\infty}}
    +\|\bpartial (A\rho_0\Theta)\|_{L^4}\|D^2v\|_{L^4}\\
     &\qquad+ \| D(A\rho_0\Theta)\|_{L^3} \|D^2v\|_{L^6}
     +\| (A\rho_0\Theta)\|_{L^{\infty}}\|D^3v\|_{L^2} \Big]^2 \ dt'\\
&\lesssim \int_{0}^{t} \Big( t'P(\sup\limits_{[0,t']}F)+P(M_0)  \Big)
     \Big(\|\rho_0\|_{L^{\infty}}+\|D\rho_0\|_{L^3}+\|\bpartial \rho_0\|_{H^1}\Big)^2
     \Big(\|\Theta\|_{H^3}\|v\|_{H^3} \Big)^2\ dt'\\
&\lesssim tP(\sup\limits_{t'\in[0,t]}F)+P(M_0),
\end{align*}
\begin{align*}
\mathrm{III}&\lesssim \int_{0}^{t}  \Big[
      \|\bpartial D(aA)\|_{L^4} \|Dv\|_{L^{\infty}}^2
     +\| D (aA)\|_{L^4}  \|\bpartial D v\|_{L^4}\|Dv\|_{L^{\infty}}
     + \| \bpartial (aA)\|_{L^{\infty}}  \|D^2 v\|_{L^2}\|Dv\|_{L^{\infty}}\\
         &\qquad
     +\|(aA)\|_{L^{\infty}}\Big( \|\bpartial D^2 v\|_{L^2}\|Dv\|_{L^{\infty}}+\|\bpartial D v\|_{L^4}\|D^2v\|_{L^4}\Big)\Big]^2 \ dt'\\
&\lesssim  \int_{0}^{t} \Big( t'P(\sup\limits_{[0,t']}F)+P(M_0)  \Big)
      \|v\|_{H^3}^4\ dt'\\
&\lesssim tP(\sup\limits_{t'\in[0,t]}F)+P(M_0),\\
\mathrm{IV}
&\lesssim  \frac{1}{\epsilon}tP(\sup\limits_{t'\in[0,t]}F)+\frac{1}{\epsilon}P(M_0)+\epsilon \int_{0}^{t}\|\bpartial D^3\Theta\|_{L^2}^2 \ dt',
\end{align*}
Meanwhile,
\begin{align*}
\int_{0}^{t}\sum\limits_{i=1}^3\left\|a^r_iA^k_i\bpartial D\Theta_{,rk} \right\|_{L^2}^2 \ dt'
\geq  \frac{1}{C}\int_{0}^{t}\|\bpartial\partial_3^3 \Theta\|_{L^2}^2 \ dt'-CtP(\sup\limits_{t'\in[0,t]}F)-C\int_{0}^{t} \|\bpartial^2 \Theta\|_{H^2}^2 \ dt'.
\end{align*}
So it follows from the estimates of $\mathrm{I},\mathrm{II},\mathrm{III}$ that
\begin{equation}
\label{dthetaxxx2}
\begin{aligned}
\int_{0}^{t}\|\bpartial \Theta\|_{H^3}^2 \ dt' &\lesssim tP(\sup\limits_{t'\in[0,t]}F)+P(M_0)+P\left(\int_{0}^{t}\|\Theta_t\|_{H^2}^2 +\|\bpartial^2 \Theta\|_{H^2}^2\ dt'\right),
\end{aligned}
\end{equation}
\vspace{0.3cm}
Using the method to analysis $\bpartial^2  \Big(\eqref{1}_2\Big)$, we have that
\begin{equation}
\label{dthetaxxx1}
\begin{aligned}
\int_{0}^{t}\| \bpartial^2 \Theta\|_{H^2}^2 \ dt' &\lesssim tP(\sup\limits_{t'\in[0,t]}F)+\frac{1}{\epsilon}P(M_0)+P\left(\int_{0}^{t}\|\Theta_t\|_{H^2}^2 +\|\bpartial^3 \Theta\|_{H_0^1}^2\ dt'\right).
\end{aligned}
\end{equation}

\vspace{0.3cm}

\vspace{0.3cm}

\begin{proof}[\textbf{Proof of Proposition \ref{prop-of-estimates}}]
	\eqref{a priori estimates} can be closed by combining the inequalities from \eqref{direct} to \eqref{dthetaxxx1}, and choosing $\epsilon$ small enough.

\end{proof}

\vspace{0.7cm}
\section{Existence and uniqueness proofs}
Based on the a priori estimate \eqref{a priori estimates}, in this section we will prove the local existence and uniqueness of the strong solutions to the free boundary problem \eqref{1} can be proved by a fixed point argument.

\vspace{0.3cm}
Let $V_T$ be the space of the completion of $C^{\infty}(\overline{\Omega}\times[0,T])$ under the norm
\begin{equation}\label{def_VT}
\begin{aligned}
\sup\limits_{t\in[0,T]}\Big(\|\rho_{0}^{1/2}v\|_{L^{2}}^2+\|\rho_{0}^{1/2}\Theta\|_{L^{2}}^2\Big)+\int_{0}^{T}\|v\|_{H^1}^2 +\|\Theta\|_{H^1}^2\ dt
\end{aligned}
\end{equation}
and $X_T$ be the subset of $V_T$ defined by
\begin{equation}\label{def_XT}
\begin{aligned}
X_T&=\{ (v,\Theta)\in V_T\large|~ v\in L^{\infty}(0,T;H^3),\Theta\in L^{\infty}(0,T;H^3\cap H_0^1),)\\
& ~~v_t\in L^{\infty}(0,T;H^1)\cap L^2(0,T;H^3),~ \Theta_t\in L^{\infty}(0,T;H^1_0)\cap L^2(0,T;H^3),\\
&~~ v_{tt}\in L^2(0,T;H^1), ~\Theta_{tt}\in L^2(0,T;H^1),~\sup\limits_{t\in[0,T]}F(v,\Theta) \leq M_1\}, 
\end{aligned}
\end{equation}
for some $M_1>0$.
Therefore, $X_T$ is the compact and convex subset of $V_T$.

Now we define $\Xi: X_T\rightarrow X_T$ as the solving operator as follows: for any $(\widetilde{v},\widetilde{\Theta})\in X_T$, $\Xi(\widetilde{v}, \widetilde{\Theta})$ is the solution to the following equation
\begin{equation}
\label{2}
\left\{
\begin{aligned}
& \rho_0 v^i_t+\widetilde{a^r_i}(\frac{R\rho_0 \widetilde{\Theta}}{\widetilde{J}})_{,r} = \widetilde{a^r_j} \mathbb{S}^{ij}_{\widetilde{\eta}}[v]_{,r} &&\text{in}\ \Omega\times(0,T],\\
& c_v \rho_0\Theta_t +\frac{R\rho_0\widetilde{\Theta}}{\widetilde{J}}\widetilde{a^r_i}v^i_{,r}=  \mathbb{S}^{ij}_{\widetilde{\eta}}[v]\widetilde{a^r_j}v^i_{,r}
+\kappa \widetilde{a^r_i}(\nabla_{\widetilde{\eta}}\Theta)^i_{,r} &&\text{in}\ \Omega\times(0,T],\\
& \Theta=0,\ \widetilde{a_j^3} \mathbb{S}_{\widetilde{\eta}}^{ij}[v]=0 &&\text{on}\ \Gamma\times(0,T],\\
& (v,\Theta)=(u_0,\theta_0) && \text{on}\ \Omega\times\{t=0 \},
\end{aligned}
\right. 
\end{equation}
where $\widetilde{\eta}$ is the flow map determined by $\widetilde{v}$ and the ODE \eqref{def-eta}, and  $\widetilde{A},~\widetilde{a}$ are the corresponding matrices of $\widetilde{\eta}$ given similarly as in Section 2.1.


In the following, we will prove that when $T$ is sufficiently small, $\Xi:X_T\rightarrow X_T$ is well-defined, and also a contraction mapping in $X_T$. Therefore, there exists the unique fixed point of $\Xi$, which is exactly the strong solution to \eqref{1}.

\vspace{0.3cm}


\subsection{Solutions to the parabolic equations \eqref{2}}~

\vspace{0.3cm}
In this subsection, we will prove that $\Xi:X_T\rightarrow X_T$ is well-defined provided $T$ is sufficiently small. We split the argument into two steps. First, existence and uniqueness of the weak solutions to \eqref{2} are proved; in step 2, we will verify the boundary conditions, and improve the regularity of the solutions.

\textbf{Step 1: Existence and uniqueness of the weak solutions to the parabolic equations \eqref{2}.} 

\begin{Def}\label{def_weak}
	$(v,\Theta)\in L^2(0,T; H^1)\times L^2(0,T;H^1_0)$ with $(\rho_0v_t,\rho_0\Theta_t)\in L^2(0,T; H^{1*})\times L^2(0,T;H^{-1})$ is called the weak solution to \eqref{2} provided 
	for a.e. $t\in (0,T)$ and arbitrarily $\phi^i\in C^{\infty}(\overline{\Omega}\times[0,T]),~\psi\in C^{\infty}([0,T];C^{\infty}_c(\Omega))$,  $i=1,2,3$, we have
\begin{equation}
\label{weak}
\left\{\begin{aligned}
&\Big(\rho_0v^i (\cdot, t)~, \phi^i(\cdot,t) \Big)
   -\int_{0}^{t}\Big(\rho_0v^i~, \phi^i_t \Big) \ dt'
   +\int_{0}^{t}\Big(\widetilde{a^r_j}  \mathbb{S}^{ij}_{\widetilde{\eta}}[v]~, \phi^i_{,r}\Big)\ dt'\\
 &\qquad\qquad  =\Big(\rho_0u^i_0~, \phi^i(\cdot,0) \Big)
   +\int_{0}^{t} \Big(\widetilde{a^r_i}(\frac{R\rho_0 \widetilde{\Theta}}{\widetilde{J}})~, \phi^i_{,r}\Big)\ dt',\\
&\Big(c_v \rho_0\Theta(\cdot,t)~,\psi(\cdot,t)\Big)
   -\int_{0}^{t} \Big(c_v \rho_0\Theta~,\psi_t\Big) \ dt'
   +\kappa \int_{0}^{t} \Big(\widetilde{a^r_i}(\nabla_{\widetilde{\eta}}\Theta)^i~, \psi_{,r}\Big)\ dt'\\ 
&\qquad\qquad =\Big(c_v \rho_0\theta_0~,\psi(\cdot,0)\Big)
    -\int_{0}^{t}\Big(\frac{R\rho_0\widetilde{\Theta}}{\widetilde{J}}\widetilde{a^r_i}v^i_{,r}~,\psi\Big)\ dt' \\
 &\qquad\qquad\qquad\qquad +\int_{0}^{t}\Big(\mathbb{S}^{ij}_{\widetilde{\eta}}[v]\widetilde{a^r_j}v^i_{,r}~, \psi\Big)\ dt',
\end{aligned}
\right.
\end{equation}
where $\Big(f~,g\Big):=\int_{\Omega} f(x)g(x)\ dx$, $H^{1*}$ is the dual space of $H^1$ and $C^{\infty}_c(\Omega)$ is the set of smooth functions with compact supports in the open set $\Omega$.
\end{Def}

\begin{Lem}
	Assume that $\rho_0^{1/2} u_0\in L^2$ and  $\rho_0^{1/2} \theta_0\in L^2$. 
	Then there exists a small enough $T>0$ and a unique weak solution $(v,\Theta) \in L^2(0,T; H^1)\times L^2(0,T;H^1_0)$ with $(\rho_0v_t,\rho_0\Theta_t)\in L^2(0,T; H^{1*})\times L^2(0,T;H^{-1})$ to \eqref{2}.
	Moreover, 
	\begin{align*}
	&	\sup\limits_{t\in[0,T]} (\|\rho_0^{1/2} v\|_{L^2}^2+\|\rho_0^{1/2} \Theta\|_{L^2}^2) +\int_{0}^{T}\|v\|_{H^1}^2+\|\Theta\|_{H^1}^2 dt 
	\\
	&\qquad\qquad\qquad\qquad\qquad\lesssim \|\rho_0^{1/2} u_0\|_{L^2}^2+\|\rho_0^{1/2} \Theta_0\|_{L^2}^2+TP(M_1).
	\end{align*}
\end{Lem}
\begin{proof}
	We deal with the equation of $v$ $\eqref{1}_1$ first.
	Suppose that $\{w_s\}_{s=1}^{\infty}$ are the linearly independent basis of $H^1$ and the orthonormal basis of $L^2$.
	We define the Galerkin approximation at order $1\leq m<\infty$, and let the function $v^i_m:[0,T]\rightarrow H^1$ ($i=1,2,3$) of the form
	\begin{equation}
	\label{ode1}
	v^i_m(t)=\sum\limits_{s=1}^{m} d^i_{ms}(t)w_s,
	\end{equation}
	which satisfies that for any $s=\{1,2,\cdots,m\}$,
	\begin{align}
	\Big(\rho_0 (v^i_m)_t~, w_s \Big) 
	+ \Big(\widetilde{a^r_j}\mathbb{S}^{ij}_{\widetilde{\eta}}[v_m]~, w_{s,r} \Big)
	&=\Big(\widetilde{a^r_i} \frac{R\rho_0 \widetilde{\Theta}}{\widetilde{J}}~, w_{s,r} \Big),
	\label{ode2}\\
	d^i_{ms}(0)&=\Big(u^i_0~, w_s\Big) \label{ode3}.
	\end{align}
	Substituting \eqref{ode1} into \eqref{ode2} gives the following ODE system
	\begin{equation}
	\label{ode4}
	\begin{aligned}
	&\sum\limits_{l=1}^{m}\Big\{ 
	\Big(\rho_0 w_l~, w_s \Big) \big(d^i_{ml}(t)\big)'
	+ \Big(\widetilde{a^r_j} \mu \widetilde{A^k_j} w_{l,k}~, w_{s,r} \Big) d^i_{ml}(t)
	+ \Big(\widetilde{a^r_j} \mu \widetilde{A^k_i} w_{l,k}~, w_{s,r} \Big) d^j_{ml}(t)\\
	& \qquad
	+ \Big(\widetilde{a^r_i} \lambda \widetilde{A^k_j} w_{l,k}~, w_{s,r} \Big) d^j_{ml}(t) \Big\}
	=\Big(\widetilde{a^r_i} \frac{R\rho_0 \widetilde{\Theta}}{\widetilde{J}}~, w_{s,r} \Big) \qquad\qquad s=1,2,\cdots,m,~i=1,2,3.
	\end{aligned}
	\end{equation}
	Since $\{w_s\}$ are linearly independent, so are $\{\rho_0^{1/2} w_s\}$.
	Then the matrix
	\begin{equation*}
	\left(\Big(\rho_0w_l,w_s\Big)\right)_{m\times m}
	\end{equation*}
	is invertible, and so is the matrix
	\begin{equation*}
	\begin{pmatrix}
	\left(\Big(\rho_0w_l,w_s\Big)\right)_{m\times m} &0 &0\\
	0 &\left(\Big(\rho_0w_l,w_s\Big)\right)_{m\times m} &0\\
	0 &0 &\left(\Big(\rho_0w_l,w_s\Big)\right)_{m\times m}\\
	\end{pmatrix}.
	\end{equation*}
	Therefore, according to the standard theory of ODE, there exists a unique absolutely continuous function $d^i_{ms}(t)$ ($s=1,2,\cdots,m,~i=1,2,3$) satisfying \eqref{ode4} and \eqref{ode3}.
	Then $v^i_m(t)$ defined by \eqref{ode1} solves \eqref{ode2} and \eqref{ode3} on a time interval $[0,T_m]$.
	
	Now multiplying \eqref{ode2} by $d^i_{ms}$ and summing over all $s$ and $i$, we have
	\begin{equation*}
	\Big(\rho_0 (v^i_m)_t~, v^i_m \Big) 
    + \Big(\widetilde{a^r_j}\mathbb{S}^{ij}_{\widetilde{\eta}}[v_m]~, v^i_{m,r} \Big)
    =\Big(\widetilde{a^r_i} \frac{R\rho_0 \widetilde{\Theta}}{\widetilde{J}}~, v^i_{m,r} \Big),
	\end{equation*}
	\begin{align*}
	\qquad \Big(\widetilde{a^r_j}\mathbb{S}^{ij}_{\widetilde{\eta}}[v_m]~, v^i_{m,r} \Big)&=\Big(\widetilde{a^r_j}[\mu (\widetilde{A^{k}_j }v^i_{,k}+\widetilde{A^k_i} v^j_{,k})+\lambda (\widetilde{A_l^k} v^l_{,k} )\delta^i_j]~, v^i_{m,r} \Big)\\
	&\geq \frac{1}{C} \int_{\Omega} \sum\limits_{i,j}\Big(\widetilde{A^k_j} v^i_{m,k}+\widetilde{A^k_i}  v^j_{m,k}\Big)^2\\
	&\geq\frac{1}{C}\int_{\Omega}\sum\limits_{i,j}( v^i_{m,j}+v^j_{m,i})^2 
	-C|\widetilde{A}-\delta| |Dv_m|^2\\
	&\geq\frac{1}{C}\|v_m\|_{H^1}^2-C\|\rho_0^{1/2}v_m\|_{L^2}^2-CT_mM_1^{1/2} \|v_m\|_{H^1}^2,\\
	\left|\Big(\widetilde{a^r_i} \frac{R\rho_0 \widetilde{\Theta}}{\widetilde{J}}~, v^i_{m,r} \Big) \right|
	&\lesssim \frac{1}{\epsilon}\|\widetilde{a^r_i} \frac{R\rho_0 \widetilde{\Theta}}{\widetilde{J}}\|_{L^{\infty}}+\epsilon \|v_m\|_{H^1}^2
	\lesssim \frac{1}{\epsilon}P(M_1)+\epsilon \|v_m\|_{H^1}^2.
	\end{align*}
	Here $\| \widetilde{A}-\delta\|_{L^{\infty}}^2
	\lesssim  \Big(T_m\sup\limits_{t'\in[0,t]}\|\partial_t \widetilde{A}\|_{L^{\infty}}\Big)^2
	\lesssim \Big(T_m\sup\limits_{t\in[0,T_m]}\|D \widetilde{v}\|_{L^{\infty}}\Big)^2
	\lesssim T_m^2 M_1$.
	So by choosing $T_m$(independent of $u_0,\theta_0,\widetilde{v},\widetilde{\Theta}$) and $\epsilon$ small enough, 
	integrating it with respect to $t$ and using Gr\"{o}nwall's inequality, we have
	\begin{equation*}
	\sup\limits_{t\in[0,T_m]} \|\rho_0^{1/2} v_m(\cdot, t)\|_{L^2}^2 +\int_{0}^{T_m}\|v_m\|_{H^1}^2\ dt 
	\lesssim \|\rho_0^{1/2} u_0\|_{L^2}^2+T_mP(M_1).
	\end{equation*}
	Since the above inequality is uniformly bounded if $T_m\leq T$ for some $T>0$ independent of $u_0,\theta_0,\widetilde{v},\widetilde{\Theta}$, the existence time can be chosen uniformly, and we denote it just as $T$.
    Therefore, there exists a subsequence of $v_m$ that converges weakly to some $v\in L^2(0,T;H^1)$ which satisfies 
    \begin{equation*}
	\sup\limits_{t\in[0,T]} \|\rho_0^{1/2} v(\cdot, t)\|_{L^2}^2 +\int_{0}^{T}\|v\|_{H^1}^2 dt 
	\lesssim \|\rho_0^{1/2} u_0\|_{L^2}^2+TP(M_1).
	\end{equation*}
    It follows from the standard arguments of the parabolic equations theory that $v$ is a solution of $\eqref{2}_1$ in the sense of distribution for any test functions in $L^2(0,T;H^1)$.
    Further more,  the equation implies that $\rho_0v_t\in L^2(0,T;H^{1*})$ and similar arguments as obtaining the energy inequality implies that $v$ is unique.

    The arguments for $\Theta$ are the same except that a different basis is used: $\{\omega_s \}_{s=1}^{\infty}$ are the linearly independent basis of $H^1_0$ and the orthonormal basis of $L^2$. 
    The rest of the proof is omitted.
\end{proof}

\vspace{0.3cm}
\textbf{Step 2: Verification of the boundary conditions and the improved regularities of the solutions.} 

\begin{Lem}
	Suppose that $(v,\Theta)$ is the weak solution to \eqref{2} in the Definition \ref{def_weak}, and  $v\in L^2(0,T;H^2)$ and $\Theta \in L^2(0,T;H^2\cap H^1_0)$,
	then $(v,\Theta)$ satisfies the boundary condition $\eqref{2}_3$.
\end{Lem}
\begin{proof}
	$\Theta=0$ on $\Gamma\times (0,T]$ can be directly obtained by $\Theta \in L^2(0,T;H^2\cap H^1_0)$.
	
    Since \eqref{weak} holds for any compacted supported smooth test functions,
	then when $v$ and $\Theta$ are regular, $v$ and $\Theta$ satisfy the equation \eqref{2} in $\Omega$ everywhere.
	So $\widetilde{a_j^3} \mathbb{S}_{\widetilde{\eta}}^{ij}[v]=0$ on $\Gamma\times (0,T]$ is verified by multiplying $\eqref{2}_1$ by an arbitrary test function $\phi=(\phi^1, \phi^2, \phi^3)$, where  $\phi^i\in C^{\infty}(\overline{\Omega}\times[0,T])$, integrating by parts and then comparing the result with $\eqref{weak}_1$.
\end{proof}

Having obtained the above two lemmas, by similar arguments to prove the a priori estimates Proposition \ref{prop-of-estimates}, and the standard arguments to improve the regularities of the solutions obtained by the Galerkin's method, 
one has the following:

\begin{Lem}
	Assume that $u_0, \theta_0$ satisfy the regularity $M_0=E(u_0, \theta_0)+1<\infty$, and the  compatible condition \eqref{compatible}
	Then the weak solution $(v,\Theta)$ is regular and satisfies
	\begin{equation*}
	\sup\limits_{t\in[0,T]}F(v,\Theta)\lesssim  P(M_0)+TP(M_1).
	\end{equation*}
\end{Lem}

Note that here the a priori assumption \eqref{a priori assumption} is verified by replacing $2P(M_0)$ in Remark \ref{rem-assumption} as $M_1$ and choosing $T$ small enough.
So when $P(M_0)\leq M_1$ and $T$ is small enough, given any $(\widetilde{v},\widetilde{\Theta})\in X_T$, there exists a unique $(v,\Theta)\in X_T$ which satisfies \eqref{2}. 
Therefore, $\Xi: X_T\rightarrow X_T$ is well-defined.

\vspace{0.3cm}

\subsection{Contraction mapping}~ 

\vspace{0.3cm} 
In this subsection, we will prove that 
$\Xi$ is actually a contraction mapping in $X_T$ under the norm \eqref{def_VT}, and then finish the proof of Theorem \ref{mainthm}.

\vspace{0.3cm}
For any $(\widetilde{v}^{(n)}, \widetilde{\Theta}^{(n)}) \in X_T$, $n=1,2$, denote that $(v^{(n)},\Theta^{(n)}):=\Xi(\widetilde{v}^{(n)}, \widetilde{\Theta}^{(n)})$.
Moreover, we write 
where $\widetilde{\eta}^{(n)}$ is the flow map determined by $\widetilde{v}^{(n)}$, and $\widetilde{A}^{(n)}$, $\widetilde{a}^{(n)}$ are the corresponding matrices of $\widetilde{\eta}^{(n)}$ given similarly as in Section 2.1. Here for simplicity, symbols like  $(\widetilde{a}^{(n)})^r_i$ are written as $\widetilde{a}^{(n)r}_i$. Then we estimate the norm of $((v^{(1)},\Theta^{(1)})-(v^{(2)},\Theta^{(2)}))$ by the basic energy estimates.

Integrating the difference of the velocity equations of $((v^{(n)},\Theta^{(n)})$ \eqref{2} ($n=1,2$) dot product by $(v^{(1)i}-v^{(2)i})$ yields the following
\begin{align*}
&\frac{1}{2}\int_{\Omega}\rho_{0}\Big| v^{(1)}-v^{(2)} \Big|^2 \\
&\qquad -\underbrace{\iint \Big(\widetilde{a}^{(1)r}_i \frac{R\rho_{0}\widetilde{\Theta}^{(1)} }{\widetilde{J}^{(1)}}-\widetilde{a}^{(2)r}_i \frac{R\rho_{0}\widetilde{\Theta}^{(2)} }{\widetilde{J}^{(2)}}  \Big) \Big(v^{(1)i}_{,r}-v^{(2)i}_{,r}\Big)}_{\mathrm{I}}\\
&\qquad +\underbrace{\iint \Big( \widetilde{a}^{(1)r}_j\mathbb{S}^{ij}_{\widetilde{\eta}^{(1)}}[v^{(1)}]_{,r}-\widetilde{a}^{(2)r}_j\mathbb{S}^{ij}_{\widetilde{\eta}^{(2)}}[v^{(2)}]_{,r} \Big)\Big(v^{(1)i}_{,r}-v^{(2)i}_{,r}\Big)}_{\mathrm{II}}=0,
\end{align*}
where the first term of the above equation at $t=0$ equals to $0$ and thus is omitted, since  $(v^{(n)},\Theta^{(n)})$, $(n=1,2)$ share the same initial data. 
\begin{align*}
\mathrm{I}&=\iint \Big(\widetilde{A}^{(1)r}_i -\widetilde{A}^{(2)r}_i \Big)R\rho_{0}\widetilde{\Theta}^{(1)} \Big(v^{(1)i}_{,r}-v^{(2)i}_{,r}\Big)\\
 &\qquad +\iint \widetilde{A}^{(2)r}_iR\rho_{0}\Big(\widetilde{\Theta}^{(1)}-\widetilde{\Theta}^{(2)}  \Big) \Big(v^{(1)i}_{,r}-v^{(2)i}_{,r}\Big)\\
 &:= \mathrm{I}_1+\mathrm{I}_2,\\
|\mathrm{I}_1|&\lesssim\int_{0}^{t}\|\widetilde{A}^{(1)}-\widetilde{A}^{(2)}\|_{L^2}\|R\rho_{0}\widetilde{\Theta}^{(2)}\|_{L^{\infty}}\|Dv^{(1)}-Dv^{(2)}\|_{L^2} \ dt'\\
 &\lesssim \int_{0}^{t} \Big( \int_{0}^{t'} \|D\widetilde{v}^{(1)}-D\widetilde{v}^{(2)}\|_{L^2} \ d\tau \Big) \|\widetilde{\Theta}^{(2)}\|_{H^2}\|Dv^{(1)}-Dv^{(2)}\|_{L^2}\ dt'\\
 &\lesssim \int_{0}^{t} t'^{1/2} \Big(\int_{0}^{t'} \|\widetilde{v}^{(1)}-\widetilde{v}^{(2)}\|_{H^1}^2 \ d\tau \Big)^{1/2} M_1^{1/2}\|v^{(1)}-v^{(2)}\|_{H^1}\ dt'\\
 &\lesssim \frac{M_1}{\epsilon}t^2 \int_{0}^{t} \|\widetilde{v}^{(1)}-\widetilde{v}^{(2)}\|_{H^1}^2 \ dt' +\epsilon\int_{0}^{t} \|v^{(1)}-v^{(2)}\|_{H^1}^2 \ dt',\\
|\mathrm{I}_2|&\lesssim \int_{0}^{t} \|\widetilde{A}^{(2)}R\rho_{0}^{1/2}\|_{L^{\infty}}\|\rho_{0}^{1/2}(\widetilde{\Theta}^{(1)}-\widetilde{\Theta}^{(2)})\|_{L^{2}}\|Dv^{(1)}-Dv^{(2)}\|_{L^2} \ dt'\\
 &\lesssim \frac{1}{\epsilon}t\sup\limits_{t'\in[0,t]}\|\rho_{0}^{1/2}(\widetilde{\Theta}^{(1)}-\widetilde{\Theta}^{(2)})\|_{L^{2}}^2 +\epsilon\int_{0}^{t} \|v^{(1)}-v^{(2)}\|_{H^1}^2 \ dt',
\end{align*}
where $\|\widetilde{A}^{(1)}-\widetilde{A}^{(2)}\|_{L^2}  \int_{0}^{t'} \|D\widetilde{v}^{(1)}-D\widetilde{v}^{(2)}\|_{L^2} \ d\tau $ since the initial mapping of $\widetilde{\eta}^{(1)}, \widetilde{\eta}^{(2)}$ are both the identity map.
\begin{align*}
\mathrm{II}&=\iint \Big( \widetilde{a}^{(1)r}_j-\widetilde{a}^{(2)r}_j\Big)\mathbb{S}^{ij}_{\widetilde{\eta}^{(1)}}[v^{(1)}]\Big(v^{(1)i}_{,r}-v^{(2)i}_{,r}\Big) \\
 &\qquad+\iint \widetilde{a}^{(2)r}_j\Big(\mathbb{S}^{ij}_{\widetilde{\eta}^{(1)}}[v^{(1)}]-\mathbb{S}^{ij}_{\widetilde{\eta}^{(2)}}[v^{(1)}]\Big)\Big(v^{(1)i}_{,r}-v^{(2)i}_{,r}\Big) \\
 &\qquad+\iint \widetilde{a}^{(2)r}_j\Big(\mathbb{S}^{ij}_{\widetilde{\eta}^{(2)}}[v^{(1)}]-\mathbb{S}^{ij}_{\widetilde{\eta}^{(2)}}[v^{(2)}]\Big)\Big(v^{(1)i}_{,r}-v^{(2)i}_{,r}\Big)\\
 &:=\mathrm{II}_1+\mathrm{II}_2+\mathrm{II}_3,\\
|\mathrm{II}_1|+|\mathrm{II}_2|&\lesssim \int_{0}^{t} \|\widetilde{A}^{(1)}-\widetilde{A}^{(2)}\|_{L^2}\|v^{(1)}\|_{L^{\infty}}\|Dv^{(1)}-Dv^{(2)}\|_{L^{2}} \ dt'\\
 &\lesssim \int_{0}^{t} t'^{1/2} \Big(\int_{0}^{t'} \|\widetilde{v}^{(1)}-\widetilde{v}^{(2)}\|_{H^1}^2\ d\tau \Big)^{1/2} M_1^{1/2}\|v^{(1)}-v^{(2)}\|_{H^1} \ dt'\\
 &\lesssim \frac{M_1}{\epsilon}t^2 \int_{0}^{t} \|\widetilde{v}^{(1)}-\widetilde{v}^{(2)}\|_{H^1}^2 \ dt' +\epsilon\int_{0}^{t} \|v^{(n1)}-v^{(2)}\|_{H^1}^2 \ dt',\\
\mathrm{II}_3&\geq \frac{1}{C}\int_{0}^{t} \|v^{(1)}-v^{(2)}\|_{H^1}^2-C\|\rho_{0}^{1/2}(v^{(1)}-v^{(2)})\|_{L^{2}}^2\\
 &\qquad-C\|\widetilde{A}^{(2)}-\delta\|_{L^{\infty}}\|v^{(1)}-v^{(2)}\|_{H^1}^2 \ dt' \\
 &\geq \frac{1}{C}\int_{0}^{t} \|v^{(1)}-v^{(2)}\|_{H^1}^2\ dt'-Ct\sup\limits_{t'\in[0,t]}\|\rho_{0}^{1/2}(v^{(1)}-v^{(2)})\|_{L^{2}}^2\\
 &\qquad-C t^{1/2}M_1^{1/2}\int_{0}^{t}\|v^{(1)}-v^{(2)}\|_{H^1}^2 \ dt'.
\end{align*}
So when choosing $t$ small enough, one has
\begin{equation}\label{contraction1}
\begin{aligned}
&\sup\limits_{t'\in[0,t]}\|\rho_{0}^{1/2}(v^{(1)}-v^{(2)})\|_{L^{2}}^2+\int_{0}^{t} \|v^{(1)}-v^{(2)}\|_{H^1}^2\ dt'\lesssim\\
&\qquad\qquad\qquad t\left\{\sup\limits_{t'\in[0,t]}\|\rho_{0}^{1/2}(\widetilde{\Theta}^{(1)}-\widetilde{\Theta}^{(2)})\|_{L^{2}}^2+\int_{0}^{t} \|\widetilde{v}^{(1)}-\widetilde{v}^{(2)}\|_{H^1}^2\ dt'\right\}.
\end{aligned}
\end{equation}

Integrating the difference of the temperature equations $((v^{(n)},\Theta^{(n)})$ \eqref{2} ($n=1,2$) multiplied by $(\Theta^{(1)}-\Theta^{(2)})$ yields the following
\begin{align*}
& \frac{c_v}{2}\int_{\Omega}\rho_0\Big| \Theta^{(1)}-\Theta^{(2)} \Big|^2 \\
&\qquad+\underbrace{\iint   \Big(\frac{R\rho_0\widetilde{\Theta}^{(1)}}{\widetilde{J}^{(1)}}\widetilde{a}^{(1)r}_iv^{(1)i}_{,r}-\frac{R\rho_0\widetilde{\Theta}^{(2)}}{\widetilde{J}^{(2)}}\widetilde{a}^{(2)r}_iv^{(2)i}_{,r}\Big) \Big( \Theta^{(1)}-\Theta^{(2)}\Big)}_{\mathrm{I}}\\
&\qquad-\underbrace{\iint  \Big(\mathbb{S}^{ij}_{\widetilde{\eta}^{(1)}}[v^{(1)}]\widetilde{a}^{(1)r}_jv^{(1)i}_{,r}-\mathbb{S}^{ij}_{\widetilde{\eta}^{(2)}}[v^{(2)}]\widetilde{a}^{(2)r}_jv^{(2)i}_{,r}\Big)\Big( \Theta^{(1)}-\Theta^{(2)}\Big)}_{\mathrm{II}}\\
&+\underbrace{\kappa\iint \Big(\widetilde{a}^{(1)r}_i\widetilde{A}^{(1)k}_j \delta^{ij} \Theta^{(1)}_{,k}-\widetilde{a}^{(2)r}_i\widetilde{A}^{(2)k}_j \delta^{ij} \Theta^{(2)}_{,k}\Big)\Big( \Theta^{(1)}_{,r}-\Theta^{(2)}_{,r}\Big)}_{\mathrm{III}}=0.
\end{align*}
Then estimate the terms $\mathrm{I},\mathrm{II},\mathrm{III}$ one by one:
\begin{align*}
\mathrm{I}&=\iint R\rho_0\Big(\widetilde{\Theta}^{(1)}-\widetilde{\Theta}^{(2)}\Big)\widetilde{A}^{(1)r}_iv^{(1)i}_{,r} \Big( \Theta^{(1)}-\Theta^{(2)}\Big)\\
 &\qquad +\iint R\rho_0\widetilde{\Theta}^{(2)}\Big(\widetilde{A}^{(1)r}_i-\widetilde{A}^{(2)r}_i\Big)v^{(1)i}_{,r} \Big( \Theta^{(1)}-\Theta^{(2)}\Big)\\
 &\qquad +\iint R\rho_0\widetilde{\Theta}^{(2)}\widetilde{A}^{(2)r}_i\Big(v^{(1)i}_{,r} -v^{(2)i}_{,r} \Big)\Big( \Theta^{(1)}-\Theta^{(2)}\Big)\\
 &:=\mathrm{I}_1+ \mathrm{I}_2+\mathrm{I}_3,\\
|\mathrm{I}_1|&\lesssim \int_{0}^{t}\|R\rho_0^{1/2}\|_{L^{\infty}}\|\rho_0^{1/2}(\widetilde{\Theta}^{(1)}-\widetilde{\Theta}^{(2)})\|_{L^2}\|\widetilde{A}^{(1)}Dv^{(1)}\|_{L^{\infty}} \|\Theta^{(1)}-\Theta^{(2)}\|_{L^{2}} \ dt'\\
 &\lesssim \int_{0}^{t}\|\rho_0^{1/2}(\widetilde{\Theta}^{(1)}-\widetilde{\Theta}^{(2)})\|_{L^2}M_1^{1/2}\|\Theta^{(1)}-\Theta^{(2)}\|_{L^{2}} \ dt'\\
 &\lesssim \frac{M_1}{\epsilon}t\sup\limits_{t'\in[0,t]}\|\rho_{0}^{1/2}(\widetilde{\Theta}^{(1)}-\widetilde{\Theta}^{(2)})\|_{L^{2}}^2 +\epsilon\int_{0}^{t} \|\Theta^{(1)}-\Theta^{(2)}\|_{H^1}^2 \ dt',\\
 |\mathrm{I}_2|&\lesssim \int_{0}^{t}\|R\rho_0\widetilde{\Theta}^{(1)}\|_{L^{\infty}} \|\widetilde{A}^{(1)}-\widetilde{A}^{(2)}\|_{L^2}\|Dv^{(1)}\|_{L^{\infty}} \|\Theta^{(1)}-\Theta^{(2)}\|_{L^{2}} \ dt'\\
 &\lesssim \int_{0}^{t}M_1^{1/2} t'^{1/2} \Big(\int_{0}^{t'} \|\widetilde{v}^{(1)}-\widetilde{v}^{(2)}\|_{H^1}^2\ d\tau \Big)^{1/2} M_1^{1/2} \|\Theta^{(1)}-\Theta^{(2)}\|_{L^{2}} \ dt'\\
 &\lesssim \frac{M_1^2}{\epsilon}t^2\int_{0}^{t} \|\widetilde{v}^{(1)}-\widetilde{v}^{(2)}\|_{H^1}^2 \ dt' +\epsilon\int_{0}^{t} \|\Theta^{(1)}-\Theta^{(2)}\|_{H^1}^2 \ dt',\\
|\mathrm{I}_3|&\lesssim
\int_{0}^{t} \|R\rho_0\widetilde{\Theta}^{(1)}\widetilde{A}^{(1)}\|_{L^{\infty} } \|Dv^{(1)}-Dv^{(2)} \|_{L^2} \|\Theta^{(1)}-\Theta^{(2)}\|_{L^2}\ dt'\\
 &\lesssim \frac{M_1}{\epsilon}\int_{0}^{t}  \|v^{(1)}-v^{(2)} \|_{H^1}^2\ dt'+\epsilon\int_{0}^{t}\|\Theta^{(1)}-\Theta^{(2)} \|_{H^1}^2 \ dt',\\
|\mathrm{II}|&\lesssim \int_{0}^{t} M_1^{1/2} \|Dv^{(1)}-Dv^{(2)}\|_{L^{2}}\|\Theta^{(1)}-\Theta^{(2)}\|_{L^2} \ dt'\\
&\qquad+\int_{0}^{t} M_1 \|\widetilde{A}^{(1)}-\widetilde{A}^{(2)}\|_{L^{2}}\|\Theta^{(1)}-\Theta^{(2)}\|_{L^2} \ dt'\\
&\lesssim \frac{M_1}{\epsilon}\int_{0}^{t}  \|v^{(1)}-v^{(2)} \|_{H^1}^2\ dt'+\frac{M_1^2}{\epsilon}t^2\int_{0}^{t} \|\widetilde{v}^{(1)}-\widetilde{v}^{(2)}\|_{H^1}^2 \ dt'\\
&\qquad+\epsilon\int_{0}^{t}\|\Theta^{(1)}-\Theta^{(2)} \|_{H^1}^2 \ dt',
\end{align*}
\begin{align*}
\mathrm{III}&=\kappa\iint \Big(\widetilde{a}^{(1)r}_i-\widetilde{a}^{(2)r}_i\Big)\widetilde{A}^{(1)k}_j \delta^{ij} \Theta^{(1)}_{,k}\Big( \Theta^{(1)}_{,r}-\Theta^{(2)}_{,r}\Big)\\
 &\qquad+\kappa\iint \widetilde{a}^{(2)r}_i\Big(\widetilde{A}^{(1)k}_j -\widetilde{A}^{(2)k}_j \Big)\delta^{ij} \Theta^{(1)}_{,k}\Big( \Theta^{(1)}_{,r}-\Theta^{(2)}_{,r}\Big)\\
 &\qquad+\kappa\iint \widetilde{a}^{(2)r}_i\widetilde{A}^{(2)k}_j \delta^{ij} \Big( \Theta^{(1)}_{,k}- \Theta^{(2)}_{,k}\Big)\Big( \Theta^{(1)}_{,r}-\Theta^{(2)}_{,r}\Big)\\
&:=\mathrm{III}_1+\mathrm{III}_2+\mathrm{III}_3,\\
|\mathrm{III}_1|+|\mathrm{III}_2|&\lesssim \int_{0}^{t} \|\widetilde{A}^{(1)}-\widetilde{A}^{(2)}\|_{L^2} \|D\Theta^{(1)}\|_{L^{\infty}} \|D\Theta^{(1)}-D\Theta^{(2)}\|_{L^2} \ dt'\\
 &\lesssim \int_{0}^{t} t'^{1/2} \Big(\int_{0}^{t'} \|\widetilde{v}^{(1)}-\widetilde{v}^{(2)}\|_{H^1}^2\ d\tau \Big)^{1/2} M_1^{1/2} \|\Theta^{(1)}-\Theta^{(2)}\|_{H^1} \ dt'\\
&\lesssim \frac{M_1}{\epsilon}t^2\int_{0}^{t} \|\widetilde{v}^{(1)}-\widetilde{v}^{(2)}\|_{H^1}^2 \ dt' +\epsilon\int_{0}^{t} \|\Theta^{(1)}-\Theta^{(2)}\|_{H^1}^2 \ dt',\\
\mathrm{III}_3&\geq \frac{1}{C}\int_{0}^{t} \|\Theta^{(1)}-\Theta^{(2)}\|_{H^1}^2-
C\|\widetilde{A}^{(2)}-\delta\|_{L^{\infty}}\|\Theta^{(1)}-\Theta^{(2)}\|_{H^1}^2 \ dt' \\
&\geq \frac{1}{C}\int_{0}^{t} \|\Theta^{(1)}-\Theta^{(2)}\|_{H^1}^2\ dt'-C t^{1/2}M_1^{1/2}\int_{0}^{t}\|\Theta^{(1)}-\Theta^{(2)}\|_{H^1}^2 \ dt'. 
\end{align*}
So when choosing $t$ small enough, one has
\begin{equation}\label{contraction2}
\begin{aligned}
&\sup\limits_{t'\in[0,t]}\|\rho_{0}^{1/2}(\Theta^{(1)}-\Theta^{(2)})\|_{L^{2}}^2+\int_{0}^{t} \|\Theta^{(1)}-\Theta^{(2)}\|_{H^1}^2\ dt'\lesssim\\
&\qquad\qquad\qquad \int_{0}^{t} \|v^{(1)}-v^{(2)}\|_{H^1}^2\ dt'+ t\Big\{\sup\limits_{t'\in[0,t]}\|\rho_{0}^{1/2}(\widetilde{\Theta}^{(1)}-\widetilde{\Theta}^{(2)})\|_{L^{2}}^2\\&\qquad\qquad\qquad\qquad+\int_{0}^{t} \|\widetilde{v}^{(1)}-\widetilde{v}^{(2)}\|_{H^1}^2\ dt'\Big\}.
\end{aligned}
\end{equation}

Therefore, combining \eqref{contraction1} and \eqref{contraction2}, we have
\begin{equation}
\label{contraction3}
\begin{aligned}
&\sup\limits_{t\in[0,T]}\Big(\|\rho_{0}^{1/2}(v^{(1)}-v^{(2)})\|_{L^{2}}^2+\|\rho_{0}^{1/2}(\Theta^{(1)}-\Theta^{(2)})\|_{L^{2}}^2\Big)\\
&\qquad\qquad+\int_{0}^{T}\|v^{(1)}-v^{(2)}\|_{H^1}^2 +\|\Theta^{(1)}-\Theta^{(2)}\|_{H^1}^2\ dt\\
&\qquad \leq T \Big\{\sup\limits_{t\in[0,T]}\Big(\|\rho_{0}^{1/2}(\widetilde{v}^{(1)}-\widetilde{v}^{(2)})\|_{L^{2}}^2+\|\rho_{0}^{1/2}(\widetilde{\Theta}^{(1)}-\widetilde{\Theta}^{(2)})\|_{L^{2}}^2\Big)\\
&\qquad\qquad+\int_{0}^{T}\|\widetilde{v}^{(1)}-\widetilde{v}^{(2)}\|_{H^1}^2 +\|\widetilde{\Theta}^{(1)}-\widetilde{\Theta}^{(2)}\|_{H^1}^2\ dt \Big\}
\end{aligned}
\end{equation}
This inequality implies that $\Xi$ is a contraction mapping in $X_T$ provided $T$ is chosen small enough only depending on the $M_0$ and $M_1$.

\begin{proof}[\textbf{Proof of Theorem \ref{mainthm}}]
Since $\Xi$ is a contraction mapping in $X_T$ if $T$ is sufficient small than a constant depending on $M_0$ (we can choose $M_1=2P(M_0)$), it follows from the Banach fixed point theorem that, there exists a unique fixed point $(v,\Theta)$ in $X_T$, such that $(v,\Theta)=\Xi(v,\Theta)$. By checking the regularity of $(v,\Theta)\in X_T$, one has that $(v,\Theta)$ is the strong solution to \eqref{1}. The uniqueness can be derived by the similar arguments obtaining the contracting inequalities \eqref{contraction3}. This finishes the proof of Theorem \ref{mainthm}.
\end{proof}

\vspace{0.8cm}
\textbf{Acknowledgement} The authors would like to express gratitude to Professor Zhouping Xin and Professor Tao Luo for their constructive suggestions on this paper.
\clearpage

\vspace{0.7cm}
\bibliographystyle{amsplain}
\bibliography{library1}

\end{document}